\documentclass[reqno]{amsart}

\usepackage{macros}

\toggletrue{cyclo}

\makeatletter
\def\enumfix{%
\if@inlabel
 \noindent \par\nobreak\vskip-\topsep\hrule\@height\z@
\fi}

\let\olditemize\itemize
\def\itemize{\enumfix\olditemize}

\makeatother

\begin{document}

\title{A naive approach to genuine $G$-spectra and cyclotomic spectra}

\author{David Ayala, Aaron Mazel-Gee, and Nick Rozenblyum}

\date{\today}

\begin{abstract}
For any compact Lie group $G$, we give a description of genuine $G$-spectra in terms of the naive equivariant spectra underlying their geometric fixedpoints.  We use this to give an analogous description of cyclotomic spectra in terms of naive $\TT$-spectra (where $\TT$ denotes the circle group), generalizing Nikolaus--Scholze's recent work in the eventually-connective case.  We also give an explicit formula for the homotopy invariants of the cyclotomic structure on a cyclotomic spectrum in these terms.
\end{abstract}

\maketitle

\setcounter{tocdepth}{2}
\tableofcontents
\setcounter{tocdepth}{2}

\setcounter{section}{-1}

\section{Introduction}
\label{section.intro}

\startcontents[sections]


\subsection{Overview}
\label{subsection.overview}

\introparagraph

\traceismysterious
This is largely due to the fact that $\TC$ is defined through \textit{genuine equivariant stable homotopy theory}, which likewise remains mysterious from an algebro-geometric point of view.  Specifically, $\TC$ is defined using the \bit{cyclotomic structure} on \textit{topological Hochschild homology} ($\THH$), which is in turn defined using \textit{genuine} $\TT$-spectra (where $\TT$ denotes the circle group), as will be recalled in \Cref{subsection.cyclotomic.spt}.

The main output of this paper is a reidentification of the $\infty$-category of {cyclotomic spectra} in terms of \textit{naive} $\TT$-spectra (\Cref{mainthm.cyclo.spt}), as well as a formula for the operation taking $\THH$ to $\TC$ in these terms (\Cref{mainthm.formula.for.TC}); these both generalize recent work of Nikolaus--Scholze \cite{NS} in the eventually-connective case (see \Cref{remark.compare.with.NS}).  Along the way, we also provide a naive reidentification of the $\infty$-category of genuine $G$-spectra for any compact Lie group $G$ (\Cref{mainthm.genuine.G.spt}), which is inspired by recent work of Glasman \cite{Saul-strat} and Mathew--Naumann--Noel \cite{MNN}.  Our work relies heavily on the \textit{generalized Tate construction}, whose functoriality in extremely broad generality we also establish (see \Cref{remark.sell.proto.tate.package}).\footnote{The generalization Tate construction is a variation on the usual Tate construction, but where one quotients by norms from \textit{all} proper subgroups rather than just from the trivial subgroup (see Remarks \ref{rmk.tate.is.genzd.tate} \and \ref{remark.genzd.tate}).}

This paper is part of a trilogy, whose overarching purpose is to provide a precise conceptual description of the cyclotomic trace at the level of derived algebraic geometry: this is explained in
\cite[\S 0]{AMR-trace}.  In \cite{AMR-fact} we construct the $\THH$ of spectrally-enriched $\infty$-categories, and in \cite{AMR-trace} we endow this construction with its cyclotomic structure in the sense defined here and construct the cyclotomic trace in these terms.  The results of the latter paper both arise from the \textit{linearization} (in the sense of Goodwillie calculus) of more primitive structures for spaces obtained in the former paper, namely an unstable cyclotomic structure on spaces-enriched $\THH$ and an unstable cyclotomic trace map to the resulting unstable version of $\TC$.

\subsection{Genuine cyclotomic spectra}
\label{subsection.cyclotomic.spt}

In their influential paper \cite{BHM}, B\"okstedt--Hsiang--Madsen defined $\TC$ by observing that the $\THH$ spectrum (of an associative ring spectrum) carries certain additional structure maps -- resulting from what is now called its \bit{cyclotomic structure} -- and taking a limit over those maps.  Following work of Hesselholt--Madsen \cite{HessMad-Witt}, this construction was placed on firmer categorical footing by Blumberg--Mandell \cite{BluMan-cyclo}: they defined a homotopy theory -- precisely, a ``spectrally-enriched model* category'' -- of (what we'll refer to as \bit{genuine}) \bit{cyclotomic spectra}, which we denote by $\Cyclo^\gen(\Spectra)$, and they showed that $\TC$ could be recovered as the (derived) hom-spectrum
\begin{equation}
\label{TC.as.hom.spectrum}
\TC
\simeq
\ulhom_{\Cyclo^\gen(\Spectra)} ( \triv(\SS) , \THH)
\end{equation}
out of the sphere spectrum equipped with its trivial cyclotomic structure.  In \cite{BG-cyclo}, Barwick--Glasman refined this to a presentable stable $\infty$-category, which we continue to denote by $\Cyclo^\gen(\Spectra)$.

These approaches to cyclotomic spectra are based in \textit{genuine equivariant homotopy theory}.  Recall that for a compact Lie group $G$, the $\infty$-category
\[ \Spectra^{\gen G} \]
of \bit{genuine $G$-spectra} is an enhancement of the $\infty$-category
\[
\Spectra^{\htpy G}
:=
\Fun(\BG,\Spectra)
\]
of \bit{homotopy $G$-spectra}.\footnote{In addition to nicely paralleling the notation $\Spectra^{\gen G}$, the notation $\Spectra^{\htpy G}$ is consistent: this is the homotopy fixedpoints of the trivial $G$-action on the $\infty$-category $\Spectra$.}\footnote{The term ``naive $G$-spectra'' is often used in the literature to refer to the stabilization of the $\infty$-category of genuine $G$-spaces (i.e.\! spectral presheaves on the orbit category of $G$).  Thus, from here onwards we use the unambiguous term ``homotopy $G$-spectra'' whenever we wish to make precise statements, though we will still at times refer colloquially e.g.\! to ``naive equivariant spectra''.  (Homotopy $G$-spectra are also sometimes called ``very naive $G$-spectra'', ``doubly naive $G$-spectra'', or ``spectra with $G$-action'').}  Roughly speaking, this is obtained by ``remembering the genuine $H$-fixedpoints'' for all closed subgroups $H \leq G$, instead of simply taking them to be the homotopy $H$-fixedpoints of the underlying homotopy $G$-spectrum.  Thus, genuine $G$-spectra are essentially a stable analog of the $\infty$-category of \textit{genuine $G$-spaces}, which keeps track of the difference e.g.\! between $\sE G$ 
and a point.\footnote{Actually, this analogy is slightly imperfect: for various reasons, in defining genuine $G$-spectra (but not genuine $G$-spaces) one also forces certain ``representation spheres'' to be invertible under the smash product monoidal structure.}  This enhancement is crucial for many applications -- for instance, it is necessary for equivariant Poincar\'e duality -- but the algebro-geometric significance of genuine $G$-spectra is poorly understood.

More generally, there is a notion of a \textit{family} of subgroups of $G$, which permits a notion of $G$-spectra which are genuine only with respect to that family.  In particular, we obtain the $\infty$-category
\[
\Spectra^{\gen^\proper G}
\]
of \bit{proper-genuine $G$-spectra} by taking our family to consist of all proper closed subgroups of $G$.  Following \cite{BG-cyclo}, we refer to the $\infty$-category
\[
\Spectra^{\gen^\proper \TT}
\]
of proper-genuine $\TT$-spectra as that of \bit{cyclonic spectra}.

Now, the world of genuine $G$-spectra admits yet a third type of fixedpoints beyond genuine fixedpoints and homotopy fixedpoints, namely the \bit{geometric $H$-fixedpoints} functor
\[
\Spectra^{\gen G}
\xra{\Phi^H}
\Spectra^{\gen \Weyl (H)}
~,
\]
where we write
\[
\Weyl(H)
:=
\Weyl_G(H)
:=
{\sf N}_G(H)/H
\]
for the Weyl group of $H$ (the quotient by it of its normalizer in $G$).\footnote{More invariantly, one can also define $\Weyl(H)$ as the compact Lie group of $G$-equivariant automorphisms of $G/H$.}  Allowing ourselves a slight abuse of notation, let us also write
\[
\Phi^{\Cyclic_r}
:
\Spectra^{\gen^\proper \TT}
\xra{\Phi^{\Cyclic_r}}
\Spectra^{\gen^\proper (\TT/\Cyclic_r)}
\xra[\sim]{(\TT/\Cyclic_r) \simeq \TT}
\Spectra^{\gen^\proper \TT}
\]
for the composite endofunctor on cyclonic spectra.  Then, an object of $\Cyclo^\gen(\Spectra)$ is given by a cyclonic spectrum $T \in \Spectra^{\gen^\proper \TT}$ equipped with a system of equivalences
\[
\Phi^{\Cyclic_r} T
\xra[\sim]{\sigma^\gen_r}
T
\]
in $\Spectra^{\gen^\proper \TT}$ for all $r \in \Nx$, which we refer to as \bit{genuine cyclotomic structure maps}.  These must be suitably compatible: for instance, for any pair $r,s \in \Nx$ we must have a commutative square
\[ \begin{tikzcd}[row sep=1.5cm, column sep=1.5cm]
\Phi^{\Cyclic_{rs}} T
\arrow{r}{\sigma^\gen_{rs}}[swap]{\sim}
\arrow{d}[sloped, anchor=north]{\sim}
&
T
\\
\Phi^{\Cyclic_s} \Phi^{\Cyclic_r} T
\arrow{r}{\sim}[swap]{\Phi^{\Cyclic_s} (\sigma^\gen_r)}
&
\Phi^{\Cyclic_s} T
\arrow{u}[sloped, anchor=south]{\sim}[swap]{\sigma^\gen_s}
\end{tikzcd} \]
of equivalences in $\Spectra^{\gen^\proper \TT}$, where the canonical equivalence on the left arises from the fact that geometric fixedpoints functors compose.\footnote{This is actually not quite the compatibility condition given in \cite[Definitions 4.7 and 4.8]{BluMan-cyclo}, which appears to be a typo (comparing e.g.\! with the discussion of \cite[\S 6]{BluMan-cyclo}).}

\subsection{A naive approach to genuine equivariant spectra}
\label{subsection.naive.genuine.spectra}

By work of Guillou--May \cite{GM-gen}, one can actually \textit{define} genuine $G$-spectra in terms of their diagrams of genuine fixedpoints (at least when $G$ is a finite group) -- see also Barwick's work \cite{Bar-Mack}.  On the other hand, there has been much recent activity aimed towards alternative presentations of genuine $G$-spectra entirely in terms of the \textit{naive} equivariant spectra extracted from their geometric fixedpoints \cite{AbKriz-MUG,Saul-strat,MNN}.

In this paper, we contribute a new perspective to this body of work.  Namely, for an arbitrary compact Lie group $G$, we identify the $\infty$-category $\Spectra^{\gen G}$ as the {right-lax limit} of a certain {left-lax diagram}, which is comprised of the $\infty$-categories $\Spectra^{\htpy \Weyl(H)}$ as $H \leq G$ varies over conjugacy classes of closed subgroups; the canonical functor from the right-lax limit is given by the composite
\[
\Spectra^{\gen G}
\xra{\Phi^H}
\Spectra^{\gen \Weyl(H)}
\xlongra{U}
\Spectra^{\htpy \Weyl(H)}
~,
\]
where $U$ denotes the forgetful functor.  This diagram is indexed by the poset $\pos_G$ of closed subgroups of $G$ ordered by subconjugacy.  Precisely, the first main result of this paper reads as follows.

\begin{maintheorem}[A naive description of genuine $G$-spectra (\Cref{cor.gen.G.spt.as.rlax.lim})]
\label{mainthm.genuine.G.spt}
For any compact Lie group $G$, there exists a canonical left-lax left $\pos_G$-module
\[
\themod^{\gen G}
\in
\LMod_{\llax.\pos_G}
~,
\]
whose value on an object $H \in \pos_G$ is the $\infty$-category $\Spectra^{\htpy \Weyl(H)}$ of homotopy $\Weyl(H)$-spectra, as well as a canonical equivalence
\[
\Spectra^{\gen G}
\simeq
\lim^\rlax
\left(
\pos_G
\overset{\llax}{\lacts}
\themod^{\gen G}
\right)
\]
between the $\infty$-category of genuine $G$-spectra and its right-lax limit.
\end{maintheorem}

\begin{remark}
\Cref{mainthm.genuine.G.spt} follows by combining two general results regarding \bit{fractured stable $\infty$-categories}, a notion we introduce here -- a variation on the ``stratified stable $\infty$-categories'' of \cite{Saul-strat} (see \Cref{rmk.fracture.neq.strat}).  To be more specific, consider a stable $\infty$-category $\cC$ and a poset $\pos$.
\begin{itemize}

\item In \Cref{thm.fracs.are.llax.modules}, we show that a fracture of $\cC$ over $\pos$ determines a certain left-lax left $\pos$-module, of which $\cC$ is the right-lax limit.

\item In \Cref{thm.dcc.and.fracturing.gives.fracture}, we give general criteria for obtaining a fracture of $\cC$ over $\pos$.

\end{itemize}
We verify that the criteria of \Cref{thm.dcc.and.fracturing.gives.fracture} are satisfied when $\cC = \Spectra^{\gen G}$ and $\pos = \pos_G$ in \Cref{prop.can.fracture.of.gen.G.spt}.
\end{remark}

\begin{remark}
As a warmup for \Cref{thm.fracs.are.llax.modules}, in \Cref{subsection.exs.of.fracs.of.gen.spt} we study the fracture decompositions guaranteed by \Cref{mainthm.genuine.G.spt} of the $\infty$-categories
\[
\Spectra^{\gen \Cyclic_p}
~,
\qquad
\Spectra^{\gen \Cyclic_{p^2}}~
,
\qquad
\textup{and}
\qquad
\Spectra^{\Cyclic_{pq}}
\]
for $p$ and $q$ distinct primes.  This can be read largely independently of (in fact, as motivation for) \Cref{section.lax.actions.and.limits}.
\end{remark}

\begin{remark}
Though the full generality of our definition is new, fractured stable $\infty$-categories have a rich history: they are a vast generalization of \textit{recollements}, which determine fractures over the poset $[1]$.  Perhaps the most prominent example of this phenomenon is the reidentification of sheaves on a scheme in terms of
\begin{itemize}
\item sheaves on a closed subcheme,
\item sheaves on its open complement, and
\item gluing data therebetween.\footnote{On the other hand, stable homotopy theorists will also be familiar with the \textit{chromatic fracture squares} that govern the reassembly of the stable homotopy category from its various chromatic localizations.}
\end{itemize}
In light of this, \Cref{mainthm.genuine.G.spt} paves the way for an \textit{algebro-geometric} interpretation of genuine $G$-spectra, namely as the sheaves of spectra on some putative stack which admits a decomposition into the stacks $\{ \sB \Weyl(H) \}_{H \leq G}$.
\end{remark}

\subsection{A naive approach to cyclotomic spectra}
\label{subsection.intro.cyclo.spt}

Since cyclotomic spectra are defined in terms of cyclonic (i.e.\! proper-genuine $\TT$-)spectra, we actually seek a slight variant of \Cref{mainthm.genuine.G.spt} which reidentifies these in naive terms.  For this, observe that
\[
\pos_\TT
\cong
(\Ndiv)^\rcone
~,
\]
where $\Ndiv$ denotes the poset of natural numbers under divisibility: the element $r \in \Ndiv$ corresponds to the subgroup $\Cyclic_r \leq \TT$, while the cone point corresponds to the full subgroup $\TT \leq \TT$.  As it turns out, passing from genuine $\TT$-spectra to proper-genuine $\TT$-spectra amounts to ignoring the cone point.

\begin{cor}[A naive description of cyclonic spectra (\Cref{cor.proper.gen.G.spt.as.rlax.lim})]
\label{mainthm.cyclonic.spt}
There exists a canonical left-lax left $\Ndiv$-module
\[
\themod^{\gen^\proper \TT}
\in
\LMod_{\llax.\Ndiv}
~,
\]
whose value on an object $r \in \Ndiv$ is the $\infty$-category $\Spectra^{\htpy (\TT/\Cyclic_r)} \simeq \Spectra^{\htpy \TT}$ of homotopy $\TT$-spectra, as well as a canonical equivalence
\[
\Spectra^{\gen^\proper \TT}
\simeq
\lim^\rlax
\left(
\Ndiv
\overset{\llax}{\lacts}
\themod^{\gen^\proper \TT}
\right)
\]
between the $\infty$-category of cyclonic spectra and its right-lax limit.
\end{cor}

\noindent Of course, the left-lax left $\Ndiv$-module $\themod^{\gen^\proper \TT}$ is obtained simply by restricting the left-lax left $\pos_\TT$-module $\themod^{\gen \TT}$ along the inclusion $\Ndiv \hookra \pos_\TT$.

\begin{remark}
To a morphism
\[
i
\longra
ir
\]
in $\Ndiv$, the left-lax left $\Ndiv$-module $\themod^{\gen^\proper \TT}$ assigns the \bit{generalized Tate construction}, a functor
\[
\Spectra^{\htpy (\TT/\Cyclic_i)}
\xra{(-)^{\tate \Cyclic_r}}
\Spectra^{\htpy (\TT/\Cyclic_{ir})}
~.
\]
By definition, this to be the composite
\[
{(-)^{\tate \Cyclic_r}}
:
\Spectra^{\htpy (\TT/\Cyclic_i)}
\xlongra{\beta}
\Spectra^{\gen (\TT/\Cyclic_i)}
\xra{\Phi^{\Cyclic_r}}
\Spectra^{\gen (\TT/\Cyclic_{ir})}
\xlongra{U}
\Spectra^{\htpy (\TT/\Cyclic_{ir})}
~,
\]
where $\beta$ denotes the right adjoint of the forgetful functor $U$.  However, we note here that this functor admits a characterization which makes no reference to genuine equivariant homotopy theory (see \Cref{rmk.genzd.tate.isnt.genuine}).
\end{remark}

\begin{remark}
The module $\themod^{\gen^\proper \TT}$ is only \textit{left-lax} because generalized Tate constructions don't strictly compose.  Rather, a commutative triangle
\[ \begin{tikzcd}[row sep=1.5cm]
&
ir
\arrow{rd}
\\
i
\arrow{ru}
\arrow{rr}
&
&
irs
\end{tikzcd} \]
in $\Ndiv$ is assigned to a \textit{lax-commutative} triangle
\begin{equation}
\label{lax.comm.triangle.witnessing.left.laxness.of.the.mod.for.cyclonic.spt}
\begin{tikzcd}[row sep=1.5cm, column sep=0cm]
&
\Spectra^{\htpy (\TT/\Cyclic_{ir})}
\arrow{rd}{(-)^{\tate \Cyclic_s}}
\\
\Spectra^{\htpy (\TT/\Cyclic_i)}
\arrow{ru}{(-)^{\tate \Cyclic_r}}
\arrow{rr}[transform canvas={yshift=0.6cm}]{\rotatebox{90}{$\Rightarrow$}}[swap]{(-)^{\tate \Cyclic_{rs}}}
&
&
\Spectra^{\htpy (\TT/\Cyclic_{irs})}
\end{tikzcd}~.
\end{equation}
Of course, these lax-commutative triangles are just the beginning of an infinite hierarchy of compatibility data that collectively witnesses the left-laxness of the module $\Spectra^{\gen^\proper \TT}$.
\end{remark}

Now, observe that the commutative monoid $\Nx$ of natural numbers under multiplication acts on the poset $\Ndiv$ by dilation: the element $r \in \Nx$ acts as the endomorphism
\[ \begin{tikzcd}[row sep=0cm]
\Ndiv
&
\Ndiv
\arrow{l}[swap]{r}
\\
\rotatebox{90}{$\in$}
&
\rotatebox{90}{$\in$}
\\
ir
&
i
\arrow[mapsto]{l}
\end{tikzcd}~. \]
This endomorphism induces a functor
\begin{equation}
\label{functor.to.be.identified.as.an.endofunctor.on.gen.proper.T.spt}
\lim^\rlax
\left(
\Ndiv
\overset{\llax}{\lacts}
\themod^{\gen^\proper \TT}
\right)
\longra
\lim^\rlax
\left(
\Ndiv
\overset{\llax}{\lacts}
r^* \themod^{\gen^\proper \TT}
\right)
\end{equation}
on right-lax limits.  The equivariance of our constructions furnishes compatible equivalences
\[
r^*\themod^{\gen^\proper \TT}
\simeq
\themod^{\gen^\proper \TT}
~,
\]
and thereafter it will follow that the functor \Cref{functor.to.be.identified.as.an.endofunctor.on.gen.proper.T.spt} corresponds through \Cref{mainthm.cyclonic.spt} with the endofunctor
\[
\Spectra^{\gen^\proper \TT}
\xra{\Phi^{\Cyclic_r}}
\Spectra^{\gen^\proper (\TT/\Cyclic_r)}
\xra[\sim]{(\TT/\Cyclic_r) \simeq \TT}
\Spectra^{\gen^\proper \TT}
~.
\]
All in all, these considerations lead to the main result of this paper.

\begin{maintheorem}[A naive description of cyclotomic spectra (\Cref{thm.in.body.cyclo.spt})]
\label{mainthm.cyclo.spt}
There exists a canonical left-lax right $\BN$-module
\[
\themod^\Cyclo
\in
\RMod_{\llax.\BN}
~,
\]
whose value on the unique object of $\BN$ is the $\infty$-category $\Spectra^{\htpy \TT}$ of homotopy $\TT$-spectra, as well as a canonical equivalence
\[
\Cyclo^\gen(\Spectra)
\simeq
\lim^\rlax
\left(
\themod^\Cyclo
\overset{\llax}{\racts}
\BN
\right)
\]
between the $\infty$-category of genuine cyclotomic spectra and its right-lax limit.
\end{maintheorem}

\begin{remark}
It is ultimately the construction of the cyclotomic structure on $\THH$ of \cite{AMR-trace} which dictates that $\Spectra^\Cyclo$ should be a left-lax \textit{right} module over $\BN$.
\end{remark}

\begin{remark}
\label{remark.naive.cyclo.structure}
In order to unwind the content of \Cref{mainthm.cyclo.spt}, let us allow ourselves a slight abuse of notation by simply writing
\[
(-)^{\tate \Cyclic_r}
:
\Spectra^{\htpy \TT}
\xra{(-)^{\tate \Cyclic_r}}
\Spectra^{\htpy (\TT/\Cyclic_r)}
\xra[\sim]{(\TT/\Cyclic_r) \simeq \TT}
\Spectra^{\htpy \TT}
\]
for the composite endofunctor, which is the value of the module $\themod^\Cyclo$ on the morphism $[1] \xra{r} \BN$.  Then, \Cref{mainthm.cyclo.spt} asserts that a genuine cyclotomic spectrum is specified by a homotopy $\TT$-spectrum
\[
T \in \Spectra^{\htpy \TT}
\]
equipped with a \bit{cyclotomic structure map}
\begin{equation}
\label{naive.cyclo.str.map}
T
\xlongra{\sigma_r}
T^{\tate \Cyclic_r}
\end{equation}
in $\Spectra^{\htpy \TT}$ for each $r \in \Nx$, along with an infinite hierarchy of compatibility data.  As a first example of such compatibility data, for any pair of elements $r,s \in \Nx$ we must be given the \textit{data} of a commutative square
\begin{equation}
\label{comm.square.in.defn.of.cyclo.spt}
\begin{tikzcd}[row sep=1.5cm, column sep=1.5cm]
T
\arrow{r}{\sigma_s}
\arrow{d}[swap]{\sigma_{rs}}
&
T^{\tate \Cyclic_s}
\arrow{d}{(\sigma_r)^{\tate \Cyclic_s}}
\\
T^{\tate \Cyclic_{rs}}
\arrow{r}
&
\left( T^{\tate \Cyclic_r} \right)^{\tate \Cyclic_s}
\end{tikzcd}
\end{equation}
in $\Spectra^{\htpy \TT}$, where the bottom map is essentially the natural transformation in the lax-commutative triangle \Cref{lax.comm.triangle.witnessing.left.laxness.of.the.mod.for.cyclonic.spt}.  More generally, for each word $W = (r_1,\ldots,r_n)$ in $\Nx$, we must be given the \textit{data} of a commutative $n$-cube in $\Spectra^{\htpy \TT}$, and these must be suitably compatible as the word $W$ varies.\footnote{The case where $n=3$ is illustrated in \cite[\Cref*{trace:tau.action.only.left.lax}]{AMR-trace}.}  Given a genuine cyclotomic spectrum
\[
\tilde{T} \in \Cyclo^\gen(\Spectra)
~,
\]
the cyclotomic structure map \Cref{naive.cyclo.str.map} on its underlying homotopy $\TT$-spectrum
\[
U\tilde{T} \in \Spectra^{\htpy \TT}
\]
is given by the composite
\[
U\tilde{T}
\xla[\sim]{U(\sigma^\gen_r)}
U\Phi^{\Cyclic_r}\tilde{T}
\longra
U \Phi^{\Cyclic_r} \beta U \tilde{T}
=:
(U\tilde{T})^{\tate \Cyclic_r}
~.
\]
\end{remark}

\begin{remark}
The sorts of words in $\Nx$ that are relevant to us are parametrized by the subdivision category $\sd(\BN)$: its objects are precisely the equivalence classes of words in $\Nx$ under the relation that any instances of the element $1 \in \Nx$ may be freely inserted or omitted. 
In fact, this is not a coincidence: subdivision plays a key role in the definition of a \textit{right}-lax limit of a \textit{left}-lax action (i.e.\! it appears because these handednesses disagree), and it will reappear in our explicit formula for $\TC$
in \Cref{mainthm.formula.for.TC} below.
\end{remark}

\subsection{The formula for $\TC$}

Justified by \Cref{mainthm.cyclo.spt}, we write
\[
\Cyclo(\Spectra)
:=
\lim^\rlax
\left(
\themod^\Cyclo
\overset{\llax}{\racts}
\BN
\right)
\]
for the $\infty$-category of \bit{cyclotomic spectra}.  There is an adjunction
\begin{equation}
\label{intro.adjn.htpy.fps.of.cyclo.str}
\begin{tikzcd}[column sep=2cm]
\Spectra
\arrow[transform canvas={yshift=0.9ex}]{r}{\triv}
\arrow[leftarrow, transform canvas={yshift=-0.9ex}]{r}[yshift=-0.2ex]{\bot}[swap]{(-)^{\htpy \Cyclo}}
&
\Cyclo(\Spectra)
\end{tikzcd}
~,
\end{equation}
whose right adjoint we refer to informally and notationally as the \textit{homotopy invariants} of the cyclotomic structure.  As this right adjoint takes $\THH$ to $\TC$, it is crucial for us to obtain an explicit description thereof: this justifies our algebro-geometric description of $\TC$ given in \cite[\S 0]{AMR-trace}.\footnote{The expression \Cref{TC.as.hom.spectrum} for $\TC$ as a hom-spectrum does not suffice for our purposes: we require a precise description of how $\TC$ is built using the data described in \Cref{remark.naive.cyclo.structure}.}
Such an explicit description is the content of our final main theorem. 

\begin{maintheorem}[The explicit formula for $\TC$ (\Cref{thm.formula.for.cyclo.fixedpts})]
\label{mainthm.formula.for.TC}
There is a canonical factorization
\[
\begin{tikzcd}[row sep=1.5cm]
\Spectra
\arrow[leftarrow]{rr}{(-)^{\htpy \Cyclo}}
&
&
\Cyclo(\Spectra)
\arrow{ld}{(-)^{\htpy \TT}}
\\
&
\Fun(\sd(\BN),\Spectra)
\arrow{lu}{\lim}
\end{tikzcd}
~,\footnote{It would be slightly more correct to write $\sd \left( \BNop \right)$ instead of $\sd(\BN)$, but these two categories are canonically equivalent (due to the commutativity of the monoid $\Nx$) and so to simplify our notation we elide the distinction here.}
\]
where the first functor takes a cyclotomic spectrum $T \in \Cyclo(\Spectra)$ to the diagram
\begin{equation}
\label{diagram.indexed.by.sd.BN}
\begin{tikzcd}[row sep=0cm, column sep=1.5cm]
&
\sd(\BN)
\arrow{r}{T^{\htpy \TT}}
&
\Spectra
\\
&
\rotatebox{90}{$\in$}
&
\rotatebox{90}{$\in$}
\\
(r_1,\ldots,r_k)
=:
\hspace{-2.2cm}
&
W
\arrow[maps to]{r}
&
\left( T^{\tate \Cyclic_W} \right)^{\htpy \TT}
&
\hspace{-1.7cm}
:= \left( T^{\tate \Cyclic_{r_1} \cdots \tate \Cyclic_{r_k}} \right)^{\htpy \TT}
\end{tikzcd}
\end{equation}
of spectra.
\end{maintheorem}

\begin{remark}
We unwind the definitions of the category $\sd(\BN)$ and the diagram \Cref{diagram.indexed.by.sd.BN} in \Cref{unwind.sd.BN}, and we explain how \Cref{mainthm.formula.for.TC} supports our algebro-geometric description of $\TC$ in \cite[\Cref*{trace:geometry.of.lim.sd.BN}]{AMR-trace}.
\end{remark}

\subsection{Miscellaneous remarks}

\begin{remark}
\label{remark.compare.with.NS}
This paper finds an important precedent in the recent work \cite{NS} of Nikolaus--Scholze.  Their main theorem provides a similar -- but far simpler -- naive reidentification of the full subcategory
\[
\Cyclo(\Spectra)_+
\subset
\Cyclo(\Spectra)
\]
of cyclotomic spectra whose underlying spectra are \textit{eventually-connective} (i.e.\! bounded below).  Let us describe some of their various remarkable results as they relate to the present paper.

\begin{enumerate}

\item\label{NS.prove.primes.dont.talk}
First of all, they prove that the generalized Tate construction is an essentially $p$-primary phenomenon: as soon as the product $r_1 \cdots r_n$ is divisible by more than one prime (for any $n \geq 1$), the functor
\[
(-)^{\tate \Cyclic_{r_1} \cdots \tate \Cyclic_{r_n}}
\]
is zero.  It follows that the data of a cyclotomic structure on any $T \in \Spectra^{\htpy \TT}$ may be specified one prime at a time, e.g.\! we only need specify the data of the commutative square \Cref{comm.square.in.defn.of.cyclo.spt} when $r = p^i$ and $s = p^j$ for some fixed prime number $p$.

\item Next, they prove a generalization of the Segal conjecture, which (after \Cref{NS.prove.primes.dont.talk}) can be read as saying that the left-lax right $\BN$-module $\themod^\Cyclo$ of \Cref{mainthm.cyclo.spt} restricts to a \textit{strict} right $\BN$-module when considering only eventually-connective spectra.  It then follows that the $p$-primary part of a cyclotomic structure on any $T \in \Spectra^{\htpy \TT}_+$ is already canonically determined by the single cyclotomic structure map
\[
T
\xra{\sigma_p}
T^{\tate \Cyclic_p}
~;
\]
for instance, this determines both the structure map
\[
T
\xra{\sigma_{p^2}}
T^{\tate \Cyclic_{p^2}}
\]
and the commutative square \Cref{comm.square.in.defn.of.cyclo.spt} in the case where $r=s=p$ through the unique lift
\[
\begin{tikzcd}[row sep=1.5cm, column sep=1.5cm]
T
\arrow{r}{\sigma_p}
\arrow[dashed]{d}[swap]{\sigma_{p^2}}
&
T^{\tate \Cyclic_p}
\arrow{d}{(\sigma_p)^{\tate \Cyclic_p}}
\\
T^{\tate \Cyclic_{p^2}}
\arrow{r}[swap]{\sim}
&
\left( T^{\tate \Cyclic_p} \right)^{\tate \Cyclic_p}
\end{tikzcd}
~.
\]

\item Finally, they prove that the homotopy fixedpoints of any $T \in \Cyclo(\Spectra)_+$ can be recovered as the equalizer
\[
T^{\htpy \Cyclo}
\xlongra{\sim}
{\sf eq}
\left(
\begin{tikzcd}[column sep=0.5cm]
T^{\htpy \TT}
\arrow[transform canvas={yshift=0.7ex}]{r}
\arrow[transform canvas={yshift=-0.7ex}]{r}
&
{\displaystyle \left( \prod_{p \textup{ prime}} \left( T^{\tate \Cyclic_p} \right)^{\htpy \TT} \right) }
\end{tikzcd}
\right)~,
\]
where in the factor corresponding to the prime $p$, one of the maps is $(\sigma_p)^{\htpy \TT}$ and the other is the canonical map
\[
T^{\htpy \TT}
\simeq
\left( T^{\htpy \Cyclic_r} \right)^{\htpy (\TT/\Cyclic_r)}
\longra
\left( T^{\tate \Cyclic_r} \right)^{\htpy (\TT/\Cyclic_r)}
\simeq
\left( T^{\tate \Cyclic_r} \right)^{\htpy \TT}
\]
(where the first equivalence is due to the fact that homotopy fixedpoints compose while the second equivalence passes through the identification $(\TT/\Cyclic_r) \simeq \TT$).  In fact, the parallel morphisms
\[ \begin{tikzcd}
T^{\htpy \TT}
\arrow[transform canvas={yshift=0.7ex}]{r}
\arrow[transform canvas={yshift=-0.7ex}]{r}
&
\left( T^{\tate \Cyclic_p} \right)^{\htpy \TT}
\end{tikzcd} \]
are precisely the image under the functor \Cref{diagram.indexed.by.sd.BN} in \Cref{mainthm.formula.for.TC} of certain parallel morphisms
\[ \begin{tikzcd}
\emptyword
\arrow[transform canvas={yshift=0.7ex}]{r}
\arrow[transform canvas={yshift=-0.7ex}]{r}
&
(p)
\end{tikzcd} \]
in $\sd(\BN)$.\footnote{From here, it is not hard to see that our formula for the functor $(-)^{\htpy \Cyclo}$ given in \Cref{mainthm.formula.for.TC} reduces to theirs in the eventually-connective case.}\footnote{On the other hand, we view the additional complexity present in our formula (as a limit over $\sd(\BN)$) as a feature: it is necessary in order to capture the full algebro-geometric significance of $\TC$ described in \cite[\S 0]{AMR-trace} (see in particular \cite[\Cref*{trace:geometry.of.lim.sd.BN}]{AMR-trace}).}

\end{enumerate}

We note that the methods employed in \cite{NS} are quite different from those employed here.  We also note that the framework of \cite{NS} only applies to $\THH(A)$ for \textit{connective} ring spectra $A$, since $\THH(A)$ may not be eventually-connective if $A$ is nonconnective.
\end{remark}

\begin{remark}
\label{remark.sell.proto.tate.package}
As mentioned in \Cref{subsection.overview}, in this work we establish the functoriality of the generalization Tate construction (for finite groups) in extremely broad generality: simultaneously
\begin{itemize}
\item in parametrized families,
\item as it receives a map from the homotopy fixedpoints,
\item for extensions of groups, and
\item interspersed with pullbacks (i.e.\! changing the base $\infty$-category that parametrizes the family).
\end{itemize}
We refer to this result (\Cref{proto.tate}) as the \bit{proto Tate package}: in addition to its use in the present paper, it is an essential input for the \textit{Tate package} (constructed in \cite{AMR-trace}), which endows $\THH$ of an arbitrary spectrally-enriched $\infty$-category with its cyclotomic structure.  We find this result to be of independent interest, and we expect that it may be useful yet elsewhere -- for instance, in establishing higher-dimensional versions of $\TC$.
\end{remark}

\subsection{Outline}

This paper is organized as follows.

\begin{itemize}

\item We begin in \Cref{section.lax.actions.and.limits} with an overview of the theory of lax actions and lax limits.  Since these notions will be very important throughout this trilogy, we go into explicit detail and provide many worked-through examples.  Our presentation of certain of these notions is new, but we show in \Cref{subsection.lax.and.adjns} that they all agree with the corresponding notions coming from $(\infty,2)$-category theory \cite{GR}.

\item Then, in \Cref{section.naive.approach.to.gen.G.spt} we study fractured stable $\infty$-categories, and we apply their general theory to deduce our naive descriptions of genuine $G$-spectra and of cyclotomic spectra (Theorems \ref{mainthm.genuine.G.spt} \and \ref{mainthm.cyclo.spt}).

\item Finally, in \Cref{section.proof.of.TC.formula} we obtain our explicit formula for $\TC$ (\Cref{mainthm.formula.for.TC}).

\end{itemize}

\subsection{Notation and conventions}

\begin{enumerate}

\item \catconventions \inftytwoconventions  

\item \functorconventions

\item \circconventions

\item \kanextnconventions

\item \spacescatsspectraconventions


\item \fibrationconventions \dualfibnsconvention

\item \efibconventions



\end{enumerate}

\subsection{Acknowledgments}

We would like to thank Thomas Nikolaus and Peter Scholze for generously sharing their ideas about cyclotomic spectra and for numerous inspiring discussions of this material.  AMG additionally thanks Saul Glasman, Denis Nardin, and Aaron Royer for a number of helpful conversations, and NR thanks Akhil Mathew for the same.

\acksupport

\stopcontents[sections]

\section{Actions and limits, strict and lax}
\label{section.lax.actions.and.limits}

In this section, for a group (or monoid) $G$, we provide definitions of strict, left-lax, and right-lax $G$-actions on $\infty$-categories, as well as strict, left-lax, and right-lax limits thereof.\footnote{Perhaps surprisingly, these various limits are actually well-defined in all cases.}  In fact, these notions only depend on the category $\BG$, and indeed it will be useful for us to give our definitions in terms of an arbitrary base $\infty$-category $\cB$; the above cases are obtained by taking $\cB = \BG$.

\begin{warning}
\label{warning.about.module.terminology}
This choice has the slightly unfortunate consequence that for instance the terminology ``$\cB$-modules'' specializes to ``$\BG$-modules'' instead of ``$G$-modules''.  Thus, in this paper (but \textit{not} in the papers \cite{AMR-fact} and \cite{AMR-trace}), to maintain uniformity we use the terminology ``$\BG$-modules'', as well as its accompanying notation.
\end{warning}

In \Cref{subsection.lax.actions}, we introduce all of our notions of $\cB$-modules and most of our notions of equivariant functors.  Then, in \Cref{subsection.lax.limits} we define most of the sorts of limit that exist.  Next, in \Cref{subsection.lax.limits.in.hard.case} we introduce the remaining sorts of limit.  This leads to our remaining notions of equivariant functors, which we introduce in \Cref{subsection.lax.equivt.functors.in.hard.case}.  Finally, in \Cref{subsection.lax.and.adjns} we assemble a pair of results regarding lax actions and lax limits which will be immensely useful throughout this trilogy.

\begin{remark}
We omit essentially all mention of lax \textit{colimits}, as we will have no explicit need for them.  On the other hand, they will certainly be present: for example, the left-lax colimit of a functor $\cB \ra \Cat$ is nothing other than the total $\infty$-category $\cE$ of the cocartesian fibration $\cE \da \cB$ that it classifies.  See \Cref{rmk.lax.colimit.isnt.a.segal.space} for a somewhat exotic variation on this idea.
\end{remark}

\subsection{Strict and lax actions}
\label{subsection.lax.actions}

We begin this subsection with an omnibus definition, which the remainder of the subsection is dedicated to discussing.

\begin{definition}
\label{define.almost.all.modules}
In \Cref{figure.define.almost.all.modules}, various $\infty$-categories of \bit{$\cB$-modules} depicted on the left side are defined as indicated on the right side.  The objects in the $\infty$-categories in the upper left diagram are (various sorts of) \bit{left} $\cB$-modules, while the objects in the $\infty$-categories in the lower left diagram are (various sorts of) \bit{right} $\cB$-modules.  In both diagrams on the left side, we refer
\begin{itemize}
\item to the objects
\begin{itemize}
\item in the middle rows as (\bit{strict}) $\cB$-modules,
\item in the top rows as \bit{left-lax} $\cB$-modules, and
\item in the bottom rows as \bit{right-lax} right $\cB$-modules,
\end{itemize}
and
\item to the morphisms
\begin{itemize}
\item in the middle columns as (\bit{strictly}) \bit{equivariant},
\item in the left columns as \bit{left-lax equivariant}, and
\item in the right columns as \bit{right-lax equivariant}.
\end{itemize}
\end{itemize}
So in our notation, laxness of the actions is indicated by a subscript (placed before ``.$\cB$''), while laxness of the morphisms is indicated by a superscript.
\begin{sidewaysfigure}
\vspace{450pt}
\[ \begin{tikzcd}[row sep=1.5cm]
\LMod^\llax_{\llax.\cB}
\arrow[\surjmonoleft]{r}
&
\LMod_{\llax.\cB}
\\
\LMod^\llax_\cB
\arrow[\surjmonoleft]{r}
\arrow[hook]{u}{\ff}
&
\LMod_\cB
\arrow[hook, two heads]{r}
\arrow[hook]{u}[swap]{\ff}
\arrow[hook]{d}[swap]{\ff}
&
\LMod^\rlax_\cB
\arrow[hook]{d}{\ff}
\\
&
\LMod_{\rlax.\cB}
\arrow[hook, two heads]{r}
&
\LMod^\rlax_{\rlax.\cB}
\end{tikzcd}
\qquad
:=
\qquad
\begin{tikzcd}[row sep=1.5cm]
\Cat_{\loc.\cocart/\cB}
\arrow[\surjmonoleft]{r}
&
\loc.\coCart_\cB
\\
\Cat_{\cocart/\cB}
\arrow[\surjmonoleft]{r}
\arrow[hook]{u}{\ff}
&
\coCart_\cB
\arrow[hook]{u}[swap]{\ff}
\\[-1.7cm]
&
&[-1.3cm]
\rotatebox{-30}{$\simeq$}
&[-1.2cm]
\\[-1.7cm]
&
&
&
\Cart_{\cB^\op}
\arrow[hook, two heads]{r}
\arrow[hook]{d}[swap]{\ff}
&
\Cat_{\cart/\cB^\op}
\arrow[hook]{d}{\ff}
\\
&
&
&
\loc.\Cart_{\cB^\op}
\arrow[hook, two heads]{r}
&
\Cat_{\loc.\cart/\cB^\op}
\end{tikzcd} \]

\vspace{50pt}

\[ \begin{tikzcd}[row sep=1.5cm]
\RMod^\llax_{\llax.\cB}
\arrow[\surjmonoleft]{r}
&
\RMod_{\llax.\cB}
\\
\RMod^\llax_\cB
\arrow[\surjmonoleft]{r}
\arrow[hook]{u}{\ff}
&
\RMod_\cB
\arrow[hook, two heads]{r}
\arrow[hook]{u}[swap]{\ff}
\arrow[hook]{d}[swap]{\ff}
&
\RMod^\rlax_\cB
\arrow[hook]{d}{\ff}
\\
&
\RMod_{\rlax.\cB}
\arrow[hook, two heads]{r}
&
\RMod^\rlax_{\rlax.\cB}
\end{tikzcd}
\qquad
:=
\qquad
\begin{tikzcd}[row sep=1.5cm]
\Cat_{\loc.\cocart/\cB^\op}
\arrow[\surjmonoleft]{r}
&
\loc.\coCart_{\cB^\op}
\\
\Cat_{\cocart/\cB^\op}
\arrow[\surjmonoleft]{r}
\arrow[hook]{u}{\ff}
&
\coCart_{\cB^\op}
\arrow[hook]{u}[swap]{\ff}
\\[-1.5cm]
&
&[-1.3cm]
\rotatebox{-30}{$\simeq$}
&[-1.2cm]
\\[-1.5cm]
&
&
&
\Cart_\cB
\arrow[hook, two heads]{r}
\arrow[hook]{d}[swap]{\ff}
&
\Cat_{\cart/\cB}
\arrow[hook]{d}{\ff}
\\
&
&
&
\loc.\Cart_\cB
\arrow[hook, two heads]{r}
&
\Cat_{\loc.\cart/\cB}
\end{tikzcd} \]

\vspace{50pt}

\caption{The commutative diagrams of monomorphisms among $\infty$-categories on the left are defined to be those on the right.}
\label{figure.define.almost.all.modules}
\end{sidewaysfigure}
\end{definition}

\begin{remark}
We'll give definitions in \Cref{subsection.lax.equivt.functors.in.hard.case} that extend the diagrams of \Cref{figure.define.almost.all.modules} to full $3 \times 3$ grids (see \Cref{obs.that.defn.of.rlax.lim.extends}).
\end{remark}

\begin{notation}
In this paper, we denote e.g.\! a right-lax left $\cB$-module
\[
(\cE \da \cB^\op)
\in
\loc.\Cart_{\cB^\op}
:=
\RMod_{\rlax.\cB^\op}
:=
\LMod_{\rlax.\cB}
\qquad
\textup{by}
\qquad
\cB
\overset{\rlax}{\lacts}
\cE
~,
\]
and analogously for all other sorts of $\cB$-modules.
\end{notation}

\begin{warning}
In concordance with \Cref{warning.about.module.terminology}, in the papers \cite{AMR-fact} and \cite{AMR-trace} we endow the action arrows $\lacts$ and $\racts$ with their more usual meaning: for instance, in those works we would write
\[
G
\overset{\rlax}{\lacts}
\cE_0
\qquad
\textup{instead of}
\qquad
\BG
\overset{\rlax}{\lacts}
\cE
~,
\]
where $\cE_0$ denotes the fiber of $\cE$ over the unique object of $\BG^\op$.
\end{warning}

\begin{example}
\label{lax.equivariance.of.strict.modules.over.walking.arrow}
Let us unwind the definitions of the $\infty$-categories
\[
\LMod_\cB
~,
\qquad
\LMod_\cB^\llax
~,
\qquad
\RMod_\cB
~,
\qquad
\textup{and}
\qquad
\RMod_\cB^\rlax
\]
in the simplest nontrivial case, namely when $\cB = [1]$.
\begin{enumerate}
\item\label{left.lax.equivariance.of.strict.modules}
Let $\cE \da [1]$ and $\cF \da [1]$ be cocartesian fibrations, the unstraightenings of functors
\[
[1]
\xra{\brax{\cE_{|0} \xra{E} \cE_{|1}}}
\Cat
\]
and
\[
[1]
\xra{\brax{\cF_{|0} \xra{F} \cF_{|1}}}
\Cat
~,
\]
respectively.  Then, let us consider a left-lax equivariant functor
\[ \begin{tikzcd}
\cE
\arrow{rr}{\varphi}
\arrow{rd}
&
&
\cF
\arrow{ld}
\\
&
{[1]}
\end{tikzcd}~. \]
Given a cocartesian morphism $e \ra E(e)$ in $\cE$ with $e \in \cE_{|0}$ and $E(e) \in \cE_{|1}$, the functor $\varphi$ takes it to some not-necessarily-cocartesian morphism $\varphi(e) \ra \varphi(E(e))$ in $\cF$ with $\varphi(e) \in \cF_{|0}$ and $\varphi(E(e)) \in \cF_{|1}$.  This admits a unique factorization
\[ \begin{tikzcd}
\varphi(e)
\arrow[dashed]{r}
\arrow{rd}
&
F(\varphi(e))
\arrow{d}
\\
&
\varphi(E(e))
\end{tikzcd} \]
as a cocartesian morphism followed by a fiber morphism.  This operation is functorial in $e \in \cE_{|0}$, which implies that our left-lax equivariant functor amounts to the data of a lax-commutative square
\[ \begin{tikzcd}
\cE_{|0}
\arrow{r}{E}[swap, transform canvas={yshift=-0.4cm}]{\rotatebox{45}{$\Rightarrow$}}
\arrow{d}[swap]{\varphi_{|0}}
&
\cE_{|1}
\arrow{d}{\varphi_{|1}}
\\
\cF_{|0}
\arrow{r}[swap]{F}
&
\cF_{|1}
\end{tikzcd}~. \]
To say that the left-lax equivariant functor is actually strictly equivariant is equivalently to say that this square actually commutes, i.e.\! that the natural transformation is a natural equivalence.

\item\label{right.lax.equivariance.of.strict.modules}
Dually, let $\cE \da [1]$ and $\cF \da [1]$ be cartesian fibrations, the unstraightenings of functors
\[
[1]^\op
\xra{\brax{\cE_{|0^\circ} \xla{E} \cE_{|1^\circ}}}
\Cat
\]
and
\[
[1]^\op
\xra{\brax{\cF_{|0^\circ} \xla{F} \cF_{|1^\circ}}}
\Cat
~,
\]
respectively.  Then, a right-lax equivariant functor
\[ \begin{tikzcd}
\cE
\arrow{rr}{\varphi}
\arrow{rd}
&
&
\cF
\arrow{ld}
\\
&
{[1]}
\end{tikzcd} \]
likewise amounts to the data of a lax-commutative square
\[ \begin{tikzcd}
\cE_{|0^\circ}
\arrow[leftarrow]{r}{E}[swap, transform canvas={yshift=-0.4cm}]{\rotatebox{-45}{$\Rightarrow$}}
\arrow{d}[swap]{\varphi_{|0^\circ}}
&
\cE_{|1^\circ}
\arrow{d}{\varphi_{|1^\circ}}
\\
\cF_{|0^\circ}
\arrow[leftarrow]{r}[swap]{F}
&
\cF_{|1^\circ}
\end{tikzcd}~. \]
To say that the right-lax equivariant functor is actually strictly equivariant is equivalently to say that this square actually commutes, i.e.\! that the natural transformation is a natural equivalence.
\end{enumerate}
\end{example}

\begin{example}
Let us unwind the definitions of the $\infty$-categories
\[
\LMod^\llax_\cB
~,
\qquad
\LMod^\rlax_\cB
~,
\qquad
\RMod^\rlax_\cB
~,
\qquad
\textup{and}
\qquad
\RMod^\llax_\cB
\]
in the simple but illustrative case that $\cB = \BG$ for a group $G$.  Choose any two objects
\[
\cE
,
\cF
\in
\Cat_{({\sf co})\cart/\BG^{(\op)}}
~,
\]
with the two choices of whether or not to include the parenthesized bits made independently.  These are classified by left or right $G$-actions on the fibers $\cE_0$ and $\cF_0$ over the basepoint of $\BG^{(\op)}$ -- right if the choices coincide, left if they do not -- and morphisms between them are left-lax equivariant in the case of ``$\cocart$'' and right-lax equivariant in the case of ``$\cart$''.  In all four cases, a morphism
\[ \begin{tikzcd}
\cE
\arrow{rr}{\varphi}
\arrow{rd}
&
&
\cF
\arrow{ld}
\\
&
\BG^{(\op)}
\end{tikzcd} \]
is the data of a functor
\[ \cE_0 \xra{\varphi_0} \cF_0 \]
on underlying $\infty$-categories equipped with certain natural transformations indexed over all $g \in G$, as recorded in \Cref{table.laxness.G.action}.
\begin{figure}[h]
\begin{tabular}{ | c | c | }
\hline
$\LMod^\llax_\BG$
&
$g \cdot \varphi_0(-) \longra \varphi_0(g \cdot -)$
\\
\hline
$\RMod^\rlax_\BG$
&
$\varphi_0(- \cdot g) \longra \varphi_0(-) \cdot g$
\\
\hline
$\LMod^\rlax_\BG$
&
$\varphi_0(g \cdot -) \longra g \cdot \varphi_0( - )$
\\
\hline
$\RMod^\llax_\BG$
&
$\varphi_0(- \cdot g) \longra \varphi_0(-) \cdot g$
\\
\hline
\end{tabular}
\caption{Given two $\infty$-categories equipped with (strict) left or right $G$-actions, defining a left- or right-lax equivariant functor between them amounts to defining a functor on underlying $\infty$-categories along with compatible lax structure maps indexed by $g \in G$, as indicated.
\label{table.laxness.G.action}}
\end{figure}
Moreover, these must be equipped with compatibility data with respect to the group structure: for example, in the case of $\LMod^\rlax_\BG$, for all $g,h \in G$ the diagram
\[ \begin{tikzcd}
\varphi_0(ghe)
\arrow{rr}
\arrow{rd}
&
&
gh\varphi_0(e)
\\
&
g\varphi_0(he)
\arrow{ru}
\end{tikzcd} \]
must commute, naturally in $e \in \cE_0$.
\end{example}

\begin{example}
\label{example.lax.actions}
Let us unwind the definitions of the $\infty$-categories
\[
\LMod^\llax_{\llax.\cB}
\qquad
\textup{and}
\qquad
\RMod^\rlax_{\rlax.\cB}
\]
in the simplest nontrivial case, namely when $\cB = [2]$.
\begin{enumerate}
\item
\label{describe.and.map.locally.cocart.fibns}
\begin{enumerate}[label=(\alph*)]
\item\label{describe.locally.cocart.fibn} Let $\cE \da [2]$ be a locally cocartesian fibration; let us write $\cE_{|i}$ for its fibers (for $i \in [2]$) and $E_{ij}$ for its cocartesian monodromy functors (for $0 \leq i < j \leq 2$).  An object $e \in \cE_{|0}$ determines a pair of composable cocartesian morphisms $e \ra E_{01}(e) \ra E_{12}(E_{01}(e))$ with $E_{01}(e) \in \cE_{|1}$ and $E_{12}(E_{01}(e)) \in \cE_{|2}$.  Their composite is a not-necessarily-cocartesian morphism, which admits a unique factorization
\[ \begin{tikzcd}
e
\arrow[dashed]{r}
\arrow{rd}
&
E_{02}(e)
\arrow{d}
\\
&
E_{12}(E_{01}(e))
\end{tikzcd} \]
as a cocartesian morphism followed by a fiber morphism.  This is functorial in $e \in \cE_{|0}$, which implies that our left-lax left $[2]$-module amounts to the data of a lax-commutative triangle
\[ \begin{tikzcd}
&
\cE_{|1}
\arrow{rd}[sloped, pos=0.2]{E_{12}}
\\
\cE_{|0}
\arrow{ru}[sloped, pos=0.8]{E_{01}}
\arrow{rr}[transform canvas={yshift=0.3cm}]{\rotatebox{90}{$\Rightarrow$}}[swap]{E_{02}}
&
&
\cE_{|2}
\end{tikzcd}~. \]
This should be thought as the unstraightening of a \textit{left-lax} functor
\[
[2]
\xra{\llax}
\Cat
\]
of $(\infty,2)$-categories.
\item Let $\cE \da [2]$ and $\cF \da [2]$ be locally cocartesian fibrations, and let us continue to use notation as in part \Cref{describe.locally.cocart.fibn} for both $\cE$ and $\cF$.  Then, a left-lax equivariant functor
\[ \begin{tikzcd}
\cE
\arrow{rr}{\varphi}
\arrow{rd}
&
&
\cF
\arrow{ld}
\\
&
{[2]}
\end{tikzcd} \]
amounts to the data of left-lax equivariant functors over the three nonidentity morphisms in $[2]$ (as described in \Cref{lax.equivariance.of.strict.modules.over.walking.arrow}\Cref{left.lax.equivariance.of.strict.modules}), along with an equivalence between the composite 2-morphisms
\[ \begin{tikzcd}[column sep=1.5cm]
&
\cE_{|1}
\arrow{rd}[sloped, pos=0.3]{E_{12}}[swap, transform canvas={xshift=0.25cm, yshift=-1.1cm}]{\rotatebox{30}{$\Rightarrow$}}
\arrow{dd}{\varphi_{|1}}
\\
\cE_{|0}
\arrow{ru}[sloped, pos=0.7]{E_{01}}[swap, transform canvas={xshift=-0.2cm, yshift=-1cm}]{\rotatebox{55}{$\Rightarrow$}}
\arrow{dd}[swap]{\varphi_{|0}}
&
&
\cE_{|2}
\arrow{dd}{\varphi_{|2}}
\\
&
\cF_{|1}
\arrow{rd}[sloped, pos=0.3]{F_{12}}
\\
\cF_{|0}
\arrow{ru}[sloped, pos=0.7]{F_{01}}
\arrow{rr}[transform canvas={yshift=0.3cm}]{\rotatebox{90}{$\Rightarrow$}}[swap]{F_{02}}
&
&
\cF_{|2}
\end{tikzcd} \]
and
\[ \begin{tikzcd}[column sep=1.5cm]
&
\cE_{|1}
\arrow{rd}[sloped, pos=0.3]{E_{12}}
\\
\cE_{|0}
\arrow{ru}[sloped, pos=0.7]{E_{01}}
\arrow{rr}[transform canvas={yshift=0.3cm}]{\rotatebox{90}{$\Rightarrow$}}[swap]{E_{02}}[swap, transform canvas={yshift=-1.1cm}]{\rotatebox{30}{$\Rightarrow$}}
\arrow{d}[swap]{\varphi_{|0}}
&
&
\cE_{|2}
\arrow{d}{\varphi_{|2}}
\\[1.25cm]
\cF_{|0}
\arrow{rr}[swap]{F_{02}}
&
&
\cF_{|2}
\end{tikzcd} \]
(i.e.\! a 3-morphism filling in the triangular prism).
\end{enumerate}
\item
\label{describe.and.map.locally.cart.fibns}
\begin{enumerate}[label=(\alph*)]
\item\label{describe.locally.cart.fibn}
Dually, let $\cE \da [2]$ be a locally cartesian fibration; let us write $\cE_{|i^\circ}$ for its fibers (for $i \in [2]$) and $E_{j^\circ i^\circ}$ for its cartesian monodromy functors (for $0 \leq i < j \leq 2$).  Then, this right-lax right $[2]$-module amounts to the data of a lax-commutative triangle
\[ \begin{tikzcd}
&
\cE_{|1^\circ}
\arrow[leftarrow]{rd}[sloped, pos=0.1]{E_{2^\circ1^\circ}}
\\
\cE_{|0^\circ}
\arrow[leftarrow]{ru}[sloped, pos=0.9]{E_{1^\circ0^\circ}}
\arrow[leftarrow]{rr}[transform canvas={yshift=0.3cm}]{\rotatebox{-90}{$\Rightarrow$}}[swap]{E_{2^\circ0^\circ}}
&
&
\cE_{|2^\circ}
\end{tikzcd}~. \]
This should be thought as the unstraightening of a \textit{right-lax} functor
\[
[2]^\op
\xra{\rlax}
\Cat
\]
of $(\infty,2)$-categories.
\item Let $\cE \da [2]$ and $\cF \da [2]$ be locally cartesian fibrations, and let us continue to use notation as in part \Cref{describe.locally.cart.fibn} for both $\cE$ and $\cF$.  Then, a right-lax equivariant functor
\[ \begin{tikzcd}
\cE
\arrow{rr}{\varphi}
\arrow{rd}
&
&
\cF
\arrow{ld}
\\
&
{[2]}
\end{tikzcd} \]
amounts to the data of right-lax equivariant functors over the three nonidentity morphisms in $[2]$ (as described in \Cref{lax.equivariance.of.strict.modules.over.walking.arrow}\Cref{right.lax.equivariance.of.strict.modules}), along with an equivalence between the composite 2-morphisms
\[ \begin{tikzcd}[column sep=1.5cm]
&
\cE_{|1^\circ}
\arrow[leftarrow]{rd}[sloped, pos=0.25]{E_{2^\circ1^\circ}}[swap, transform canvas={xshift=0.2cm, yshift=-1.2cm}]{\rotatebox{-55}{$\Rightarrow$}}
\arrow{dd}{\varphi_{|1^\circ}}
\\
\cE_{|0^\circ}
\arrow[leftarrow]{ru}[sloped, pos=0.75]{E_{1^\circ0^\circ}}[swap, transform canvas={xshift=-0.3cm, yshift=-1.1cm}]{\rotatebox{-30}{$\Rightarrow$}}
\arrow{dd}[swap]{\varphi_{|0^\circ}}
&
&
\cE_{|2^\circ}
\arrow{dd}{\varphi_{|2^\circ}}
\\
&
\cF_{|1^\circ}
\arrow[leftarrow]{rd}[sloped, pos=0.25]{F_{2^\circ1^\circ}}
\\
\cF_{|0^\circ}
\arrow[leftarrow]{ru}[sloped, pos=0.75]{F_{1^\circ0^\circ}}
\arrow[leftarrow]{rr}[transform canvas={yshift=0.3cm}]{\rotatebox{-90}{$\Rightarrow$}}[swap]{F_{2^\circ0^\circ}}
&
&
\cF_{|2^\circ}
\end{tikzcd} \]
and
\[ \begin{tikzcd}[column sep=1.5cm]
&
\cE_{|1^\circ}
\arrow[leftarrow]{rd}[sloped, pos=0.25]{E_{2^\circ1^\circ}}
\\
\cE_{|0^\circ}
\arrow[leftarrow]{ru}[sloped, pos=0.75]{E_{1^\circ0^\circ}}
\arrow[leftarrow]{rr}[transform canvas={yshift=0.3cm}]{\rotatebox{-90}{$\Rightarrow$}}[swap]{E_{2^\circ0^\circ}}[swap, transform canvas={yshift=-1.1cm}]{\rotatebox{-30}{$\Rightarrow$}}
\arrow{d}[swap]{\varphi_{|0^\circ}}
&
&
\cE_{|2^\circ}
\arrow{d}{\varphi_{|2^\circ}}
\\[1.25cm]
\cF_{|0^\circ}
\arrow[leftarrow]{rr}[swap]{F_{2^\circ0^\circ}}
&
&
\cF_{|2^\circ}
\end{tikzcd} \]
(i.e.\! a 3-morphism filling in the triangular prism).
\end{enumerate}
\end{enumerate}
\end{example}

\begin{example}
Let us unwind the definitions of the $\infty$-categories
\[
\LMod_{\llax.\cB}
~,
\qquad
\RMod_{\rlax.\cB}
~,
\qquad
\LMod_{\rlax.\cB}
~,
\qquad
\textup{and}
\qquad
\RMod_{\llax.\cB}
\]
in the simple but illustrative case that $\cB = \BG$ for a group $G$.  Choose an object
\[
\cE
\in \Cat_{\loc.({\sf co})\cart/\BG^{(\op)}}
~,
\]
with the two choices of whether or not to include the parenthesized bits made independently.  Write $\cE_0$ for the fiber over the basepoint of $\BG^{(\op)}$, the underlying $\infty$-category.  Then, this is the data of an endofunctor $(g \cdot -)$ or $(- \cdot g)$ of $\cE_0$ for each $g \in G$, along with compatible natural transformations, as recorded in \Cref{lax.group.action}.
\begin{figure}[h]
\begin{tabular}{|c|c|}
\hline
$\LMod_{\llax.\BG}$
&
$(gh \cdot -) \longra g \cdot (h \cdot -)$
\\
\hline
$\RMod_{\rlax.\BG}$
&
$(- \cdot g) \cdot h \longra (- \cdot gh)$
\\
\hline
$\LMod_{\rlax.\BG}$
&
$g \cdot (h \cdot -) \longra (gh \cdot -)$
\\
\hline
$\RMod_{\llax.\BG}$
&
$(- \cdot gh) \longra (- \cdot g) \cdot h$
\\
\hline
\end{tabular}
\caption{Equipping an $\infty$-category with a left- or right-lax left or right $G$-action amounts to defining endofunctors indexed by $g \in G$, equipped with lax structure maps corresponding to multiplication in $G$, as indicated.
\label{lax.group.action}}
\end{figure}
Of course, these must also be compatible with iterated multiplication in $G$.
\end{example}

\subsection{Strict and lax limits}
\label{subsection.lax.limits}

We begin this subsection with an omnibus definition, which the remainder of the subsection is dedicated to discussing.

\begin{definition}
\label{define.almost.all.limits}
In \Cref{figure.define.almost.all.limits}, we define various \bit{limit} functors on various $\infty$-categories of $\cB$-modules.
\begin{figure}
\vspace{-50pt}
\[
\begin{tikzcd}[row sep=1.5cm, column sep=1.5cm]
\LMod^\llax_{\llax.\cB}
\arrow[bend left]{rdd}{\lim^\llax}
\\
\LMod_{\llax.\cB}
\arrow{rd}[swap]{\lim}[sloped, transform canvas={xshift=-0.2cm, yshift=0.8cm}]{\rotatebox{90}{$\Rightarrow$}}
\arrow[hook, two heads]{u}
\\
\LMod_\cB
\arrow[hook]{u}{\ff}
&
\Cat
\\
\LMod_{\rlax.\cB}
\arrow[hookleftarrow]{u}{\ff}
\arrow{ru}{\lim}[swap, sloped, transform canvas={xshift=-0.2cm, yshift=-0.8cm}]{\rotatebox{-90}{$\Rightarrow$}}
\\
\LMod^\rlax_{\rlax.\cB}
\arrow[\surjmonoleft]{u}
\arrow[bend right]{ruu}[swap]{\lim^\rlax}
\end{tikzcd}
\qquad
:=
\qquad
\begin{tikzcd}[row sep=1.5cm, column sep=1.5cm]
\Cat_{\loc.\cocart/\cB}
\arrow[bend left]{rddd}{\Gamma}
\\
\loc.\coCart_\cB
\arrow{rdd}[swap, pos=0.6]{\Gamma^\cocart}[sloped, transform canvas={xshift=-0.2cm, yshift=1cm}]{\rotatebox{90}{$\Rightarrow$}}
\arrow[hook, two heads]{u}
\\
\coCart_\cB
\arrow[hook]{u}{\ff}
\\[-1.5cm]
\rotatebox{90}{$\simeq$}
&
\Cat
\\[-1.5cm]
\Cart_{\cB^\op}
\\
\loc.\Cart_{\cB^\op}
\arrow[hookleftarrow]{u}{\ff}
\arrow{ruu}[pos=0.55]{\Gamma^\cart}[swap, sloped, transform canvas={xshift=-0.2cm, yshift=-1cm}]{\rotatebox{-90}{$\Rightarrow$}}
\\
\Cat_{\loc.\cart/\cB^\op}
\arrow[\surjmonoleft]{u}
\arrow[bend right]{ruuu}[swap]{\Gamma}
\end{tikzcd}
\]

\vspace{20pt}

\[
\begin{tikzcd}[row sep=1.5cm, column sep=1.5cm]
\RMod^\llax_{\llax.\cB}
\arrow[bend left]{rdd}{\lim^\llax}
\\
\RMod_{\llax.\cB}
\arrow{rd}[swap]{\lim}[sloped, transform canvas={xshift=-0.2cm, yshift=0.8cm}]{\rotatebox{90}{$\Rightarrow$}}
\arrow[hook, two heads]{u}
\\
\RMod_\cB
\arrow[hook]{u}{\ff}
&
\Cat
\\
\RMod_{\rlax.\cB}
\arrow[hookleftarrow]{u}{\ff}
\arrow{ru}{\lim}[swap, sloped, transform canvas={xshift=-0.2cm, yshift=-0.8cm}]{\rotatebox{-90}{$\Rightarrow$}}
\\
\RMod^\rlax_{\rlax.\cB}
\arrow[\surjmonoleft]{u}
\arrow[bend right]{ruu}[swap]{\lim^\rlax}
\end{tikzcd}
\qquad
:=
\qquad
\begin{tikzcd}[row sep=1.5cm, column sep=1.5cm]
\Cat_{\loc.\cocart/\cB^\op}
\arrow[bend left]{rddd}{\Gamma}
\\
\loc.\coCart_{\cB^\op}
\arrow{rdd}[swap, pos=0.6]{\Gamma^\cocart}[sloped, transform canvas={xshift=-0.2cm, yshift=1cm}]{\rotatebox{90}{$\Rightarrow$}}
\arrow[hook, two heads]{u}
\\
\coCart_{\cB^\op}
\arrow[hook]{u}{\ff}
\\[-1.5cm]
\rotatebox{90}{$\simeq$}
&
\Cat
\\[-1.5cm]
\Cart_\cB
\\
\loc.\Cart_\cB
\arrow[hookleftarrow]{u}{\ff}
\arrow{ruu}[pos=0.55]{\Gamma^\cart}[swap, sloped, transform canvas={xshift=-0.2cm, yshift=-1cm}]{\rotatebox{-90}{$\Rightarrow$}}
\\
\Cat_{\loc.\cart/\cB}
\arrow[\surjmonoleft]{u}
\arrow[bend right]{ruuu}[swap]{\Gamma}
\end{tikzcd}
~.
\]

\vspace{20pt}

\caption{The rightwards functors to $\Cat$ on the left are defined to be those on the right.}
\label{figure.define.almost.all.limits}
\end{figure}
We refer
\begin{itemize}
\item to $\lim$ as the (\bit{strict}) limit,
\item to $\lim^\llax$ as the \bit{left-lax} limit, and
\item to $\lim^\rlax$ as the \bit{right-lax} limit.
\end{itemize}
\end{definition}

\begin{remark}
The strict limit of a strict (left or right) $\cB$-module as given in \Cref{define.almost.all.limits} computes the limit of its unstraightening, considered as a functor to $\Cat$.  In particular, the strict limit of a left $\cB$-module is canonically equivalent to the strict limit of its corresponding right $\cB^\op$-module, and dually: all middle triangles in the diagrams of \Cref{figure.define.almost.all.limits} commute.
\end{remark}

\begin{example}
Let us unwind the definitions of the functors in the diagrams
\[
\begin{tikzcd}[column sep=1.5cm]
\LMod_\cB
\arrow[bend left]{r}{\lim}[swap, transform canvas={yshift=-0.3cm}]{\rotatebox{-90}{$\Rightarrow$}}
\arrow[bend right]{r}[swap]{\lim^\llax}
&
\Cat
\end{tikzcd}
\qquad
\textup{and}
\qquad
\begin{tikzcd}[column sep=1.5cm]
\RMod_\cB
\arrow[bend left]{r}{\lim}[swap, transform canvas={yshift=-0.3cm}]{\rotatebox{-90}{$\Rightarrow$}}
\arrow[bend right]{r}[swap]{\lim^\rlax}
&
\Cat
\end{tikzcd}
\]
in the simple but illustrative case that $\cB = \BG$ for a group $G$.
\begin{enumerate}
\item Suppose that
\[
(\cE \da \BG)
\in
\coCart_\BG
=:
\LMod_\BG
\]
is classified by a left $G$-action on $\cE_0$.
\begin{enumerate}[label=(\alph*)]
\item An object of the strict limit is given by an object $e \in \cE_0$ equipped with equivalences
\[ g \cdot e \xlongra{\sim} e \]
for all $g \in G$ that are compatible with the group structure of $G$.
\item An object of the left-lax limit is given by an object $e \in \cE_0$ equipped with morphisms
\[ g \cdot e \longra e \]
for all $g \in G$ that are compatible with the group structure of $G$.
\end{enumerate}
\item Suppose that
\[
(\cE \da \BG)
\in
\Cart_\BG
=:
\RMod_\BG
\]
is classified by a right $G$-action on $\cE_0$.
\begin{enumerate}[label=(\alph*)]
\item An object of the strict limit is given by an object $e \in \cE_0$ equipped with equivalences
\[ e \xlongra{\sim} e \cdot g \]
for all $g \in G$ that are compatible with the group structure of $G$.
\item An object of the left-lax limit is given by an object $e \in \cE_0$ equipped with morphisms
\[ e \longra e \cdot g \]
for all $g \in G$ that are compatible with the group structure of $G$.
\end{enumerate}
\end{enumerate}
\end{example}

\begin{example}
\label{example.lax.limits.with.agreeing.lax.action}
Let us unwind the definitions of the functors
\[
\LMod^\llax_{\llax.\cB}
\xra{\lim^\llax}
\Cat
\qquad
\textup{and}
\qquad
\RMod^\rlax_{\rlax.\cB}
\xra{\lim^\rlax}
\Cat
\]
in the simplest nontrivial case, namely when $\cB = [2]$.
\begin{enumerate}

\item Let $\cE \da [2]$ be a locally cocartesian fibration, and let us employ the notation of \Cref{example.lax.actions}\Cref{describe.and.map.locally.cocart.fibns}\Cref{describe.locally.cocart.fibn}.  Then, an object of the left-lax limit of this left-lax left $[2]$-module is given by the data of
\begin{itemize}
\item objects $e_i \in \cE_{|i}$ (for $0 \leq i \leq 2$),
\item morphisms
\[
E_{ij}(e_i)
\xra{\vareps_{ij}}
e_j
\]
(for $0 \leq i < j \leq 2$), and
\item a commutative square
\[ \begin{tikzcd}[column sep=1.5cm, row sep=1.5cm]
E_{02}(e_0)
\arrow{r}{\vareps_{02}}
\arrow{d}[swap]{{\sf can}}
&
e_2
\\
E_{12}(E_{01}(e_0))
\arrow{r}[swap]{E_{12}(\vareps_{01})}
&
E_{12}(e_1)
\arrow{u}[swap]{\vareps_{12}}
\end{tikzcd} \]
in $\cE_{|2}$.
\end{itemize}
Note that the structure map $\vareps_{02}$ is \textit{canonically} determined by the structure maps $\vareps_{01}$ and $\vareps_{12}$.

\item Let $\cE \da [2]$ be a locally cartesian fibration, and let us employ the notation of \Cref{example.lax.actions}\Cref{describe.and.map.locally.cart.fibns}\Cref{describe.locally.cart.fibn}.  Then, an object of the right-lax limit of this right-lax right $[2]$-module is given by the data of
\begin{itemize}
\item objects $e_{i^\circ} \in \cE_{|i^\circ}$ (for $0 \leq i \leq 2$),
\item morphisms 
\[
e_{i^\circ}
\xra{\vareps_{j^\circ i^\circ}}
E_{j^\circ i^\circ}(e_{j^\circ})
\]
(for $0 \leq i < j \leq 2$), and
\item a commutative square
\[ \begin{tikzcd}[column sep=2cm, row sep=2cm]
e_{0^\circ}
\arrow{r}{\vareps_{2^\circ0^\circ}}
\arrow{d}[swap]{\vareps_{1^\circ0^\circ}}
&
E_{2^\circ0^\circ}(e_{2^\circ})
\\
E_{1^\circ0^\circ}(e_{1^\circ})
\arrow{r}[swap]{E_{1^\circ0^\circ}(\vareps_{2^\circ1^\circ})}
&
E_{1^\circ0^\circ}(E_{2^\circ1^\circ}(e_{2^\circ}))
\arrow{u}[swap]{\sf can}
\end{tikzcd} \]
in $\cE_{|0^\circ}$.
\end{itemize}
Note that the structure map $\vareps_{2^\circ0^\circ}$ is likewise \textit{canonically} determined by the structure maps $\vareps_{2^\circ1^\circ}$ and $\vareps_{1^\circ0^\circ}$.

\end{enumerate}
\end{example}

\subsection{Left-lax limits of right-lax modules and right-lax limits of left-lax modules}
\label{subsection.lax.limits.in.hard.case}

\begin{definition}
\label{define.sd}
The \bit{subdivision} functor is defined on the cosimplicial indexing category as the functor
\[ \begin{tikzcd}[row sep=0cm]
\bDelta
\arrow{r}{\sd}
&
\Cat
\\
\rotatebox{90}{$\in$}
&
\rotatebox{90}{$\in$}
\\
{[n]}
\arrow[mapsto]{r}
&
\power_{\not=\es}([n])
\end{tikzcd} \]
taking an object to the poset of its nonempty subsets, ordered by inclusion; we left Kan extend this definition as
\[ \begin{tikzcd}
\bDelta
\arrow{r}{\sd}
\arrow[hook]{d}
&
\Cat
\\
\Cat
\arrow[dashed]{ru}[swap]{\sd}
\end{tikzcd} \]
to obtain an endofunctor on $\Cat$, for which we use the same notation and terminology.
\end{definition}

\begin{observation}
For any $[n] \in \bDelta$, there are evident functors
\[
\sd([n])
\xra{\max}
[n]
\qquad
\textup{and}
\qquad
\sd([n])^\op
\xra{\min}
[n]
~,
\]
which respectively take a nonempty subset of $[n]$ to its maximal or minimal element.  These are respectively a locally cocartesian fibration and a locally cartesian fibration: in both cases the monodromy functors are given by union, as illustrated in \Cref{sd.of.brackets.2}.
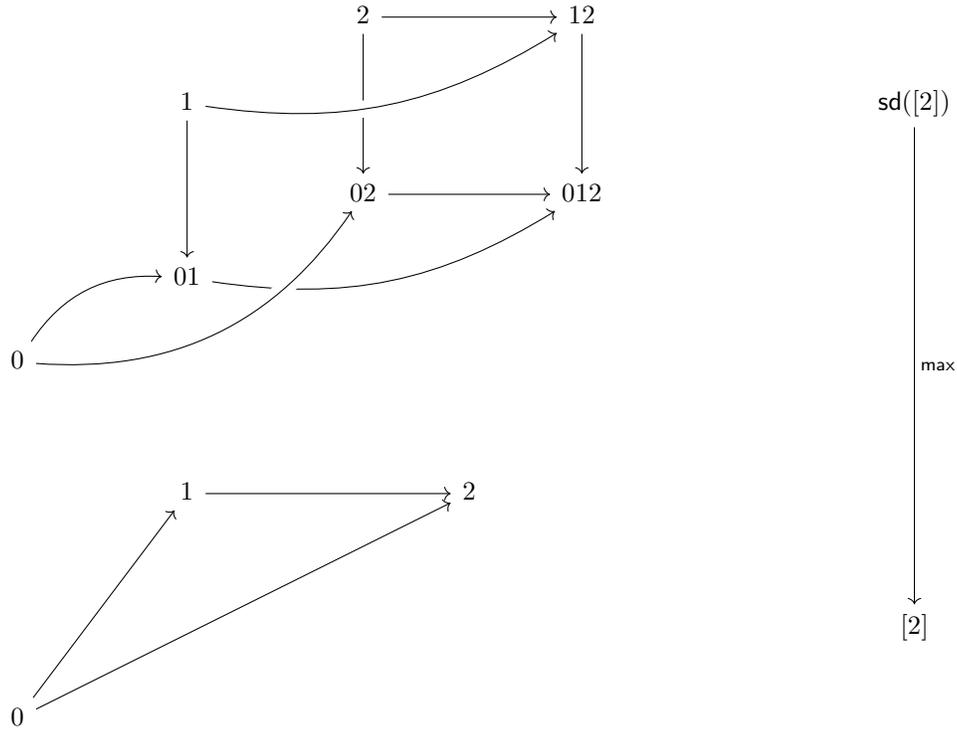
\begin{figure}[h]
\[ \begin{tikzcd}
&
&
&
&
2
\arrow{rr}
\arrow{dd}
&
&
12
\arrow{dd}
\\
&
&
1
\arrow{dd}
\arrow[bend right=20, crossing over]{rrrru}
&
&
&
&
&
&
&
&
\sd([2])
\arrow{ddddddd}{\max}
\\
&
&
&
&
02
\arrow{rr}
&
&
012
\\
&
&
01
\arrow[bend right=20]{rrrru}
\\
0
\arrow[bend right=30, crossing over]{rrrruu}
\arrow[bend left=30]{rru}
\\
&
\\
&
&
1
\arrow{rrr}
&
&
&
2
\\
&
\\
&
&
&
&
&
&
&
&
&
&
{[2]}
\\
0
\arrow{rruuu}
\arrow{rrrrruuu} 
\end{tikzcd} \]
\caption{The functor $\sd([2]) \xra{\max} [2]$ is a locally cocartesian fibration; the curved arrows denote cocartesian morphisms.  Note that the functor $\sd([1]) \xra{\max} [1]$ can also be seen here in three different ways, corresponding to the three nonidentity maps in $[2]$.}
\label{sd.of.brackets.2}
\end{figure}
By functoriality of left Kan extensions, these induce augmentations
\[
\sd
\xra{\max}
\id
\qquad
\textup{and}
\qquad
\sd^\op
\xra{\min}
\id
\]
in $\Fun(\Cat,\Cat)$, which are respectively componentwise locally cocartesian fibrations and locally cartesian fibrations.\footnote{It is not hard to see that the colimit defining $\sd(\cB)$, when translated to Segal spaces, can actually be computed in simplicial spaces; thereafter, it is straightforward to check e.g.\! that the functor $\sd(\cB) \xra{\max} \cB$ is indeed a locally cocartesian fibration.} 
\end{observation}

\begin{notation}
Given two objects
\[
(\cE_1 \da \cB)
,
(\cE_2 \da \cB)
\in
\Cat_{/\cB}~,
\]
we write
\[
\Fun^{({\sf co})\cart}_{/\cB}(\cE_1,\cE_2)
\subset
\Fun_{/\cB} ( \cE_1, \cE_2)
\]
for the full subcategory on those functors which take all (co)cartesian morphisms over $\cB$ in $\cE_1$ to (co)cartesian morphisms over $\cB$ in $\cE_2$.  This evidently defines a $\Cat$-enrichment of the subcategory
\[
\Cat^{({\sf co})\cart}_{/\cB}
\longrsurjmono
\Cat_{/\cB}
\]
on those morphisms that preserve all (co)cartesian morphisms over $\cB$.
\end{notation}

\begin{definition}
\label{define.rest.of.limits}
We define the functors
\[
\LMod_{\llax.\cB}
\xra{\lim^\rlax}
\Cat
\qquad
:=
\qquad
\loc.\coCart_\cB
\xra{\Fun^\cocart_{/\cB}(\sd(\cB),-)}
\Cat
~,
\]
\[
\RMod_{\rlax.\cB}
\xra{\lim^\llax}
\Cat
\qquad
:=
\qquad
\loc.\Cart_\cB
\xra{\Fun^\cocart_{/\cB}(\sd(\cB)^\op,-)}
\Cat
~,
\]
\[
\LMod_{\rlax.\cB}
\xra{\lim^\llax}
\Cat
\qquad
:=
\qquad
\loc.\Cart_{\cB^\op}
\xra{\Fun^\cart_{/\cB^\op}(\sd(\cB^\op)^\op,-)}
\Cat
~,
\]
and
\[
\RMod_{\llax.\cB}
\xra{\lim^\rlax}
\Cat
\qquad
:=
\qquad
\loc.\coCart_{\cB^\op}
\xra{\Fun^\cocart_{/\cB^\op}(\sd(\cB^\op),-)}
\Cat
~.
\]
We continue to refer
\begin{itemize}
\item to $\lim^\llax$ as the \bit{left-lax limit}, and
\item to $\lim^\rlax$ as the \bit{right-lax limit}.
\end{itemize}
\end{definition}

\begin{remark}
The purpose of \Cref{subsection.lax.equivt.functors.in.hard.case} is to extend the domains of the various lax limit functors of \Cref{define.rest.of.limits} along surjective monomorphisms.  We will indicate how these extensions agree with the various lax limits of \Cref{define.almost.all.limits} in \Cref{obs.that.defn.of.rlax.lim.extends}.
\end{remark}

As the following example shows, in contrast with \Cref{example.lax.limits.with.agreeing.lax.action}, when the laxness of the action does not align with the laxness of the limit, it is no longer the case that the lax structure maps for two composable morphisms in the base determine the lax structure map for their composite.

\begin{example}
\label{example.limits.of.lax.actions.when.laxness.disagrees}
Let us unwind the definitions of the functors
\[
\LMod_{\llax.\cB}
\xra{\lim^\rlax}
\Cat
\qquad
\textup{and}
\qquad
\RMod_{\rlax.\cB}
\xra{\lim^\llax}
\Cat
\]
in the simplest nontrivial case, namely when $\cB = [2]$.
\begin{enumerate}

\item Let $\cE \da [2]$ be a locally cocartesian fibration, and let us employ the notation of \Cref{example.lax.actions}\Cref{describe.and.map.locally.cocart.fibns}\Cref{describe.locally.cocart.fibn}.  Then, an object of the right-lax limit of this left-lax left $[2]$-module is given by the data of
\begin{itemize}
\item objects $e_i \in \cE_{|i}$ (for $0 \leq i \leq 2$),
\item morphisms
\[
e_j
\xra{\vareps_{ij}}
E_{ij}(e_i)
\]
(for $0 \leq i < j \leq 2$), and
\item a commutative square
\[ \begin{tikzcd}[column sep=1.5cm, row sep=1.5cm]
e_2
\arrow{r}{\varepsilon_{01}}
\arrow{d}[swap]{\varepsilon_{02}}
&
E_{12}(e_1)
\arrow{d}{E_{12}(\varepsilon_{01})}
\\
E_{02}(e_0)
\arrow{r}[swap]{\sf can}
&
E_{12}(E_{01}(e_0))
\end{tikzcd} \]
in $\cE_{|2}$.
\end{itemize}
Note that the structure map $\vareps_{02}$ is \textit{not} determined by the structure maps $\vareps_{01}$ and $\vareps_{12}$ (unless the morphism marked ``${\sf can}$'' -- the component at $e_0$ of the left-lax structure map -- is an equivalence).

\item Let $\cE \da [2]$ be a locally cartesian fibration, and let us employ the notation of \Cref{example.lax.actions}\Cref{describe.and.map.locally.cart.fibns}\Cref{describe.locally.cart.fibn}.  Then, an object of the left-lax limit of this right-lax right $[2]$-module is given by the data of
\begin{itemize}
\item objects $e_{i^\circ} \in \cE_{|i^\circ}$ (for $0 \leq i \leq 2$),
\item morphisms 
\[
E_{j^\circ i^\circ}(e_{j^\circ})
\xra{\vareps_{j^\circ i^\circ}}
e_{i^\circ}
\]
(for $0 \leq i < j \leq 2$), and
\item a commutative square
\[ \begin{tikzcd}[column sep=2cm, row sep=2cm]
E_{1^\circ0^\circ}(E_{2^\circ1^\circ}(e_{2^\circ}))
\arrow{r}{\sf can}
\arrow{d}[swap]{E_{1^\circ0^\circ}(\vareps_{2^\circ1^\circ})}
&
E_{2^\circ0^\circ}(e_{2^\circ})
\arrow{d}{\vareps_{2^\circ0^\circ}}
\\
E_{1^\circ0^\circ}(e_{1^\circ})
\arrow{r}[swap]{\vareps_{1^\circ0^\circ}}
&
e_{0^\circ}
\end{tikzcd} \]
in $\cE_{|0^\circ}$.
\end{itemize}
Note that the structure map $\vareps_{2^\circ0^\circ}$ is likewise \textit{not} determined by the structure maps $\vareps_{2^\circ1^\circ}$ and $\vareps_{1^\circ0^\circ}$ (unless the morphism marked ``${\sf can}$'' -- the component at $e_{0^\circ}$ of the right-lax structure map -- is an equivalence).

\end{enumerate}
\end{example}

\begin{remark}
It is because we are taking e.g.\! the \textit{right}-lax limit of a \textit{left}-lax module that we end up with the perhaps unfamiliar compatibility conditions of the commutative squares in \Cref{example.limits.of.lax.actions.when.laxness.disagrees}.  Comparing with \Cref{example.lax.limits.with.agreeing.lax.action}, we see that the analogous compatibility condition for e.g.\! the left-lax limit of a left-lax module as a section of a locally cocartesian fibration is simply that the section preserves composition of morphisms -- that is, that it is even a functor at all.
\end{remark}

\begin{observation}
\label{obs.rlax.lim.of.trivial.action}
Consider the projection from the product
\[
\ul{\cG}
:=
\cG \times \cB
\longra
\cB
\]
as an object of $\LMod_\cB \subset \LMod_{\llax.\cB}$, the trivial left-lax left $\cB$-module with value $\cG$.  Note that the localization of the category $\sd([n])$ with respect to the cocartesian morphisms for the locally cocartesian fibration
\[
\sd([n]) \xra{\max} [n]
\]
is $[n]^\op$; it follows that the localization of the $\infty$-category $\sd(\cB)$ with respect to the cocartesian morphisms for the locally cocartesian fibration
\[
\sd(\cB) \xra{\max} \cB
\]
is simply $\cB^\op$.  Hence, we find that
\[
\lim^\rlax \left( \cB \lacts \ul{\cG} \right)
:=
\Fun^\cocart_{/\cB} \left( \sd(\cB) , \ul{\cG} \right)
:=
\Fun^\cocart_{/\cB} \left( \sd(\cB) , \cG \times \cB \right)
\simeq
\Fun \left( \cB^\op , \cG \right)
~.
\]
Dually, considering
\[
\left( \ul{\cG} \da \cB \right) \in \RMod_\cB \subset \RMod_{\rlax.\cB}
\]
as the trivial right-lax right $\cB$-module with value $\cG$, we find that
\[
\lim^\llax \left( \ul{\cG} \racts \cB \right)
\simeq
\Fun(\cB,\cG)
~.
\]
\end{observation}

\begin{observation}
\label{obs.map.from.strict.limit.to.lax.limit}
Given a $\cB$-module of any sort, there are canonical fully faithful inclusions from its strict limit to its various lax limits: these are corepresented by the epimorphisms (in fact localizations)
\[
\sd(\cB)
\xra{\max}
\cB
\qquad
\textup{and}
\qquad
\sd(\cB)
\xra{\min}
\cB
~,
\]
which respectively preserve cocartesian and cartesian morphisms over $\cB$.
\end{observation}

\subsection{Left-lax equivariant functors of right-lax modules and right-lax equivariant functors of left-lax modules}
\label{subsection.lax.equivt.functors.in.hard.case}

\begin{observation}
\label{define.Yo.rlax}
Directly from \Cref{define.rest.of.limits}, we see that there are commutative diagrams
\[ \begin{tikzcd}[row sep=0cm]
\LMod_{\llax.\cB}
\arrow[bend left=15]{rrdd}[pos=0.4]{\lim^\rlax}
\\
\rotatebox{-90}{$:=$}
\\
\loc.\coCart_\cB
\arrow{r}[swap]{\Yo^\rlax}
&
\Fun \left( \left( \bDelta_{/\cB} \right)^\op , \Cat \right)
\arrow{r}[swap]{\lim}
&
\Cat
\\
\rotatebox{90}{$\in$}
&
\rotatebox{90}{$\in$}
\\
(\cE \da \cB)
\arrow[mapsto]{r}
&
\left(
\left( [n] \xlongra{\varphi} \cB \right)
\longmapsto
\Fun^\cocart_{/[n]} \left( \sd([n]) , \varphi^* \cE \right)
\right)
\end{tikzcd} \]
and
\[ \begin{tikzcd}[row sep=0cm]
\RMod_{\rlax.\cB}
\arrow[bend left=15]{rrdd}[pos=0.4]{\lim^\llax}
\\
\rotatebox{-90}{$:=$}
\\
\loc.\Cart_\cB
\arrow{r}[swap]{\Yo^\llax}
&
\Fun \left( \left( \bDelta_{/\cB} \right)^\op , \Cat \right)
\arrow{r}[swap]{\lim}
&
\Cat
\\
\rotatebox{90}{$\in$}
&
\rotatebox{90}{$\in$}
\\
(\cE \da \cB)
\arrow[mapsto]{r}
&
\left(
\left( [n] \xlongra{\varphi} \cB \right)
\longmapsto
\Fun^\cart_{/[n]} \left( \sd([n])^\op , \varphi^* \cE \right)
\right)
\end{tikzcd}~. \]
\end{observation}

\begin{remark}
\label{rmk.lax.Yo.fails.Segal}
Another way of phrasing the fact that the structure maps $\varepsilon_{02}$ and $\varepsilon_{2^\circ0^\circ}$ of \Cref{example.limits.of.lax.actions.when.laxness.disagrees} are not determined by the others is to say that the functors in the image of $\Yo^\rlax$ and $\Yo^\llax$ typically do \textit{not} satisfy the ($\cB$-parametrized) Segal condition.
\end{remark}

\begin{definition}
\label{define.rest.of.equivariant.functors}
We define the full images
\[ \begin{tikzcd}
\LMod_{\llax.\cB}
:=
\hspace{-1cm}
&
\loc.\coCart_\cB
\arrow[dashed, hook, two heads]{r}
\arrow{rd}[swap, pos=0.6]{\Yo^\rlax}
&
\LMod^\rlax_{\llax.\cB}
\arrow[hook]{d}{\ff}
\\
&
&
\Fun \left( \left( \bDelta_{/\cB} \right)^\op , \Cat \right)
\end{tikzcd}~, \]
\[ \begin{tikzcd}
\RMod_{\rlax.\cB}
:=
\hspace{-1cm}
&
\loc.\Cart_\cB
\arrow[dashed, hook, two heads]{r}
\arrow{rd}[swap, pos=0.6]{\Yo^\llax}
&
\RMod^\llax_{\rlax.\cB}
\arrow[hook]{d}{\ff}
\\
&
&
\Fun \left( \left( \bDelta_{/\cB} \right)^\op , \Cat \right)
\end{tikzcd}~, \]
\[ \begin{tikzcd}
\LMod_{\rlax.\cB}
:=
\hspace{-1cm}
&
\loc.\Cart_{\cB^\op}
\arrow[dashed, hook, two heads]{r}
\arrow{rd}[swap, pos=0.6]{\Yo^\llax}
&
\LMod^\llax_{\rlax.\cB}
\arrow[hook]{d}{\ff}
\\
&
&
\Fun \left( \left( \bDelta_{/\cB^\op} \right)^\op , \Cat \right)
\end{tikzcd}~, \]
and
\[ \begin{tikzcd}
\RMod_{\llax.\cB}
:=
\hspace{-1cm}
&
\loc.\coCart_{\cB^\op}
\arrow[dashed, hook, two heads]{r}
\arrow{rd}[swap, pos=0.6]{\Yo^\rlax}
&
\RMod^\rlax_{\llax.\cB}
\arrow[hook]{d}{\ff}
\\
&
&
\Fun \left( \left( \bDelta_{/\cB^\op} \right)^\op , \Cat \right)
\end{tikzcd}~.\footnote{That $\Yo^\rlax$ and $\Yo^\llax$ are indeed monomorphisms follows immediately from \Cref{obs.defns.of.rlax.map.of.llax.modules.agree}.}
\]
Our terminology for these $\infty$-categories extends that given in \Cref{define.almost.all.modules} in the evident way, so that for instance a morphism in $\RMod^\rlax_{\llax.\cB}$ is called a right-lax equivariant functor of left-lax right $\cB$-modules.
\end{definition}

\begin{observation}
\label{obs.could.use.Delta.over.B.or.Cat.over.B}
The functor
\[
\bDelta_{/\cB}
\longhookra
\Cat_{/\cB}
\]
is the inclusion of a \textit{strongly generating} subcategory: left Kan extension along it (into a cocomplete target) is fully faithful.  In particular, the right Kan extension
\[
\Fun \left( \left( \bDelta_{/\cB} \right)^\op , \Cat \right)
\longhookra
\Fun \left( \left( \Cat_{/\cB} \right)^\op , \Cat \right)
\]
is fully faithful, and so we may equivalently state \Cref{define.rest.of.equivariant.functors} in terms of this larger functor $\infty$-category.\footnote{In particular, this explains the ``$\Yo$'' notation: it's a sort of restricted Yoneda functor.}
\end{observation}

\begin{remark}
\label{rmk.lax.colimit.isnt.a.segal.space}
Given a left-lax left $\cB$-module
\[
(\cE \da \cB)
\in \loc.\coCart_\cB
=:
\LMod_{\llax.\cB}
~,
\]
let us consider its right-lax Yoneda functor
\[
\left( \bDelta_{/\cB} \right)^\op
\xra{\Yo^\rlax(\cE)}
\Cat
~.
\]
Observe that any functor $\cB' \ra \cB$ determines a composite
\[
\left( \bDelta_{/\cB'} \right)^\op
\longra
\left( \bDelta_{/\cB} \right)^\op
\xra{\Yo^\rlax(\cE)}
\Cat
~;
\]
we may think of this as the \textit{restriction} of $\Yo^\rlax(\cE)$, which allows us to study this object ``locally''.  In particular, restricting along a functor
\[
[1]
\longra
\cB
\]
determines a cartesian fibration over $[1]$: the cocartesian dual of the cocartesian fibration $\cE_{|[1]}$.  As noted in \Cref{rmk.lax.Yo.fails.Segal}, the essential subtlety of the situation lies in the fact that the corresponding cartesian monodromy functors don't generally compose: restricting along a functor
\[
[2]
\longra
\cB
\]
does \textit{not} generally determine a cartesian fibration over $[2]$.  Nevertheless, one might still consider the functor $\Yo^\rlax(\cE)$ as the ``locally cocartesian dual'' of the locally cocartesian fibration $\cE \da \cB$: the right-lax \textit{colimit} of the left-lax functor
\[
\cB
\xra{\llax}
\Cat
\]
that classifies it.
\end{remark}

\begin{observation}
\label{obs.that.defn.of.rlax.lim.extends}
It follows immediately from \Cref{obs.defns.of.rlax.map.of.llax.modules.agree} that the diagrams of \Cref{figure.define.almost.all.modules} defined in \Cref{define.almost.all.modules} extend to commutative diagrams
\[ \hspace{-1cm}
\begin{tikzcd}[row sep=1.5cm]
\LMod^\llax_{\llax.\cB}
\arrow[\surjmonoleft]{r}
&
\LMod_{\llax.\cB}
\arrow[hook, two heads]{r}
&
\LMod^\rlax_{\llax.\cB}
\\
\LMod^\llax_\cB
\arrow[\surjmonoleft]{r}
\arrow[hook]{u}{\ff}
\arrow[hook, dashed]{d}{\ff}
&
\LMod_\cB
\arrow[hook, two heads]{r}
\arrow[hook]{u}{\ff}
\arrow[hook]{d}{\ff}
&
\LMod^\rlax_\cB
\arrow[dashed, hook]{u}[swap]{\ff}
\arrow[hook]{d}{\ff}
\\
\LMod^\llax_{\rlax.\cB}
\arrow[\surjmonoleft]{r}
&
\LMod_{\rlax.\cB}
\arrow[hook, two heads]{r}
&
\LMod^\rlax_{\rlax.\cB}
\end{tikzcd}
\qquad
\textup{and}
\qquad
\begin{tikzcd}[row sep=1.5cm]
\RMod^\llax_{\llax.\cB}
\arrow[\surjmonoleft]{r}
&
\RMod_{\llax.\cB}
\arrow[hook, two heads]{r}
&
\RMod^\rlax_{\llax.\cB}
\\
\RMod^\llax_\cB
\arrow[\surjmonoleft]{r}
\arrow[hook]{u}{\ff}
\arrow[hook, dashed]{d}{\ff}
&
\RMod_\cB
\arrow[hook, two heads]{r}
\arrow[hook]{u}{\ff}
\arrow[hook]{d}{\ff}
&
\RMod^\rlax_\cB
\arrow[dashed, hook]{u}[swap]{\ff}
\arrow[hook]{d}{\ff}
\\
\RMod^\llax_{\rlax.\cB}
\arrow[\surjmonoleft]{r}
&
\RMod_{\rlax.\cB}
\arrow[hook, two heads]{r}
&
\RMod^\rlax_{\rlax.\cB}
\end{tikzcd}
\]
(with the dotted arrows being nontrivial to obtain), such that all relevant notions -- strict limit, left-lax limit, right-lax limit, and underlying $\infty$-category -- all remain unambiguous.
\end{observation}

\subsection{Lax equivariance and adjunctions}
\label{subsection.lax.and.adjns}

The primary output of this subsection is a pair of results -- Lemmas \ref{lemma.ptwise.radjt.has.ptwise.ladjt} \and \ref{lem.get.r.lax.left.adjt} -- which will be repeatedly useful throughout the trilogy to which this paper belongs.  In contrast with the rest of this section, in this final subsection we only work out the specific things we'll need, without spelling out their variously dual assertions.

The material in this subsection is drawn from the appendix of \cite{GR}.  We also use some of the notation introduced there (possibly with some minor cosmetic changes), the meaning of which should always be clear from context.

\begin{notation}
For an $\infty$-category $\cB$, we write
\[
\llax(\cB) \in 2\Cat
\]
for its \textit{left-laxification}: the $(\infty,2)$-category corepresenting left-lax functors from $\cB$.  Precisely, this is the 1-opposite of the non-unital right-laxification \cite[Chapter 11, \S A.1.2]{GR} of its (1-)opposite, which can be unitalized by \cite[Chapter 11, Lemma A.3.5]{GR}.  It follows from \cite[Chapter 11, Theorem-Construction 3.2.2]{GR} that left-lax functors out of $\cB$ are equivalent to strict functors out of $\llax(\cB)$.  In particular, we have an equivalence
\[
\LMod_{\llax.\cB}
:=
\loc.\coCart_\cB
\simeq
\Fun(\llax(\cB),\Cat)
~.
\]
\end{notation}

\begin{observation}
\label{left.laxification.is.left.kan.extended}
It is immediate from the definition of a left-lax functor that the commutative triangle
\[ \begin{tikzcd}[column sep=1.5cm]
\bDelta
\arrow{r}{\llax(-)}
\arrow[hook]{d}
&
\Cat_2
\\
\Cat
\arrow{ru}[swap, sloped, pos=0.2]{\llax(-)}
\end{tikzcd} \]
is a left Kan extension diagram: for any $\infty$-category $\cB \in \Cat$, we have an identification
\[
\llax(\cB)
\simeq
\colim_{([n] \da \cB) \in \bDelta_{/\cB}}
\llax([n])
\]
of its left-laxification.
\end{observation}

\begin{remark}
In fact, $\llax([n])$ can be described quite explicitly: the $\infty$-category
\[
\hom_{\llax([n])}(i,j)
\]
is just the poset of strictly increasing sequences
\[
i < k_1 < \cdots < k_l < j
\]
in $[n]$ from $i$ to $j$ (ordered by inclusion), with composition given by concatenation.\footnote{This can be presented as the simplicially-enriched category $\fC (\Delta^n)$ (where $\fC$ denotes the left adjoint of the homotopy-coherent nerve functor to simplicial sets), but thought of as enriched in $\infty$-categories (via quasicategories) rather than in spaces (via Kan complexes, after fibrant replacement).}
\end{remark}

\begin{observation}
\label{obs.defns.of.rlax.map.of.llax.modules.agree}
The notion of right-lax equivariant morphism between left-lax modules given in \Cref{define.rest.of.equivariant.functors} agrees with the one given in \cite[Chapter 10, \S 3]{GR}, as we now explain.

We use the terminology of \cite[Chapter 11]{GR}: a \textit{1-co/cartesian fibration} over an $(\infty,2)$-category is a co/cartesian fibration whose fibers are all $(\infty,1)$-categories (so that a 0-co/cartesian fibration is a left/right fibration).  With the evident corresponding notation, we have the composite equivalence
\[
\LMod_{\llax.\cB}
:=
\loc.\coCart_\cB
\simeq
\Fun(\llax(\cB) , \Cat)
\simeq
1\Cart_{\llax(\cB)^{1\op}}
~,
\]
whereafter the surjective monomorphism
\[
1\Cart_{\llax(\cB)^{1\op}}
\longrsurjmono
2\Cat_{1\cart/\llax(\cB)^{1\op}}
\]
witnesses the passage from strictly equivariant functors to right-lax equivariant functors in the sense of \cite[Chapter 11]{GR}.  So it suffices to provide a fully faithful inclusion
\begin{equation}
\label{want.ff.incl.so.notions.of.rlax.functors.of.llax.modules.are.the.same}
2\Cat_{1\cart/\llax(\cB)^{1\op}}
\xlonghookra{\ff}
\Fun \left( \left( \bDelta_{/\cB} \right)^\op , \Cat \right)
~,
\end{equation}
such that the composite
\begin{equation}
\label{composite.which.should.be.rlax.Yo}
\loc.\coCart_\cB
\simeq
1\Cart_{\llax(\cB)^{1\op}}
\longrsurjmono
2\Cat_{1\cart/\llax(\cB)^{1\op}}
\xlonghookra{\ff}
\Fun \left( \left( \bDelta_{/\cB} \right)^\op , \Cat \right)
\end{equation}
is the right-lax Yoneda embedding.  We take the functor \Cref{want.ff.incl.so.notions.of.rlax.functors.of.llax.modules.are.the.same} to be the restricted ($\Cat$-enriched) Yoneda functor along the functor
\begin{equation}
\label{functor.out.of.bDelta.over.B.along.which.we.take.restricted.Yo}
\begin{tikzcd}[column sep=2cm, row sep=0cm]
\bDelta_{/\cB}
\arrow{r}
&
2\Cat_{1\cart/\llax(\cB)^{1\op}}
\\
\rotatebox{90}{$\in$}
&
\rotatebox{90}{$\in$}
\\
([n] \da \cB)
\arrow[mapsto]{r}
&
\left( \llax([n])^{1\op} \da \llax(\cB)^{1\op} \right)
\end{tikzcd}
~.
\end{equation}
The functor \Cref{functor.out.of.bDelta.over.B.along.which.we.take.restricted.Yo} is the fully faithful inclusion of a strongly generating subcategory, so that the functor \Cref{want.ff.incl.so.notions.of.rlax.functors.of.llax.modules.are.the.same} is indeed a fully faithful inclusion.  To see that the resulting composite \Cref{composite.which.should.be.rlax.Yo} is indeed the right-lax Yoneda embedding, it suffices to check this when $\cB = [n]$, in which case this is readily verified: one can check directly that the locally cocartesian fibration $\sd([n]) \xra{\max} [n]$ is obtained through the composite operation of
\begin{itemize}
\item freely turning the identity functor on $\llax([n])^{1\op}$ into a cartesian fibration,
\item turning this into a cocartesian fibration over $\llax([n])$, and then
\item turning this into a locally cocartesian fibration over $[n]$.  
\end{itemize}
\end{observation}

\begin{lemma}
\label{lemma.ptwise.radjt.has.ptwise.ladjt}
The datum of a morphism
\[
\cE_0
\longla
\cE_1
\]
in $\LMod^\llax_{\llax.\cB}$ whose restriction to each $b \in \cB$ is a right adjoint is equivalent to the datum of a morphism
\[
\cE_0
\longra
\cE_1
\]
in $\LMod^\rlax_{\llax.\cB}$ whose restriction to each $b \in \cB$ is a left adjoint, with the inverse equivalences given fiberwise by passing to adjoints.
\end{lemma}

\begin{proof}
This follows from \cite[Chapter 12, Corollary 3.1.7]{GR} by taking $\SS_1 = \llax(\cB)$, $\SS_2 = [1]$, and $\TT = \Cat$.
\end{proof}

\begin{remark}
One couldn't hope to improve \Cref{lemma.ptwise.radjt.has.ptwise.ladjt} to obtain an adjunction after passing to some sort of limits, because the two functors back and forth have different sorts of lax equivariance.
\end{remark}

\begin{lemma}
\label{lem.get.r.lax.left.adjt}
A morphism
\[
\cE_0
\longla
\cE_1
\]
in $\LMod_{\llax.\cB}$ whose restriction to each $b \in \cB$ is a right adjoint admits a (necessarily unique) extension
\[ \begin{tikzcd}
{[1]}
\arrow{r}
\arrow[hook]{d}[swap]{\sf r.adjt}
&
\LMod_{\llax.\cB}
\arrow[hook, two heads]{d}
\\
\Adj
\arrow[dashed]{r}
&
\LMod^\rlax_{\llax.\cB}
\end{tikzcd} \]
to an adjunction in $\LMod^\rlax_{\llax.\cB}$.
\end{lemma}

\begin{proof}
Unwinding \cite[Chapter 12, Theorem 1.2.4]{GR} in the case that
\begin{itemize}
\item we take
$
\SS
=
[1]
$
to be the walking arrow,
\item we take the marking
$
{\bf C}
\subset
\SS
$
to be maximal (i.e.\! ${\bf C} = \SS$), and
\item we take
\[
\TT
=
\Fun^\rlax(\llax(\cB),\Cat)
\simeq
\LMod^\rlax_{\llax.\cB}
\]
to be the $(\infty,2)$-category whose objects are (strict) functors $\llax(\cB) \ra \Cat$ and whose morphisms are right-lax natural transformations,
\end{itemize}
we see that to give this extension it is equivalent to specify, functorially for all pairs of objects $[i],[j] \in \bDelta$ and all maps
\begin{equation}
\label{test.map.from.grid.to.walking.arrow}
[i] \times [j]
\longra
[1]
~,
\end{equation}
a functor
\begin{equation}
\label{provide.this.functor}
[j]^\op \circledast [i]
\longra
\Fun^\rlax(\llax(\cB),\Cat)
\end{equation}
out of the Gray product, such that in the case of the equivalence
\[
[1] \times [0] \xlongra{\sim} [1]
\]
we recover (the postcomposition of) our original map
\[
[0]^\op \circledast [1]
\simeq
[1]
\longra
\Fun(\llax(\cB),\Cat)
\longhookra
\Fun^\rlax(\llax(\cB),\Cat)
~.
\]
For this, we postcompose the map \Cref{test.map.from.grid.to.walking.arrow} to obtain a composite functor
\[
[i] \times [j]
\longra
[1]
\longra
\Fun(\llax(\cB),\Cat)
~,
\]
which is equivalent data by adjunction to a functor
\[
[i] \times [j] \times \llax(\cB)
\longra
\Cat
~,
\]
which is in turn equivalent data to a functor
\begin{equation}
\label{swap.cartesian.factors.involving.llax.B}
[i] \times \llax(\cB) \times [j]
\longra
\Cat
\end{equation}
since the cartesian product is symmetric.\footnote{This is where we use that our original morphism was in the subcategory $\LMod_{\llax.\cB} \subset \LMod^\rlax_{\llax.\cB}$: by contrast, the Gray product is not symmetric.}  Now, the assumption that our original morphism in $\LMod_{\llax.\cB}$ restricts to a right adjoint on each object $b \in \cB$ implies that the functor \Cref{swap.cartesian.factors.involving.llax.B} sends each morphism ``along $[j]$'' (i.e.\! each morphism in the source whose images in $[i]$ and $\llax(\cB)$ are equivalences) to a right adjoint.  By \cite[Chapter 12, Corollary 3.1.7]{GR}, the functor \Cref{swap.cartesian.factors.involving.llax.B} therefore induces a functor
\[
[j]^\op \circledast \left( [i] \times \llax(\cB) \right)
\longra
\Cat
~.
\]
This precomposes to a functor
\[
[j]^\op
\circledast
[i]
\circledast
\llax(\cB)
\longra
\Cat
~,
\]
which is in turn adjoint to our desired functor \Cref{provide.this.functor}.  By construction, this satisfies the stated compatibility condition.
\end{proof}

\section{A naive approach to genuine $G$-spectra}
\label{section.naive.approach.to.gen.G.spt}

In this section, we prove Theorems \ref{mainthm.genuine.G.spt} and \ref{mainthm.cyclo.spt}: our naive descriptions of genuine $G$-spectra and of cyclotomic spectra.

We prove these theorems using the formalism of \textit{fractured stable $\infty$-categories}.  This notion is introduced in \Cref{subsection.frac.cats}, where we also prove (\Cref{thm.fracs.are.llax.modules}) that a fracture $\fF$ of a stable $\infty$-category $\cC$ over a poset $\pos$ is equivalent data to a left-lax left $\pos$-module $\makemod(\fF)$, with the inverse equivalence being given by right-lax limits.  In particular, we recover the underlying stable $\infty$-category as the right-lax limit
\[
\cC
\simeq
\lim^\rlax \left( \pos \overset{\llax}{\lacts} \makemod(\fF) \right)
~.
\]

In \Cref{subsection.can.frac.of.gen.spt}, we establish a canonical fracture of the stable $\infty$-category $\Spectra^{\gen G}$ of genuine $G$-spectra over the poset $\pos_G$ of subgroups of $G$ ordered by subconjugacy.  We achieve this by first proving a general result (\Cref{thm.dcc.and.fracturing.gives.fracture}) which gives sufficient conditions under which an assignment
\[
\pos
\ni
p
\longmapsto
X_p
\in
\cC
\]
of a compact object of $\cC$ to each element of $\pos$ determines a fracture of $\cC$ over $\pos$, and then showing (\Cref{prop.can.fracture.of.gen.G.spt}) that the assignment
\[
\pos_G
\ni
H
\longmapsto
\Sigma^\infty_G(G/H)_+
\in
\Spectra^{\gen G}
\]
satisfies these conditions.  \Cref{mainthm.genuine.G.spt} follows; we record this as \Cref{cor.gen.G.spt.as.rlax.lim}.  We also establish its analog for proper-genuine $G$-spectra (stated as \Cref{mainthm.cyclonic.spt} in the special case that $G = \TT$) as \Cref{cor.proper.gen.G.spt.as.rlax.lim}.

In \Cref{subsection.naive.cyclo.spectra.for.real}, we recall the definition of genuine cyclotomic spectra, and we prove \Cref{mainthm.cyclo.spt} as \Cref{obs.Ndiv.mod.Nx.is.BN} through a straightforward analysis of the equivariance properties of \Cref{mainthm.genuine.G.spt}.

This section begins with \Cref{subsection.equivrt.conventions}, whose main purpose is to fix notation and recall some classical facts regarding genuine and naive equivariant spectra.  However, we also include some discussion of the generalized Tate construction, and we establish a generalization of a classical fact regarding its exactness properties (\Cref{Cr.genzd.tate.of.rth.power.is.exact}), which will be essential in our construction of the cyclotomic structure on $\THH$ and of the cyclotomic trace in \cite{AMR-trace}.

Before proceeding to the general formalism of fractured stable $\infty$-categories, as motivation we unpack a few sample applications of \Cref{mainthm.genuine.G.spt} in \Cref{subsection.exs.of.fracs.of.gen.spt}.

\subsection{Preliminaries on equivariant spectra}
\label{subsection.equivrt.conventions}

In this subsection, we fix our notation concerning equivariant spectra and we study the generalized Tate construction.  For further background on genuine equivariant spectra, we refer the reader to \cite{LMS,May-Alaska,ManMay-eq}

\begin{notation}
In this subsection, we write $G$ for an arbitrary compact Lie group, we write $H \leq G$ for an arbitrary closed subgroup, and we write
\[
\Weyl(H)
:=
\Weyl_G(H)
:=
{\sf N}_G(H)/H
\]
for its Weyl group (the quotient by it of its normalizer in $G$).
\end{notation}

\begin{remark}
By work of Guillou--May \cite{GM-gen}, when $G$ is a \textit{finite} group then everything presented in this subsection can be expressed in the language of spectral Mackey functors.  This will be crucial for us in \Cref{subsection.proto.tate}.
\end{remark}

\needspace{2\baselineskip}
\begin{notation}
\label{notation.various.genuine.spectra}
\begin{enumerate}
\item[]

\item We write
\[
\Spectra^{\htpy G}
:=
\Fun(\BG,\Spectra)
\]
for the $\infty$-category of \textit{homotopy $G$-spectra}.

\item We write
\[ \Spectra^{\gen G} \]
for the $\infty$-category of \textit{genuine $G$-spectra}, i.e.\! the stable $\infty$-category of spectral presheaves on the orbit $\infty$-category of $G$ with the representation spheres inverted under the smash product monoidal structure.

\item We write
\[ \Spectra^{\gen^\proper G} \]
for the $\infty$-category of \textit{proper-genuine $G$-spectra}, i.e.\! the reflective subcategory of $\Spectra^{\gen G}$ on those objects that are local with respect to the family of proper closed subgroups of $G$.

\end{enumerate}
\end{notation}

\begin{notation}
We write
\[ \begin{tikzcd}[column sep=2cm, row sep=0cm]
\Spectra^{\gen G}
\arrow[transform canvas={yshift=0.9ex}]{r}{U_G}
\arrow[hookleftarrow, transform canvas={yshift=-0.9ex}]{r}[yshift=-0.2ex]{\bot}[swap]{\beta_G}
&
\Spectra^{\htpy G}
\end{tikzcd} \]
for the adjunction -- a reflective localization -- whose left adjoint is the forgetful functor and whose right adjoint is the ``Borel-complete genuine $G$-spectrum'' functor.  The right adjoint factors through the subcategory $\Spectra^{\gen^\proper G} \subset \Spectra^{\gen G}$ of proper-genuine $G$-spectra, and we use the same notation for this factorization.  We omit the subscripts whenever the group $G$ is clear from the notation, simply writing $U \adj \beta$ instead of $U_G \adj \beta_G$.
\end{notation}

\begin{remark}
\label{remark.borel.complete}
The functor $\beta$ takes a homotopy $G$-spectrum $E \in \Spectra^{\htpy G}$ to a genuine $G$-spectrum whose genuine $H$-fixedpoints are the homotopy $H$-fixedpoints $E^{\htpy H}$.
\end{remark}

\begin{observation}
\label{obs.Ind.commutes.with.beta}
The square
\[ \begin{tikzcd}[row sep=1.5cm, column sep=1.5cm]
\Spectra^{\htpy H}
\arrow{r}{\beta_H}
\arrow{d}[swap]{\Ind_H^G}
&
\Spectra^{\gen H}
\arrow{d}{\Ind_H^G}
\\
\Spectra^{\htpy G}
\arrow{r}[swap]{\beta_G}
&
\Spectra^{\gen G}
\end{tikzcd} \]
commutes: it consists of right adjoints, and the corresponding square of left adjoints commutes.
\end{observation}

\begin{definition}
\label{define.tate}
The \bit{generalized Tate construction} with respect to $H$ is the composite
\[
(-)^{\tate H}
:
\Spectra^{\htpy G}
\xra{\beta}
\Spectra^{\gen G}
\xra{\Phi^H}
\Spectra^{\gen \Weyl(H)}
\xra{U}
\Spectra^{\htpy \Weyl(H)}
~,
\]
where $\Phi^H$ denotes the geometric $H$-fixedpoints functor.  In a slight abuse of notation, we also write
\[
\Spectra^{\htpy \TT}
\xra{(-)^{\tate \Cyclic_r}}
\Spectra^{\htpy \TT}
\]
for either long composite endofunctor in the commutative diagram
\[ \begin{tikzcd}[column sep=2cm, row sep=1.5cm]
\Spectra^{\htpy \TT}
\arrow{r}{\beta}
&
\Spectra^{\gen^\proper\TT}
\arrow{r}{\Phi^{\Cyclic_r}}
&
\Spectra^{\gen^\proper(\TT/\Cyclic_r)}
\arrow{r}{(\TT/\Cyclic_r) \simeq \TT}[swap]{\sim}
\arrow{d}[swap]{U}
&
\Spectra^{\gen^\proper\TT}
\arrow{d}{U}
\\
&
&
\Spectra^{\htpy (\TT/\Cyclic_r)}
\arrow{r}{\sim}[swap]{(\TT/\Cyclic_r) \simeq \TT}
&
\Spectra^{\htpy \TT}
\end{tikzcd}~. \]
\end{definition}

\begin{remark}
\label{rmk.tate.is.genzd.tate}
When $G = H = \Cyclic_p$ for a prime $p$, then the generalized Tate construction reduces to the ordinary Tate construction: the cofiber
\[
(-)_{\htpy \Cyclic_p}
\xra{\Nm}
(-)^{\htpy \Cyclic_p}
\longra
(-)^{{\sf t} \Cyclic_p}
\]
of the norm map.\footnote{This will be consistent with the slogan of \Cref{remark.genzd.tate}, as $\Cyclic_p$ has no proper subgroups.}
\end{remark}

\begin{warning}
In \cite{AMR-trace}, for brevity we omit the word ``generalized'' from the terminology ``generalized Tate construction''.
\end{warning}

\begin{remark}
A similar construction was studied by Greenlees--May in \cite{GreenMay-Tate}.
\end{remark}

\begin{observation}
\label{obs.genzd.tate}
There is a reflective localization
\[ \begin{tikzcd}[column sep=2cm, row sep=0cm]
\Spectra^{\gen G}
\arrow[transform canvas={yshift=0.9ex}]{r}{\locL_{\Phi^H}}
\arrow[hookleftarrow, transform canvas={yshift=-0.9ex}]{r}[yshift=-0.2ex]{\bot}
&
\Spectra^{\gen G}/\Spectra^{\gen G}_{^< H}
\end{tikzcd}~, \]
where the right side (which denotes the quotient in $\PrL$) might be more standardly identified as the localization
\[
\Spectra^{\gen G}/\Spectra^{\gen G}_{^< H}
\simeq
\Spectra^{\gen G}[\ms{F}^{-1}_{^< H}]
\]
away from the family $\ms{F}_{^< H}$ of closed subgroups of $G$ that are properly subconjugate to $H$.\footnote{Our notation here is concordant with that of \Cref{subsection.can.frac.of.gen.spt}, where we construct the canonical fracture of $\Spectra^{\gen G}$.}  The geometric $H$-fixedpoints functor can then be identified as the composite
\[
\Phi^H
:
\Spectra^{\gen G}
\xra{\locL_{\Phi^H}}
\Spectra^{\gen G}/\Spectra^{\gen G}_{^< H} 
\xlonghookra{\ff}
\Spectra^{\gen G}
\xra{(-)^H}
\Spectra^{\gen \Weyl(H)}
~, \]
which is also a reflective localization.
\end{observation}

\begin{remark}
When $H \trianglelefteq G$ is a \textit{normal} closed subgroup, \Cref{obs.genzd.tate} first appears as \cite[Corollary II.9.6]{LMS}, which moreover implies that in this case the composite functor
\[
\Spectra^{\gen G}/\Spectra^{\gen G}_{^< H} 
\xlonghookra{\ff}
\Spectra^{\gen G}
\xra{(-)^H}
\Spectra^{\gen (G/H)}
\]
is an equivalence.
\end{remark}

\begin{remark}
\label{remark.genzd.tate}
Combining \Cref{obs.genzd.tate} with \Cref{remark.borel.complete} and \Cref{obs.Ind.commutes.with.beta} justifies the heuristic description of the generalized Tate construction as
``quotienting the homotopy fixedpoints by norms from all proper subgroups''.
Explicitly, this quotient map
\begin{equation}
\label{h.to.tau.quotient.map}
(-)^{\htpy H}
\longra
(-)^{\tate H}
\end{equation}
is given by the composite
\[
U_{\Weyl(H)}
\left(
(-)^H
\longra
\Phi^H
\right)
\beta_G
~.
\]
\end{remark}

\begin{observation}
\label{obs.interesting.embedding.of.htpy.Weyl.H.spt.in.gen.G.spt}
From \Cref{obs.genzd.tate}, we obtain a composite localization
\[ \begin{tikzcd}[column sep=2cm]
\Spectra^{\gen G}
\arrow[transform canvas={yshift=0.9ex}]{r}{\Phi^H}
\arrow[hookleftarrow, transform canvas={yshift=-0.9ex}]{r}[yshift=-0.2ex]{\bot}
&
\Spectra^{\gen \Weyl(H)}
\arrow[transform canvas={yshift=0.9ex}]{r}{U}
\arrow[hookleftarrow, transform canvas={yshift=-0.9ex}]{r}[yshift=-0.2ex]{\bot}[swap]{\beta}
&
\Spectra^{\htpy \Weyl(H)}
\end{tikzcd}~. \]
In particular, we obtain a fully faithful embedding
\begin{equation}
\label{embed.htpy.Weyl.H.spectra.into.genuine.G.spectra}
\Spectra^{\htpy \Weyl(H)}
\xlonghookra{\ff}
\Spectra^{\gen G}
~.
\end{equation}
\end{observation}

\begin{remark}
The embedding \Cref{embed.htpy.Weyl.H.spectra.into.genuine.G.spectra} is a nonstandard one, which will be essential to our work.  As we will see in \Cref{subsection.frac.cats}, a fracture of a stable $\infty$-category $\cC$ over a poset $\pos$ in particular determines an assignment
\[
\pos
\ni
p
\longmapsto
\cC_p
\subset
\cC
\]
of a reflective subcategory of $\cC$ to each element of $\pos$.  The canonical fracture of $\Spectra^{\gen G}$ of \Cref{prop.can.fracture.of.gen.G.spt} will assign this reflective subcategory to the element $H \in \pos_G$.
\end{remark}

\begin{remark}
\label{rmk.genzd.tate.isnt.genuine}
The generalized Tate construction
\[
\Spectra^{\htpy G}
\xra{(-)^{\tate G}}
\Spectra
\]
admits a description making no reference to genuine equivariant homotopy theory: it is the lower composite in the left Kan extension diagram
\[ \begin{tikzcd}[row sep=1.5cm, column sep=1.5cm]
\Spectra^{\htpy G}
\arrow{r}{(-)^{\htpy G}}
\arrow{d}[transform canvas={xshift=0.7cm, yshift=0.1cm}]{\rotatebox{-30}{$\Downarrow$}}
&
\Spectra
\\
\dfrac{\Spectra^{\htpy G}}{\{ \Sigma^\infty_G ( G/H)_+ \}_{H \proper G}}
\arrow{ru}
\end{tikzcd}~, \]
where the quotient is taken in the $\infty$-category $\StCat$ of stable $\infty$-categories and exact functors.  Of course, the natural transformation is nothing but the quotient map \Cref{h.to.tau.quotient.map}.
\end{remark}

\begin{remark}
Despite \Cref{rmk.genzd.tate.isnt.genuine}, in \Cref{subsection.proto.tate} it will be convenient to use the existing framework of genuine equivariant homotopy theory obtain the \textit{proto Tate package}, which simultaneously packages a great multitude of functorialities enjoyed by the generalized Tate constructions (recall \Cref{remark.sell.proto.tate.package}).
\end{remark}

\begin{remark}
By \cite[Proposition 5.25]{MNN-ind}, the functor $(-)^{\tate \Cyclic_n}$ vanishes on $H\ZZ$-module spectra as soon as $n$ is not prime.\footnote{This is closely related to the Tate orbit and fixedpoints lemmas of \cite{NS}.}\footnote{In turn, Mathew--Naumann--Noel credit this result independently to Carlson \cite{Carlson-coh} and Balmer \cite{Balmer-extns}.}  One consequence of this fact is that the $\infty$-category $\Mod_{H\ZZ} \left( \Spectra^{\gen \TT} \right)$ of genuine $\TT$-equivariant $H\ZZ$-module spectra admits a much simpler description than the full $\infty$-category $\Spectra^{\gen \TT}$ of genuine $\TT$-spectra.
\end{remark}

We end this subsection by recording the following crucial fact about the generalized Tate construction.

\begin{lemma}
\label{Cr.genzd.tate.of.rth.power.is.exact}
The functor
\begin{equation}
\label{tate.Cr.of.rth.power}
\begin{tikzcd}[row sep=0cm]
\Spectra
\arrow{r}
&
\Spectra
\\
\rotatebox{90}{$\in$}
&
\rotatebox{90}{$\in$}
\\
E
\arrow[maps to]{r}
&
(E^{\otimes r})^{\tate \Cyclic_r}
\end{tikzcd}
\end{equation}
is exact (where $\Cyclic_r$ acts on $E^{\otimes r}$ by its action on the indexing set).
\end{lemma}

\begin{proof}
It is easy to see that the functor \Cref{tate.Cr.of.rth.power} preserves the zero object, so it suffices to show that it also preserves cofiber sequences.  So, let
\[
E_0
\longra
E_1
\xlongra{\varphi}
E_2
\]
be a cofiber sequence of spectra.  By the same argument as in the proof of \cite[Proposition 2.2.3]{LurieDAGXIII}, we obtain a filtration of the object
\[
\fib(\varphi^{\otimes r})
\in
\Spectra^{\htpy \Cyclic_r}
\]
whose associated graded object decomposes as
\[
\bigoplus_{i=1}^r
\left(
\bigoplus_{ \substack{ \{1,\ldots r\} = I \sqcup J \\ |I|= i} }
(E_0)^{\otimes I}
\otimes
(E_1)^{\otimes J}
\right)
~,
\]
where the $\Cyclic_r$-action on each parenthesized summand is through its action on the set $\{1,\ldots,r\}$.  But for all $i<r$, the $i\th$ summand is induced up from a proper subgroup of $\Cyclic_r$.  It then follows from Observations \ref{obs.Ind.commutes.with.beta} \and \ref{obs.genzd.tate} and the fact that the generalized $\Cyclic_r$-Tate construction is exact that we have an identification
\[
\fib \left( \left( \varphi^{\otimes r} \right)^{\tate \Cyclic_r} \right)
\simeq
\fib \left( \varphi^{\otimes r} \right)^{\tate \Cyclic_r}
\simeq
\left( (E_0)^{\otimes r} \right)^{\tate \Cyclic_r}
\]
(which respects the evident maps to $\left( (E_1)^{\otimes r} \right)^{\tate \Cyclic_r}$).
\end{proof}

\begin{remark}
\label{gen.tate.not.filt.cocts}
On the other hand, the functor \Cref{tate.Cr.of.rth.power} does \textit{not} preserve filtered colimits.  Rather, for any spectrum $E \in \Spectra$, there is a natural comparison morphism
\begin{equation}
\label{comparison.morphism.for.genzd.tate.vs.completion}
E \otimes \left( \SS^{\otimes r} \right)^{\tate \Cyclic_r}
\longra
\left( E^{\otimes r} \right)^{\tate \Cyclic_r}
~,
\end{equation}
which is an equivalence when $E \in \Spectra^\fin$ is finite.
\end{remark}

\begin{remark}
In the special case that $r=p$ is prime, \Cref{Cr.genzd.tate.of.rth.power.is.exact} is a classical result going back to \cite{JonesWeg} and \cite[Chapter II.3]{Hinfty}.
\end{remark}

\subsection{Examples of fracture decompositions of genuine $G$-spectra}
\label{subsection.exs.of.fracs.of.gen.spt}

In this subsection we give three sample applications of \Cref{mainthm.genuine.G.spt}, in order to motivate the abstract formalism that follows.

\begin{notation}
In this subsection, in the interest of uniformity, even in the case that $G$ is the trivial group we will often still include the forgetful functor
\[
\Spectra^{\gen G}
\xra{U}
\Spectra
\]
(now an equivalence) in our notation.
\end{notation}

\begin{example}
\label{example.genuine.Cp.spectra}
We recall the situation (which is well known) in the case that $G = \Cyclic_p$: a genuine $\Cyclic_p$-spectrum
\[
\tilde{E} \in \Spectra^{\gen \Cyclic_p}
\]
is completely specified by the data of
\begin{itemize}
\item its underlying homotopy $\Cyclic_p$-spectrum
\[
E_0
:=
U \tilde{E}
\in
\Spectra^{\htpy \Cyclic_p}
~,
\]
\item its geometric $\Cyclic_p$-fixedpoints spectrum
\[
E_1
:=
U \Phi^{\Cyclic_p} \tilde{E}
\in
\Spectra
~,
\]
and
\item the gluing data of a comparison map
\[
U \Phi^{\Cyclic_p}
\left(
\tilde{E}
\longra
\beta U \tilde{E}
\right)
=:
\left(
E_1
\longra
(E_0)^{\tate \Cyclic_p}
\right)
\]
from $E_1$ to the $\Cyclic_p$-Tate construction on $E_0$,
\end{itemize}
and conversely any such data collectively determine a genuine $\Cyclic_p$-spectrum.  Meanwhile, the poset of subgroups of $\Cyclic_p$ is simply $\pos_{\Cyclic_p} \cong [1]$, so that our left-lax left $\pos_{\Cyclic_p}$-module $\themod^{\gen \Cyclic_p}$ has no choice but to be strict: indeed, it simply classifies the $\Cyclic_p$-Tate construction
\[
\Spectra^{\htpy \Cyclic_p}
\xra{(-)^{\tate \Cyclic_p}}
\Spectra
~.
\]
Then, the above assertion can be reformulated as an identification
\[
\Spectra^{\gen \Cyclic_p}
\simeq
\lim^\rlax
\left(
\Spectra^{\htpy \Cyclic_p}
\xra{(-)^{\tate \Cyclic_p}}
\Spectra
\right)
\]
of the $\infty$-category of genuine $\Cyclic_p$-spectra as the right-lax limit of this left $\pos_{\Cyclic_p}$-module, precisely as guaranteed by \Cref{mainthm.genuine.G.spt}.
\end{example}

\begin{example}
\label{example.genuine.Cpsquared.spectra}
The situation when $G = \Cyclic_{p^2}$ is more subtle.  We would like to specify a genuine $\Cyclic_{p^2}$-spectrum
\[
\tilde{E}
\in
\Spectra^{\gen \Cyclic_{p^2}}
\]
in terms of the naive equivariant spectra
\[
E_0
:=
U \tilde{E}
\in \Spectra^{\htpy \Cyclic_{p^2}}
~,
\qquad
E_1
:=
U \Phi^{\Cyclic_p} \tilde{E}
\in \Spectra^{\htpy \Cyclic_p}
~,
\qquad
\textup{and}
\qquad
E_2
:=
U \Phi^{\Cyclic_{p^2}} \tilde{E}
\in
\Spectra
~;
\]
taking our cue from \Cref{example.genuine.Cp.spectra}, we seek gluing data that relate these naive equivariant spectra.
Again, the unit maps of various adjunctions yield
\begin{itemize}
\item a map
\begin{equation}
\label{str.map.one.to.two}
U \Phi^{\Cyclic_p}
\left(
\tilde{E}
\longra
\beta U \tilde{E}
\right)
=:
\left(
E_1
\longra
\left( U \Phi^{\Cyclic_p} \beta \right)
E_0
\right)
=:
\left(
E_1
\longra
( E_0 )^{\tate \Cyclic_p}
\right)
\end{equation}
in $\Spectra^{\htpy \Cyclic_p}$,
\item a map
\begin{equation}
\label{str.map.zero.to.two}
U \Phi^{\Cyclic_{p^2}}
\left(
\tilde{E}
\longra
\beta U \tilde{E}
\right)
=:
\left(
E_2
\longra
\left( U \Phi^{\Cyclic_{p^2}} \beta \right) E_0
\right)
=:
\left(
E_2
\longra
(E_0)^{\tate \Cyclic_{p^2}}
\right)
\end{equation}
in $\Spectra$, and
\item a map
\begin{equation}
\label{str.map.zero.to.one}
U \Phi^{\Cyclic_p}
\left(
\Phi^{\Cyclic_p} \tilde{E}
\longra
\beta U \Phi^{\Cyclic_p} \tilde{E}
\right)
\simeq
\left(
U \Phi^{\Cyclic_{p^2}} \tilde{E}
\longra
\left( U \Phi^{\Cyclic_p} \beta \right) \left( U \Phi^{\Cyclic_p} \tilde{E} \right)
\right)
=:
\left(
E_2
\longra
( E_1 )^{\tate \Cyclic_p}
\right)
\end{equation}
in $\Spectra$.
\end{itemize}
However, these maps are not unrelated: they must fit into a certain commutative square
\begin{equation}
\label{comm.square.for.gen.Cpsquared.spt}
\begin{tikzcd}[row sep=1.5cm]
E_2
\arrow{r}{\Cref{str.map.zero.to.one}}
\arrow{d}[swap]{\Cref{str.map.zero.to.two}}
&
(E_1)^{\tate \Cyclic_p}
\arrow{d}{\Cref{str.map.one.to.two}^{\tate \Cyclic_p}}
\\
(E_0)^{\tate \Cyclic_{p^2}}
\arrow{r}
&
\left( (E_0)^{\tate \Cyclic_p} \right)^{\tate \Cyclic_p}
\end{tikzcd}
\end{equation}
in $\Spectra$ as a direct consequence of the commutativity of the diagram
\[ \begin{tikzcd}[row sep=1.5cm]
\left( U \Phi^{\Cyclic_p} \right) \left( \Phi^{\Cyclic_p} \right)
\arrow{r}
\arrow{d}
&
\left( U \Phi^{\Cyclic_p} \right) \beta U \left( \Phi^{\Cyclic_p} \right)
\arrow{d}
\\
\left( U \Phi^{\Cyclic_p} \right) \left( \Phi^{\Cyclic_p} \right) U \beta
\arrow{r}
&
\left( U \Phi^{\Cyclic_p} \right) \beta U \left( \Phi^{\Cyclic_p} \right) \beta U
\end{tikzcd} \]
in $\Fun(\Spectra^{\gen \Cyclic_{p^2}},\Spectra)$ and the canonical equivalence $\Phi^{\Cyclic_{p^2}} \simeq \Phi^{\Cyclic_p} \Phi^{\Cyclic_p}$.

In this case, we have an identification $\pos_{\Cyclic_{p^2}} \cong [2]$, and our left-lax left $\pos_{\Cyclic_{p^2}}$-module $\themod^{\gen \Cyclic_{p^2}}$ classifies the lax-commutative diagram
\[ \begin{tikzcd}[row sep=1.5cm]
&
\Spectra^{\htpy \Cyclic_p}
\arrow{rd}{(-)^{\tate \Cyclic_p}}
\\
\Spectra^{\htpy \Cyclic_{p^2}}
\arrow{ru}{(-)^{\tate \Cyclic_p}}
\arrow{rr}[transform canvas={yshift=0.6cm}]{\rotatebox{90}{$\Rightarrow$}}[swap]{(-)^{\tate \Cyclic_{p^2}}}
&
&
\Spectra
\end{tikzcd}~, \]
whose natural transformation is precisely the bottom map in diagram \Cref{comm.square.for.gen.Cpsquared.spt}.  Then, \Cref{mainthm.genuine.G.spt} guarantees the reidentification
\[
\Spectra^{\gen \Cyclic_{p^2}}
\simeq
\lim^\rlax
\left( \pos_{\Cyclic_{p^2}} \overset{\llax}{\lacts} \themod^{\gen \Cyclic_{p^2}} \right)
~,
\]
which reduces to the assertion that the above data -- the naive equivariant spectra $E_0$, $E_1$, and $E_2$, the maps \Cref{str.map.one.to.two}, \Cref{str.map.zero.to.two}, and \Cref{str.map.zero.to.one}, and the commutative square \Cref{comm.square.for.gen.Cpsquared.spt} -- are indeed collectively equivalent to the datum of the genuine $\Cyclic_{p^2}$-spectrum $\tilde{E}$.
\end{example}

\begin{example}
\label{example.gen.Cpq.spt}
Finally, let us consider the case that $G = \Cyclic_{pq}$, where $p$ and $q$ are distinct primes.  First of all, the poset of subgroups of $\Cyclic_{pq} \cong [1] \times [1]$ is the commutative square
\[ \begin{tikzcd}
\Cyclic_1
\arrow{r}
\arrow{d}
&
\Cyclic_p
\arrow{d}
\\
\Cyclic_q
\arrow{r}
&
\Cyclic_{pq}
\end{tikzcd}~; \]
extrapolating from \Cref{example.genuine.Cpsquared.spectra} we see that the left-lax left $\pos_{\Cyclic_{pq}}$-module $\themod^{\gen \Cyclic_{pq}}$ classifies a lax-commutative diagram
\begin{equation}
\label{lax.comm.diagram.for.Cpq}
\begin{tikzcd}[row sep=2cm, column sep=1.5cm]
\Spectra^{\htpy \Cyclic_{pq}}
\arrow{r}{(-)^{\tate \Cyclic_p}}
\arrow{d}[swap]{(-)^{\tate \Cyclic_q}}
\arrow{rd}[sloped, pos=0.3]{(-)^{\tate \Cyclic_{pq}}}[sloped, transform canvas={xshift=0.3cm, yshift=0.55cm}]{\Uparrow}[sloped, swap, transform canvas={xshift=-0.3cm, yshift=-0.55cm}]{\Downarrow}
&
\Spectra^{\htpy \Cyclic_q}
\arrow{d}{(-)^{\tate \Cyclic_q}}
\\
\Spectra^{\htpy \Cyclic_p}
\arrow{r}[swap]{(-)^{\tate \Cyclic_p}}
&
\Spectra
\end{tikzcd}~.
\end{equation}
However, \textit{all functors from the top left to the bottom right are zero}.  It follows that the equivalence
\[
\Spectra^{\gen \Cyclic_{pq}}
\simeq
\lim^\rlax \left( \pos_{\Cyclic_{pq}} \overset{\llax}{\lacts} \themod^{\gen \Cyclic_{pq}} \right)
\]
guaranteed by \Cref{mainthm.genuine.G.spt} reduces to the assertion that a genuine $\Cyclic_{pq}$-spectrum $\tilde{E} \in \Spectra^{\gen \Cyclic_{pq}}$ is completely specified by the data of
\begin{itemize}
\item the naive equivariant spectra
\[ \begin{tikzcd}[row sep=1.5cm]
E_{00}
:=
U \tilde{E}
\in \Spectra^{\htpy \Cyclic_{pq}}
&
E_{01}
:=
U \Phi^{\Cyclic_p} \tilde{E}
\in \Spectra^{\htpy \Cyclic_q}
\\
E_{10}
:=
U \Phi^{\Cyclic_q} \tilde{E}
\in \Spectra^{\htpy \Cyclic_p}
&
E_{11} := U\Phi^{\Cyclic_{pq}} \tilde{E} \in \Spectra
\end{tikzcd} \]
and
\item the gluing data of comparison maps
\[ \begin{tikzcd}
&
&
&
\left( E_{01} \longra (E_{00})^{\tate \Cyclic_p} \right)
\\
&
\\
&
&
&
\left( E_{11} \longra (E_{01})^{\tate \Cyclic_q} \right)
\\
\left( E_{10} \longra (E_{00})^{\tate \Cyclic_q} \right)
&
\left( E_{11} \longra (E_{10})^{\tate \Cyclic_p} \right)
\end{tikzcd}~.\footnote{These data are organized so as to reflect their positions within the diagram \Cref{lax.comm.diagram.for.Cpq}.} \]
\end{itemize}
\end{example}

\begin{remark}
\Cref{example.gen.Cpq.spt} makes manifest the equivalence
\[
\Spectra^{\gen \Cyclic_{pq}} \simeq \Spectra^{\gen \Cyclic_p} \otimes \Spectra^{\gen \Cyclic_q}
\]
(where the tensor product is taken in $\PrL$).  In other words, a genuine $\Cyclic_{pq}$-spectrum is the data of a genuine $\Cyclic_p$-object in genuine $\Cyclic_q$-spectra (and vice versa).
\end{remark}

\subsection{Fractured stable $\infty$-categories}
\label{subsection.frac.cats}

In this subsection, we introduce fractured stable $\infty$-categories over a poset $\pos$, and we prove that they are equivalently specified by \textit{stable left-lax left $\pos$-modules}, with the inverse construction given by right-lax limits (\Cref{thm.fracs.are.llax.modules}).

\begin{notation}
In this subsection, we fix a stable $\infty$-category $\cC$.  We write $\Loc_\cC$ for its poset of reflective subcategories, ordered by inclusion.  We also write $\pos$ for an arbitrary poset.
\end{notation}

\begin{definition}
A nonempty full subposet $\posQ \subset \pos$ is called a \bit{convex subset} of $\pos$ if whenever there exist maps $x \ra z \ra y$ in $\pos$ with $x,y \in \posQ$, then also $z \in \posQ$.
\end{definition}

\begin{notation}
We write $\Conv_\pos$ for the poset of convex subsets of $\pos$, ordered by inclusion.  To ease notation, for an element $p \in \pos$ we simply write $p$ (rather than $\{p\}$) for the corresponding singleton element of $\Conv_\pos$.
\end{notation}

\begin{construction}
\label{constr.localization.to.gluing.category}
Suppose we are given a functor
\[
\Conv_{[1]}
\longra
\Loc_\cC
~;
\]
write
\[ \begin{tikzcd}
&
\cC_1
\\
\cC_0
\arrow[hook]{r}
&
\cC_{[1]}
\arrow[hookleftarrow]{u}
\end{tikzcd} \]
for the diagram of reflective subcategories of $\cC$ which it selects, and for any $\posQ \in \Conv_{[1]}$ write
\[ \begin{tikzcd}
\cC
\arrow[transform canvas={yshift=0.9ex}]{r}{\locL_\posQ}
\arrow[hookleftarrow, transform canvas={yshift=-0.9ex}]{r}[yshift=-0.2ex]{\bot}
&
\cC_\posQ
\end{tikzcd} \]
for the corresponding adjunction.  Consider the full subcategories
\begin{equation}
\label{get.right.adjoint.to.localization.to.gluing.category.before.limits}
\begin{tikzcd}
&
{\cC \times [1]}
\arrow{dd}
\arrow[hookleftarrow]{rd}[sloped, pos=0.35]{\ff}
\\
{\cC_{[1]} \times [1]}
\arrow[hook]{ru}[sloped, pos=0.7]{\ff}
\arrow{rd}
\arrow[hookleftarrow, crossing over, dashed]{rr}[pos=0.6]{\ff}
&
&
{\cC_0 \backslash \cC_1}
\arrow{ld}
\\
&
{[1]}
\end{tikzcd}
\end{equation}
of the product $\cC \times [1]$, where on the right the fiber over $i \in [1]$ contains those objects lying in $\cC_i$, and the factorization exists by functoriality (i.e.\! by the existence of the morphisms $i \ra [1]$ in $\Conv_{[1]}$).  Note that the two lower diagonal functors in diagram \Cref{get.right.adjoint.to.localization.to.gluing.category.before.limits} define cocartesian fibrations over $[1]$: the one on the left is classified by the identity functor on $\cC_{[1]}$, while the one on the right is classified by the composite
\[
\cC_0
\longhookra
\cC
\xra{\locL_1}
\cC_1
~.
\]
Hence, the horizontal functor in diagram \Cref{get.right.adjoint.to.localization.to.gluing.category.before.limits} defines a morphism in $\LMod^\llax_{[1]}$.  Evidently, its restriction to each object of $[1]$ is a right adjoint: the left adjoint over the object $i \in [1]$ is given by the composite
\[ \begin{tikzcd}
&
\cC
\arrow{rd}[sloped, pos=0.3]{\locL_i}
\\
\cC_{[1]}
\arrow[hook]{ru}
\arrow[dashed]{rr}[swap]{\locL_i}
&
&
\cC_i
\end{tikzcd}~. \]
Hence, by \Cref{lemma.ptwise.radjt.has.ptwise.ladjt} these left adjoints canonically assemble into a morphism
\[
\cC_{[1]} \times [1]
\longra
\cC_0 \backslash \cC_1
\]
in $\LMod^\rlax_{[1]}$, which in turn induces a composite
\begin{equation}
\label{localizn.to.gluing.cat}
\cC_{[1]}
\simeq
\lim \left( [1] \lacts \left( \cC_{[1]} \times [1] \right) \right)
\xlonghookra{\ff}
\lim^\rlax \left( [1] \lacts \left( \cC_{[1]} \times [1] \right) \right)
\longra
\lim^\rlax \left( [1] \lacts \left( \cC_0 \backslash \cC_1 \right) \right)
\end{equation}
(where the fully faithful inclusion is that of \Cref{obs.map.from.strict.limit.to.lax.limit}).
\end{construction}

\begin{remark}
An object of the target
\[ \lim^\rlax \left( [1] \lacts \left( \cC_0 \backslash \cC_1 \right) \right) \]
of the functor \Cref{localizn.to.gluing.cat} can be informally described as the data of a triple
\[
\left(
X \in \cC_0
~,~
Y \in \cC_1
~,~
\begin{tikzcd}
Y
\arrow{d}
\\
\locL_1 X
\end{tikzcd}
\right)
~.
\]
Then, the functor \Cref{localizn.to.gluing.cat} itself can be informally described by the association
\[
Z
\longmapsto
\left(
\locL_0 Z \in \cC_0
~,~
\locL_1 Z \in \cC_1
~,~
\begin{tikzcd}
\locL_1 Z
\arrow{d}
\\
\locL_1 \locL_0 Z
\end{tikzcd}
\right)
~,
\]
where the morphism (the third datum) is obtained by applying the functor $\locL_1$ to the unit $Z \ra \locL_0 Z$.
\end{remark}

\begin{definition}
\label{define.fracture}
A \bit{fracture} of a stable $\infty$-category $\cC$ over a poset $\pos$ is a functor
\[ \begin{tikzcd}[column sep=2cm, row sep=0cm]
\Conv_\pos
\arrow{r}
&
\Loc_\cC
\\
\rotatebox{90}{$\in$}
&
\rotatebox{90}{$\in$}
\\
\posQ
\arrow[mapsto]{r}
&
\cC_\posQ
\end{tikzcd} \]
that satisfies the following conditions.
\begin{enumerate}
\item\label{require.frac.preserves.terminal.objs}
It preserves terminal objects, i.e.\! $\cC_\pos = \cC$.
\item\label{require.localizn.to.gluing.cat.is.an.equivalence}
Choose any object $\posQ \in \Conv_\pos$ and any functor $\posQ \ra [1]$; this induces a functor $\Conv_{[1]} \ra \Conv_\pos$ selecting the diagram
\[ \begin{tikzcd}
&
\posQ_1
\\
\posQ_0
\arrow[hook]{r}
&
\posQ
\arrow[hookleftarrow]{u}
\end{tikzcd} ~,\]
which postcomposes to a functor
\[
\Conv_{[1]}
\longra
\Conv_\pos
\longra
\Loc_\cC
~.
\]
Then, the resulting functor
\[
\cC_\posQ
\xra{\Cref{localizn.to.gluing.cat}}
\lim^\rlax ( [1] \lacts (\cC_{\posQ_0} \backslash \cC_{\posQ_1} ) )
\]
is an equivalence.
\item\label{require.cofiltered.limit}
If a functor
\[
\cI^\rcone
\longra
\Conv_\pos
\]
witnesses $\posQ \in \Conv_\pos$ as a filtered colimit, then the canonical functor
\[
\cC_\posQ
\longra
\lim_{i \in \cI}
\cC_{\posQ_i}
\]
(to the limit over the left adjoint localization functors) is an equivalence.
\end{enumerate}
\end{definition}

\begin{remark}
\label{rmk.fracture.neq.strat}
Condition \Cref{require.localizn.to.gluing.cat.is.an.equivalence} of \Cref{define.fracture} implies the second condition in \cite[Definition 3.4]{Saul-strat}, but they are not equivalent.\footnote{The simplest example of this discrepancy occurs when $\pos = [1]$ and the functor $\Conv_{[1]} \ra \Loc_\cC$ is constant at the maximal reflective subcategory of $\cC$, namely $\cC$ itself: this defines a stratification, but not a fracture.}
\end{remark}

\begin{observation}
\label{obs.fractures.restrict}
By passing to a reflective subcategory of $\cC$, it is always possible to assume that condition \Cref{require.frac.preserves.terminal.objs} of \Cref{define.fracture} is satisfied.  Indeed, a fracture of $\cC$ over $\pos$ in particular yields a fracture of $\cC_\posQ$ over $\posQ$ for all $\posQ \in \Conv_\pos$.
\end{observation}

Loosely speaking, condition \Cref{require.localizn.to.gluing.cat.is.an.equivalence} of \Cref{define.fracture} asserts that as the object $\posQ \in \Conv_\pos$ grows (i.e.\! acquires more elements of $\pos$), the corresponding reflective subcategory $\cC_\posQ$ must be obtained by ``gluing together'' smaller reflective subcategories, in a suitable sense.  On the other hand, condition \Cref{require.cofiltered.limit} of \Cref{define.fracture} requires a certain regularity as the object $\posQ \in \Conv_\pos$ grows transfinitely.  These observations suggests that $\cC$ itself should be obtainable from the ``one-point'' localizations $\cC_p$ for $p \in \pos$ through an iterated gluing process, or more generally as a cofiltered limit of such gluings if $\pos$ is infinite.  This is exactly the content of the main result of this subsection (\Cref{thm.fracs.are.llax.modules}), which is stated as an equivalence of $\infty$-categories, whose constituent data we describe now.

\begin{notation}
We write
\[
\StCat_{\fr/\pos}
\]
for the $\infty$-category of stable $\infty$-categories fractured over $\pos$ and exact fracture-preserving functors between them.
\end{notation}

{somewhere, Observe that fractures of disjoint unions are directly seen to correspond to product decompositions.}

\begin{definition}
A module (of any flavor, including lax) is called \bit{stable} if its fibers are stable $\infty$-categories and its monodromy functors are exact.
\end{definition}

\begin{notation}
We write
\[
\StLMod^\rlax_{\llax.\pos}
\longhookra
\LMod^\rlax_{\llax.\pos}
\]
for the subcategory on the stable left-lax left $\pos$-modules and the fiberwise exact functors between them.
\end{notation}

\begin{construction}
Given a fracture
\[
\Conv_\pos
\xlongra{\fF}
\Loc_\cC
~,
\]
we construct a stable left-lax left $\pos$-module
\[
\makemod(\fF)
\in
\LMod_{\llax.\pos}
\]
as the full subcategory
\[ \begin{tikzcd}
\makemod(\fF)
\arrow[hook]{r}{\ff}
\arrow{d}
&
\cC \times \pos
\arrow{ld}
\\
\pos
\end{tikzcd} \]
on those pairs of the form
\[
(X \in \cC_p ~,~ p)
\in
\cC \times \pos
~;
\]
over a morphism $p \ra q$ in $\pos$, the cocartesian monodromy functor is thus given by the composite
\[
\cC_p
\longhookra
\cC
\xra{\locL_q}
\cC_q
~.
\]
This evidently defines a functor
\[
\StCat_{\fr/\pos}
\xra{\makemod}
\StLMod^\rlax_{\llax.\pos}
~.
\]
\end{construction}

\begin{construction}
Given a stable left-lax left $\pos$-module
\[
\cE
\in
\StLMod^\rlax_{\llax.\pos}
~,
\]
we construct a stable $\infty$-category fractured over $\pos$ by the prescription
\[
\posQ
\longmapsto
\lim^\rlax \left( \posQ \overset{\llax}{\lacts} \cE \right)
\]
taking an object $\posQ \in \Conv_\pos$ to the right-lax limit of the pullback of $\cE$ along the inclusion $\posQ \hookra \pos$ (which is left implicit in the notation).  Note that this does indeed define an object
\[
\lim^\rlax(\cE)
\in
\StCat_{\fr/\pos}
~:
\]
these right-lax limits are easily seen to be stable $\infty$-categories, right-lax limits are contravariantly functorial in the base $\infty$-category under restriction, and the requirement that an element $\posQ \in \Conv_\pos$ be a convex subset of $\pos$ (rather than an arbitrary subposet) guarantees that the resulting restriction
\[
\lim^\rlax \left( \pos \overset{\llax}{\lacts} \cE \right)
\longra
\lim^\rlax \left( \posQ \overset{\llax}{\lacts} \cE \right)
\]
admits a (diagramatically defined) fully faithful right adjoint.  As right-lax limits are functorial for right-lax equivariant functors, it is easy to see that this construction defines a functor
\[
\StLMod^\rlax_\pos
\xra{\limrlaxfam}
\StCat_{\fr/\pos}
~.
\]
\end{construction}

\begin{theorem}
\label{thm.fracs.are.llax.modules}
The functors
\[ \begin{tikzcd}[column sep=2cm, row sep=0cm]
\StCat_{\fr/\pos}
\arrow[transform canvas={yshift=0.9ex}]{r}{\makemod}[swap]{\sim}
\arrow[leftarrow, transform canvas={yshift=-0.9ex}]{r}[swap]{\limrlaxfam}
&
\StLMod^\rlax_{\llax.\pos}
\end{tikzcd} \]
are canonically inverse equivalences of $\infty$-categories.
\end{theorem}

\begin{proof}
First of all, we observe that both side are evidently compatible with filtered colimits in the variable $\pos$ (i.e.\! they take filtered colimits of posets to cofiltered limits of $\infty$-categories) -- the left side by condition \Cref{require.cofiltered.limit} of \Cref{define.fracture}.  Thus, it suffices to prove the claim in the case that $\pos$ is finite.

Now, the composite $\makemod \circ \limrlaxfam$ is obviously naturally equivalent to the identity, by construction.  So it remains to show that the composite $\limrlaxfam \circ \makemod$ is as well.  We begin by constructing a natural comparison morphism.  Suppose that the functor
\[
\Conv_\pos
\xlongra{\fF}
\Loc_\cC
\]
defines a fracture of $\cC$ over $\pos$.  Let us consider the defining inclusion
\[
\cC \times \pos
\longhookla
\makemod(\fF)
\]
as a morphism in $\LMod^\llax_{\llax.\pos}$.  By construction, over each object $p \in \pos$ this is a right adjoint, and so their left adjoints assemble into a canonical morphism
\[
\cC \times \pos
\longra
\makemod(\fF)
\]
in $\LMod^\rlax_{\llax.\pos}$ by \Cref{lemma.ptwise.radjt.has.ptwise.ladjt}.  This allows us to construct the composite
\[
\cC
\simeq
\lim \left( \pos \lacts ( \cC \times \pos ) \right)
\longhookra
\lim^\rlax \left( \pos \lacts (\cC \times \pos) \right)
\longra
\lim^\rlax \left( \pos \overset{\llax}{\lacts} \makemod(\fF) \right)
=:
\left( \limrlaxfam ( \makemod(\fF) ) \right)_\pos
~;
\]
for any $\posQ \in \Conv_\pos$, in light of the factorization
\[ \begin{tikzcd}
\cC \times \pos
\arrow[hookleftarrow]{r}
&
\makemod(\fF)
\\
\cC_\posQ \times \posQ
\arrow[hook]{u}
\arrow[dashed, hookleftarrow]{r}
&
\makemod(\fF)_{|\posQ}
\arrow[hook]{u}
\end{tikzcd}~,
\]
an analogous construction yields a functor
\[
\cC_\posQ
\longra
\left( \limrlaxfam ( \makemod(\fF) ) \right)_\posQ
~. \]
Taken together, these functors assemble into a morphism
\begin{equation}
\label{comparison.morphism.of.frac.cats}
\fF
\longra
\limrlaxfam(\makemod(\fF))
\end{equation}
in $\StCat_{\fr/\pos}$, which it remains to show is an equivalence.  We will show this by induction on the number of elements of $\pos$.  The case where $\pos = \pt$ is trivial, while the case where $\pos = [1]$ is immediate: a fracture
\[
\Conv_{[1]}
\xlongra{\fF}
\Loc_\cC
\]
is precisely the data of a diagram
\[ \begin{tikzcd}
&
\cC_1
\\
\cC_0
\arrow[hook]{r}
&
\cC
\arrow[hookleftarrow]{u}
\end{tikzcd} \]
of reflective subcategories of $\cC$ such that the functor
\[
\cC
\longra
\lim^\rlax \left( [1] \lacts \cC_0 \backslash \cC_1 \right)
\simeq
\lim^\rlax \left( [1] \overset{\llax}{\lacts} \makemod(\fF) \right)
\]
is an equivalence.

So suppose that $\pos$ is a finite poset.  By our inductive assumption, the morphism \Cref{comparison.morphism.of.frac.cats} is an equivalence at all non-maximal elements of $\Conv_\pos$, so it only remains to check that it is an equivalence at the maximal element $\pos \in \Conv_\pos$.  For this, let us choose a maximal element $\infty \in \pos$; this induces a functor $\pos \ra [1]$ whose fiber over $0 \in [1]$ is $\pos' := \pos \backslash \{ \infty \}$ and whose fiber over $1 \in [1]$ is $\{ \infty \}$.  Through this choice, the morphism \Cref{comparison.morphism.of.frac.cats} restricts to a morphism in $\StCat_{\fr/[1]}$ which is an equivalence at both singletons of $\Conv_{[1]}$ -- at the object $0 \in \Conv_{[1]}$ it is given by the morphism
\[
\cC_{\pos'}
\longra
\lim^\rlax \left( \pos' \overset{\llax}{\lacts} \makemod(\fF) \right)
~,
\]
which is an equivalence by the inductive assumption, while at the object $1 \in \Conv_{[1]}$ it is given by the morphism
\[
\cC_\infty
\longra
\lim^\rlax \left( \{ \infty \} \overset{\llax}{\lacts} \makemod(\fF) \right)
~,
\]
which is an equivalence by inspection -- and for which it remains to show that the morphism
\[
\cC
\longra
\lim^\rlax \left( \pos \overset{\llax}{\lacts} \makemod(\fF) \right)
\]
at the remaining object $[1] \in \Conv_{[1]}$ is an equivalence.  With the case of $\pos = [1]$ in hand, we see that this morphism is equivalently given by a morphism in $\StLMod^\rlax_{\llax.[1]}$ which is an equivalence on fibers, and for which it remains to show that the induced natural transformation
\begin{equation}
\label{square.with.nat.trans.giving.map.of.frac.cats.over.walking.arrow}
\begin{tikzcd}[row sep=1.5cm, column sep=1.5cm]
\cC_{\pos'}
\arrow{r}{\sim}[swap, transform canvas={xshift=0.5cm, yshift=-1cm}]{\rotatebox{45}{$\Rightarrow$}}
\arrow{d}
&
\lim^\rlax \left( \pos' \overset{\llax}{\lacts} \makemod(\fF) \right)
\arrow{d}
\\
\cC_\infty
\arrow{r}[swap]{\sim}
&
\lim^\rlax \left( \{ \infty \} \overset{\llax}{\lacts} \makemod(\fF) \right)
\end{tikzcd}
\end{equation}
is a natural equivalence.  Observe that by our inductive assumption, every object of $\cC_{\pos'}$ is a finite limit of objects of the reflective subcategories $\cC_p \subset \cC_{\pos'}$ for $p \in \pos'$.  Since both vertical functors in the diagram \Cref{square.with.nat.trans.giving.map.of.frac.cats.over.walking.arrow} are exact, it therefore suffices to show that the natural transformation therein restricts to a natural equivalence on these subcategories.  Unwinding the definitions, we see that on any such subcategory $\cC_p \subset \cC_{\pos'}$ both composites are canonically equivalent to the composite
\[
\cC_p
\longhookra
\cC
\xra{\locL_\infty}
\cC_\infty
\]
(and the natural transformation of diagram \Cref{square.with.nat.trans.giving.map.of.frac.cats.over.walking.arrow} witnesses their resulting equivalence with each other), which proves the claim.
\end{proof}

\begin{remark}
A fracture of $\cC$ over a disjoint union of posets is equivalent data to to a product decomposition of $\cC$ and a fracture of the factors of $\cC$ over the poset summands: this follows from \Cref{thm.fracs.are.llax.modules}, perhaps more readily than from \Cref{define.fracture}.
\end{remark}

\subsection{The canonical fracture of genuine $G$-spectra}
\label{subsection.can.frac.of.gen.spt}

In this section, we prove a general result (\Cref{thm.dcc.and.fracturing.gives.fracture}) which provides sufficient conditions to obtain a fracture of a stable $\infty$-category over a poset.  We then apply this to show that the stable $\infty$-category $\Spectra^{\gen G}$ of genuine $G$-spectra is canonically fractured over the poset $\pos_G$ of closed subgroups of $G$ ordered by subconjugacy (\Cref{prop.can.fracture.of.gen.G.spt}).  Using this, we deduce naive reidentifications of $\Spectra^{\gen G}$ and $\Spectra^{\gen^\proper G}$ (and thereafter \Cref{mainthm.genuine.G.spt} and \Cref{mainthm.cyclonic.spt}) as Corollaries \ref{cor.gen.G.spt.as.rlax.lim} \and \ref{cor.proper.gen.G.spt.as.rlax.lim}, respectively.

\begin{notation}
Throughout this subsection, we fix a presentable stable $\infty$-category $\cC$.
\end{notation}


\begin{definition}
A full stable subcategory $\cC_0 \subset \cC$ is called \bit{admissible} if its inclusion is colimit-preserving with colimit-preserving right adjoint.
\end{definition}

\begin{example}
\label{ex.of.adm.subcat.gend.by.cpct.objs}
Given a set of compact objects in $\cC$, the full stable subcategory that they generate under colimits is admissible.
\end{example}

\begin{observation}
By presentability, to say that $\cC_0 \subset \cC$ is admissible is precisely to say that its inclusion fits into a diagram
\[ \begin{tikzcd}[column sep=1.5cm]
\cC_0
\arrow[hook, bend left]{r}
\arrow[dashed,leftarrow]{r}[transform canvas={yshift=0.05cm}]{\bot}[swap,transform canvas={yshift=-0.05cm}]{\bot}
\arrow[dashed,bend right]{r}
&
\cC
\end{tikzcd} \]
of adjunctions.
\end{observation}

\begin{observation}
\label{obs.admissible.iff.right.orthog.colim.closed}
A full stable subcategory $\cC_0 \subset \cC$ which is closed under colimits is admissible if and only if its right-orthogonal subcategory
\[
\{ Z \in \cC : \ulhom_\cC(X,Z) \simeq 0 \textup{ for all }X \in \cC_0 \}
\subset 
\cC
\]
is also closed under colimits.
\end{observation}

\begin{observation}
\label{obs.adm.subcat.gives.recollement}
Let $\cC_0 \subset \cC$ be an admissible subcategory, and write
\[
\cC_0
\xlonghookra{i_L}
\cC
\xlongra{p_L}
\cC / \cC_0
=:
\cC_1
\]
for the presentable quotient (the cofiber in $\PrL$).  By the admissibility of $\cC_0 \subset \cC$, this extends to a diagram
\begin{equation}
\label{recollement.of.C}
\begin{tikzcd}[column sep=2cm]
\cC_0
\arrow[hook,bend left=60]{r}{i_L}
\arrow[leftarrow]{r}[pos=0.7]{\yo}[transform canvas={yshift=0.25cm}]{\bot}[swap,transform canvas={yshift=-0.25cm}]{\bot}
\arrow[hook,bend right=60]{r}[swap]{i_R}
&
\cC
\arrow[bend left=60]{r}{p_L}
\arrow[hookleftarrow]{r}[pos=0.7]{\nu}[transform canvas={yshift=0.25cm}]{\bot}[swap,transform canvas={yshift=-0.25cm}]{\bot}
\arrow[bend right=60]{r}[swap]{p_R}
&
\cC_1
\end{tikzcd}
\end{equation}
of adjunctions among presentable stable $\infty$-categories, in which all inclusions are fully faithful and all three composites are exact in the sense that we have the equalities
\[
\ker(p_L)
=
\im(i_L)
~,
\qquad
\ker(\yo)
=
\im(\nu)
~,
\qquad
\textup{and}
\qquad
\ker(p_R)
=
\im(i_R)
\]
among full subcategories of $\cC$.
\end{observation}

\begin{definition}
Diagram \Cref{recollement.of.C} defines a \bit{recollement} of the stable $\infty$-category $\cC$.
\end{definition}


\begin{observation}
It is classical that the recollement \Cref{recollement.of.C} of $\cC$ is equivalent to a fracture of $\cC$ over $[1]$, as in \Cref{constr.localization.to.gluing.category}: the functor
\[
\Conv_{[1]}
\longra
\Loc_\cC
\]
selects the diagram
\[ \begin{tikzcd}
&
\cC_1
\\
\cC_0
\arrow[hook]{r}[swap]{i_L}
&
\cC
\arrow[hookleftarrow]{u}[swap]{\nu}
\end{tikzcd} \]
(see e.g.\! \cite[\S A.8]{LurieHA}).  This fact is also straightforward to verify.  On the one hand, the functor
\[
\cC
\xra{\Cref{localizn.to.gluing.cat}}
\lim^\rlax([1] \lacts (\cC_0 \backslash \cC_1))
\]
described informally by the prescription
\[
X
\longmapsto
\left(
\yo X \in \cC_0
~,~
p_L X \in \cC_1
~,~
\begin{tikzcd}
p_L X
\arrow{d}
\\
p_L i_R \yo X
\end{tikzcd}
\right)
\]
and the functor
\[
\lim^\rlax([1] \lacts (\cC_0 \backslash \cC_1))
\longra
\cC
\]
described informally by the prescription
\[
\left(
X_0 \in \cC_0
~,~
X_1 \in \cC_1
~,~
\begin{tikzcd}
X_1
\arrow{d}
\\
p_L i_R X_0
\end{tikzcd}
\right)
\longmapsto
\lim \left( \begin{tikzcd}
&
\nu X_1
\arrow{d}
\\
i_R X_0
\arrow{r}
&
\nu p_L i_R X_0
\end{tikzcd} \right)
\]
are easily checked to define inverse equivalences, so that we have indeed defined a fracture of $\cC$ over $[1]$.\footnote{The composite endofunctor of $\lim^\rlax([1] \lacts (\cC_0 \backslash \cC_1))$ is immediately seen to be an equivalence.  To see that the composite endofunctor of $\cC$ is an equivalence, one must check that for any $X \in \cC$ the commutative square
\[ \begin{tikzcd}[ampersand replacement=\&]
X
\arrow{r}
\arrow{d}
\&
\nu X_1
\arrow{d}
\\
i_R \yo X
\arrow{r}
\&
\nu p_L i_R \yo X
\end{tikzcd} \]
is a pullback square; it suffices to check that the fibers of the horizontal functors are equivalent, and this follows from the equality $\ker(\yo) = \im(\nu)$.}  On the other hand, a fracture of $\cC$ over $[1]$ defines the data of a recollement diagrammatically.
\end{observation}

\begin{notation}
We assume that for each element $p \in \pos$ we are given an admissible subcategory $\cK_p \subset \cC$.  Then, for any $\posQ \in \Conv_\pos$, we write $\cK_\posQ \subset \cC$ for the full stable subcategory generated under colimits by the subcategories $\{ \cK_q \}_{q \in \posQ}$.  Manifestly, morphisms in $\Conv_\pos$ then determine inclusions of subcategories of $\cC$.
\end{notation}

\begin{observation}
\label{obs.KQ.in.C.admissible.and.KQprime.in.KQ.admissible}
By \Cref{obs.admissible.iff.right.orthog.colim.closed}, the subcategory $\cK_\posQ \subset \cC$ is admissible, and more generally for a morphism $\posQ' \ra \posQ$ in $\Conv_\pos$ the subcategory $\cK_{\posQ'} \subset \cK_\posQ$ is admissible.\footnote{In fact, this is true more generally for any subsets of the underlying set of $\pos$: its poset structure plays no role here.}
\end{observation}

\begin{notation}
For any $\posQ \in \Conv_\pos$, we write
\[ ^\leq \posQ \in \Conv_\pos \]
for the downwards closure of $\posQ$ in $\pos$ (the subposet of all objects of $\pos$ admitting a map to some $q \in \posQ$), and we write
\[
^< \posQ
:=
(^\leq \posQ) \backslash \posQ \in \Conv_\pos
\]
for its subset of elements not contained in $\posQ$.
\end{notation}

\begin{observation}
By construction, for any $\posQ \in \Conv_\pos$ both $^\leq \posQ$ and $^<\posQ$ are downwards closed.
\end{observation}

\begin{notation}
For any $\posQ \in \Conv_\pos$, we write
\[
\cK_{^<\posQ}
\xlongra{i_L}
\cK_{^\leq\posQ}
\xlongra{p_L}
\cK_{^\leq\posQ} / \cK_{^<\posQ}
=:
\cC_\posQ
\]
for the presentable quotient, which extends to a recollement of $\cK_{^\leq\posQ}$ by \Cref{obs.KQ.in.C.admissible.and.KQprime.in.KQ.admissible}.  We then write
\[
\rho:
\cC_\posQ
\xlonghookra{\nu}
\cC_{^\leq \posQ}
\simeq
\cK_{^\leq \posQ}
\xlonghookra{i_R}
\cC
\]
for the composite of fully faithful right adjoints, and we write
\[ \begin{tikzcd}[column sep=1.5cm]
\cC
\arrow[dashed,transform canvas={yshift=0.9ex}]{r}{\lambda}
\arrow[hookleftarrow, transform canvas={yshift=-0.9ex}]{r}[yshift=-0.2ex]{\bot}[swap]{\rho}
&
\cC_\posQ
\end{tikzcd} \]
for its left adjoint.
\end{notation}

\begin{warning}
For a downwards closed convex subset $^\leq \posQ \in \Conv_\pos$ of $\pos$, the two inclusions
\[
\cC_{^\leq \posQ}
\xlonghookra{\rho}
\cC
\qquad
\textup{and}
\qquad
\cC_{^\leq \posQ}
\simeq
\cK_{^\leq \posQ}
\xlonghookra{i_L}
\cC
\]
are distinct.
\end{warning}

\begin{notation}
For any $\posQ_0 , \posQ_1 \in \Conv_\pos$, we write
\[
\Psi_{\posQ_0,\posQ_1}
:
\cC_{\posQ_0}
\xlongra{\rho}
\cC
\xlongra{\lambda}
\cC_{\posQ_1}
\]
for the composite.
\end{notation}

\begin{remark}
In the case that our data define a stratification, and hence a stable left-lax left $\pos$-module through \Cref{thm.fracs.are.llax.modules}, the cocartesian monodromy functor thereof over a morphism $p_0 \ra p_1$ in $\pos$ will be given by
\[
\cC_{p_0}
\xra{\Psi_{p_0,p_1}}
\cC_{p_1}
~.
\]
\end{remark}

\begin{lemma}
\label{lemma.get.ladjt.factorizn}
Let $\posQ \in \Conv_\pos$ be equipped with a functor $\posQ \ra [1]$.  Then there exists a factorization
\begin{equation}
\label{factorizn.defining.iota.Qzero.Q}
\begin{tikzcd}
\cK_{\posQ_0}
\arrow{r}{i_L}
\arrow{d}[swap]{i_L}
&
\cK_{^\leq \posQ}
\arrow{r}{p_L}
&
\cC_\posQ
\\
\cK_{^\leq \posQ_0}
\arrow{r}[swap]{p_L}
&
\cC_{\posQ_0}
\arrow[dashed]{ru}[swap, sloped, pos=0.1]{\iota_{\posQ_0,\posQ}}
\end{tikzcd}
\end{equation}
which is left adjoint to $\cC_\posQ \xra{\Psi_{\posQ,\posQ_0}} \cC_{\posQ_0}$ and whose image coincides with that of the upper composite in diagram \Cref{factorizn.defining.iota.Qzero.Q}.
\end{lemma}

\begin{proof}
Consider the functor
\[
\Psi_{\posQ,\posQ_0}
:
\cC_\posQ
\xlongra{\nu}
\cC_{^\leq\posQ}
\xlongra{\yo}
\cC_{^\leq\posQ_0}
\xlongra{p_L}
\cC_{\posQ_0}
~.
\]
By assumption, we have that $(^<\posQ_0) \subset (^<\posQ)$, and therefore we have the composite inclusion
\[
\yo(\nu(\cC_\posQ))
\longhookra
\nu(\cC_{\posQ_0})
\longhookra
\cC_{^\leq \posQ_0}
\]
of full subcategories.  Hence, the composite
\[
\iota_{\posQ_0,\posQ}
:
\cC_{\posQ_0}
\xlongra{\nu}
\cC_{^\leq \posQ_0}
\xra{i_L}
\cC_{^\leq \posQ}
\xra{p_L}
\cC_\posQ
\]
is left adjoint to $\Psi_{\posQ,\posQ_0}$.  Moreover, since the composite
\[
\cK_{^<\posQ_0}
\simeq
\cC_{^<\posQ_0}
\xlongra{i_L}
\cK_{^\leq\posQ}
\simeq
\cC_{^\leq\posQ}
\xlongra{p_L}
\cC_{\posQ}
\]
vanishes we obtain the desired factorization, and since the composite
\[
\cK_{\posQ_0}
\longra
\cK_{^\leq\posQ_0}
\xlongra{p_L}
\cC_{\posQ_0}
\]
is surjective we obtain the asserted equality of images.
\end{proof}

\begin{definition}
We say that a poset $\pos$ satisfies the \bit{dcc} if it satisfies the descending chain condition, namely that every decreasing sequence stabilizes, or equivalently that every subset has a minimal element.
\end{definition}

\begin{observation}
If a poset satisfies the dcc, then so does any subposet.
\end{observation}

\begin{lemma}
\label{lemma.if.dcc.then.conservative}
Suppose that $\pos$ satisfies the dcc.  Then, for any $\posQ \in \Conv_\pos$, the functor
\[
\cC_\posQ
\xra{ (\Psi_{\posQ,q}) }
\prod_{q \in \posQ}
\cC_q
\]
is conservative.
\end{lemma}

\begin{proof}
Suppose that $X \in \cC_\posQ$ with $(\Psi_{\posQ,q})(X) \simeq 0$.  Observe that the functor
\[
\cC_{^\leq\posQ}
\xra{ \left( \Psi_{^\leq\posQ,^\leq q} \right) }
\prod_{q \in \posQ} \cC_{^\leq q}
\]
is clearly conservative, and hence so is the composite
\begin{equation}
\label{conservative.composite.from.CQ}
\cC_\posQ
\longra
\cC_{^\leq \posQ}
\xra{ \left( \Psi_{^\leq\posQ,^\leq q} \right) }
\prod_{q \in \posQ} \cC_{^\leq q}
~.
\end{equation}
So it suffices to show that for every $q \in \posQ$, the image $\Psi_{\posQ,^\leq q}(X) \in \cC_{^\leq q}$ under the composite \Cref{conservative.composite.from.CQ} followed by the projection is zero.

Suppose otherwise.  Using the dcc for $\posQ$, we can choose a minimal element $q \in \posQ$ witnessing the failure of our claim, i.e.\! which is minimal among elements $q \in \posQ$ such that $\Psi_{\posQ,^\leq q}(X)\in \cC_{^\leq q}$ is nonzero.  As by assumption $X$ is taken to zero under the composite
\[
\cC_\posQ
\longra
\cC_{^\leq \posQ}
\longra
\cC_{^\leq q}
\longra
\cC_q
~,
\]
it follows that $\Psi_{\posQ,^\leq q}(X) \in \cC_{^\leq q}$ lies in the image of
\[
\cC_{^< q}
\xlongra{i_L}
\cC_{^\leq q}
~.
\]

Now, observe that $\cC_{^< q}$ is generated by $\cK_{(^< \posQ) \cap (^< q)}$ and $\cK_{\posQ \cap (^< q)}$.  For any $Y \in \cK_{^< \posQ}$ and hence in particular for any $Y \in \cK_{(^<\posQ ) \cap (^< q)}$, we have
\[
\ulhom_\cC(Y,\Psi_{\posQ,^\leq q}(X)) \simeq 0
\]
by definition of $\cC_\posQ$.  On the other hand, by the assumed minimality of $q \in \posQ$, it must be that for any $Z \in \cK_{\posQ \cap (^< q)}$, we have
\[
\ulhom_\cC(Z,\Psi_{\posQ,^\leq q}(X)) \simeq 0
\]
as well.  Thus $\Psi_{\posQ,^\leq q}(X) \simeq 0$, which is a contradiction.
\end{proof}

\begin{definition}
\label{defn.fracturing}
We say that an assignment of an admissible subcategory $\cK_p \subset \cC$ to each element $p \in \pos$ defines a \bit{fracturing} of $\cC$ over $\pos$ if for any pair of elements $p_0,p_1 \in \pos$ with $\hom_\pos(p_1,p_0) \simeq \es$, the functor
\[
\cC_{p_1}
\xra{\Psi_{p_1,p_0}}
\cC_{p_0}
\]
is zero.
\end{definition}

\begin{remark}
Heuristically, the idea of \Cref{defn.fracturing} is as follows.  Of course, we would like our resulting assignment $\posQ \mapsto \cC_\posQ$ to define a fracture $\fF$ of $\cC$ over $\pos$.  Recall from \Cref{thm.fracs.are.llax.modules} that this determines a reidentification of $\cC$ as the right-lax limit of a left-lax left $\pos$-module $\makemod(\fF)$, whose value on $p \in \pos$ is the subcategory $\cC_p$.  Now, if there is no map from $p_1$ to $p_0$ in $\pos$, we cannot possibly expect to recover ``gluing data'' running from $\cC_{p_1}$ to $\cC_{p_0}$ inside of $\lim^\rlax(\makemod(\fF))$.  So, \Cref{defn.fracturing} amounts to the requirement that such gluing data is unnecessary (because it is necessarily trivial).
\end{remark}

\begin{lemma}
\label{lemma.fracturing.implies.backwards.monodromy.vanishes}
Suppose that $\pos$ satisfies the dcc, and suppose we are given any fracturing of $\cC$ over $\pos$.  Then, for any $\posQ \in \Conv_\pos$ and any functor $\posQ \ra [1]$, the functor
\[
\cC_{\posQ_1}
\xra{\Psi_{\posQ_1,\posQ_0}}
\cC_{\posQ_0}
\]
is zero.
\end{lemma}

\begin{proof}  
Suppose otherwise.  Using the dcc for $\posQ_1$, we see that there must exist some minimal element $q_1 \in \posQ_1$ with the property that the composite
\[
\cK_{q_1}
\longra
\cC_{\posQ_1}
\xra{\Psi_{\posQ_1,\posQ_0}}
\cC_{\posQ_0}
\]
is nonzero, since as $q_1 \in \posQ_1$ varies the $\cK_{q_1}$ generate $\cC_{\posQ_1}$.  Hence, the functor $\Psi_{\posQ_1,\posQ_0}$ factors through the quotient
\[
\cC_{\posQ_1}
\longra
\cC_{\posQ_1} / \cK_{\posQ_1 \cap (^<q_1)}
\simeq
\cK_{^{\leq \posQ_1}} / \cK_{(^<\posQ_1) \cap (^<q_1)}
~.
\]
In particular, we obtain a factorization
\[ \begin{tikzcd}
\cK_{q_1}
\arrow{r}
\arrow{d}
&
\cC_{\posQ_1} / \cK_{\posQ_1 \cap (^<q_1)}
\arrow{r}
&
\cC_{^\leq \posQ_1}
\\
\cC_{q_1}
\arrow[dashed, bend right=10]{rru}
\end{tikzcd}~, \]
so that the functor
\[
\cC_{q_1}
\xra{\Psi_{q_1,\posQ_0}}
\cC_{\posQ_0}
\]
must be nonzero, which implies that the functor
\[
\cC_{q_1}
\xra{\Psi_{q_1,(^\leq \posQ_0)}}
\cC_{^\leq\posQ_0}
\]
must also be nonzero.

Now, note that for any $q \in \posQ_0$, we have a factorization
\[ \begin{tikzcd}[column sep=1.5cm]
\cC_{q_1}
\arrow{rr}{\Psi_{q_1,(^\leq q_0)}}
\arrow{rd}[sloped,swap, pos=0.9]{\Psi_{q_1,(^\leq \posQ_0)}}
&
&
\cC_{^\leq q_0}
\\
&
\cC_{^\leq \posQ_0}
\arrow[dashed]{ru}
\end{tikzcd}~. \]
Hence, using the dcc for $^\leq \posQ_0$, we can choose a minimal element $q_0 \in (^\leq \posQ_0)$ such that
\[
\cC_{q_1}
\xra{\Psi_{q_1,(^\leq q_0)}}
\cC_{^\leq q_0}
\]
is nonzero.

By the definition of a fracturing, it must be that the composite
\[
\cC_{q_1}
\xra{\Psi_{q_1,(^\leq q_0)}}
\cC_{^\leq q_0}
\longra
\cC_{q_0}
\]
is zero.  Therefore, the functor
\[
\cC_{q_1}
\xra{\Psi_{q_1,(^<q_0)}}
\cC_{^<q_0}
\]
is nonzero.  In particular, there exists some $q \in (^< q_0)$ such that the functor
\[
\cC_{q_1}
\xra{\Psi_{q_1,(^\leq q)}}
\cC_{^\leq q}
\]
is nonzero.  But this contradicts the minimality of $q_0 \in (^\leq \posQ_0)$, proving the claim.
\end{proof}

\begin{theorem}
\label{thm.dcc.and.fracturing.gives.fracture}
Suppose that $\pos$ satisfies the dcc.  Then a fracturing of the presentable stable $\infty$-category $\cC$ over $\pos$ such that the subcategories $\{ \cK_p \}_{p \in \pos}$ generate $\cC$ defines a fracture
\[ \begin{tikzcd}[column sep=2cm, row sep=0cm]
\Conv_\pos
\arrow{r}
&
\Loc_\cC
\\
\rotatebox{90}{$\in$}
&
\rotatebox{90}{$\in$}
\\
\posQ
\arrow[mapsto]{r}
&
\rho(\cC_\posQ)
\end{tikzcd} \]
of $\cC$ over $\pos$.
\end{theorem}

The proof of \Cref{thm.dcc.and.fracturing.gives.fracture} will require the following general result.

\begin{lemma}
\label{lemma.limit.of.right.adjt.inclns.in.PrL}
Suppose we are given two diagrams $C_\bullet,D_\bullet \in \Fun(\cI,\PrL)$ whose transitions maps preserve all limits, and suppose we are given a natural transformation
\[
C_\bullet
\xlongra{F_\bullet}
D_\bullet
\]
in $\Fun(\cI,\PrL)$ such that each component $F_i$ is a limit-preserving reflective localization with right adjoint $G_i$.  Then the induced functor
\[
F
:
\lim_\cI(C_\bullet)
\xra{\lim_\cI(F)}
\lim_\cI(D_\bullet)
\]
is also a reflective localization.
\end{lemma}

\begin{proof}
We claim that the right adjoint inclusion is the functor
\[
\lim_\cI ( C_\bullet)
\xlongla{G}
\lim_\cI( D_\bullet )
\]
given by taking a compatible system $Y_\bullet := (Y_i \in D_i)_{i \in \cI}$ to the compatible system whose value at $i \in \cI$ is given by the formula
\[
\lim_{(j \da i) \in \cI_{/i}}
C_{j \da i}(G_j(Y_j))
~.\footnote{The condition that the transition maps of $C_\bullet$ and $D_\bullet$ preserve limits is implicitly used here to guarantee that this formula indeed defines a functor $\cI \ra \PrL$.}   
\]
First, to check that $F \circ G$ is the identity functor it suffices to check at $D_i$, and for this we compute that
\begin{align*}
F_i \left( \lim_{(j \da i) \in \cI_{/i}}
C_{j \da i}(G_j(Y_j)) \right)
&
\simeq
\lim_{(j \da i) \in \cI_{/i}}
F_i ( C_{j \da i}(G_j(Y_j)) )
\\
&
\simeq
\lim_{(j \da i) \in \cI_{/i}} D_{j \da i}(F_j(G_j(Y_j)))
\\
&
\simeq
\lim_{(j \da i) \in \cI_{/i}} D_{j \da i}(Y_j)
\\
&
\simeq
Y_i
~.
\end{align*}
Then, to check that $G$ is indeed a right adjoint to $F$, for any compatible systems $X_\bullet := (X_i \in C_i)_{i \in \cI}$ and $Y_\bullet := (Y_i \in D_i)_{i \in \cI}$ we compute that
\begin{align}
\nonumber
\hom_{\lim_\cI(C_\bullet)}(X_\bullet,G(Y_\bullet))
& \simeq
\lim_{(j \ra i) \in \Ar(\cI)}
\hom_{C_i}(X_i,C_{j\da i} G_j(Y_j))
\\
\label{use.initiality.of.id.arrows.in.Ar.I.op}
& \xlongra{\sim}
\lim_{i \in \cI} \hom_{C_i}(X_i,G_i(Y_i))
\\
\nonumber
& \simeq
\lim_{i \in \cI} \hom_{D_i}(F_i(X_i),Y_i)
\\
\nonumber
&
\simeq \hom_{\lim_\cI(D_\bullet)}(F(X_\bullet),Y_\bullet)
~,
\end{align}
where the equivalence \Cref{use.initiality.of.id.arrows.in.Ar.I.op} uses the initiality of the full subcategory on the equivalences inside of $\Ar(\cI)$.
\end{proof}

\begin{proof}[Proof of \Cref{thm.dcc.and.fracturing.gives.fracture}]
We first must check that this assignment indeed defines a functor $\Conv_\pos \ra \Loc_\cC$ of posets.  For this, given any morphism $\posQ' \ra \posQ$ in $\Conv_\pos$, we must show an inclusion $\rho(\cC_{\posQ'}) \subset \rho(\cC_\posQ)$ of full subcategories of $\cC$.  This is clear if both $\posQ'$ and $\posQ$ are downwards closed, because by definition $\cK_{\posQ'} \subset \cK_\posQ$ and the inclusions $\rho$ of $\cC_\posQ$ and $\cC_{\posQ'}$ are respectively the right adjoints of the left adjoint projections
\[
\cC
\xlongra{\yo}
\cC_\posQ
\xlongra{\yo}
\cC_{\posQ'}
\]
from $\cC$.  In the general case, we have that the inclusion $\rho$ of $\cC_{\posQ'}$ is given by the composite
\[
\cC_{\posQ'}
\xlongra{\nu}
\cC_{^\leq \posQ'}
\xra{i_R}
\cC_{^\leq \posQ}
\xra{i_R}
\cC
~.
\]
So it suffices to show that the image of $\cC_{\posQ'}$ in $\cC_{^\leq \posQ}$ lies inside the image of $\cC_\posQ$, i.e.\! that the composite
\[
\cC_{\posQ'}
\xlongra{\nu}
\cC_{^\leq \posQ'}
\xra{i_R}
\cC_{^\leq \posQ}
\xlongra{\yo}
\cC_{^<\posQ}
\]
is zero.  This follows immediately from \Cref{lemma.fracturing.implies.backwards.monodromy.vanishes} applied to the element $\tilde{Q} := (^< \posQ) \sqcup \posQ' \in \Conv_\pos$ equipped with the functor $\tilde{Q} \ra [1]$ sending $^< \posQ$ to 0 and $\posQ'$ to 1.

Now, condition \Cref{require.frac.preserves.terminal.objs} of \Cref{define.fracture} is satisfied by assumption.

Let us turn to condition \Cref{require.localizn.to.gluing.cat.is.an.equivalence} of \Cref{define.fracture}.  For this, suppose that we have some $\posQ \in \Conv_\pos$ equipped with a functor $\posQ \ra [1]$.  We will show that $\cC_{\posQ_0}$ embeds into $\cC_\posQ$ as an admissible subcategory via the left adjoint $\iota_{\posQ_0,\posQ}$ of the projection given by \Cref{lemma.get.ladjt.factorizn}, such that $\cC_\posQ / \cC_{\posQ_0} \simeq \cC_{\posQ_1}$.

For this, consider the projection functor
\[
\Psi_{\posQ,\posQ_0}
:
\cC_\posQ
\xlongra{\nu}
\cC_{^\leq \posQ}
\xlongra{\yo}
\cC_{^\leq \posQ_0}
\xra{p_L}
\cC_{\posQ_0}
~.
\]
By \Cref{lemma.get.ladjt.factorizn}, this admits a left adjoint $\iota_{\posQ_0,\posQ}$ whose image agrees with that of the composite
\[
\cK_{\posQ_0}
\xra{i_L}
\cK_{^\leq \posQ}
\xra{p_L}
\cC_\posQ
~.
\]
By what we have just seen, the right adjoint of $\Psi_{\posQ,\posQ_0}$ is fully faithful.  Hence, the functor $\iota_{\posQ_0,\posQ}$ is fully faithful as well, so that it witnesses $\cC_{\posQ_0}$ as an admissible subcategory of $\cC_\posQ$.

Now, by \Cref{lemma.fracturing.implies.backwards.monodromy.vanishes} the composite functor
\[
\Psi_{\posQ_1,\posQ_0}
:
\cC_{\posQ_1}
\longhookra
\cC_\posQ
\xra{\Psi_{\posQ,\posQ_0}}
\cC_{\posQ_0}
\]
is zero, so that $\cC_{\posQ_1} \subset \ker(\Psi_{\posQ,\posQ_0})$.  To show that in fact $\cC_{\posQ} = \ker(\Psi_{\posQ,\posQ_0})$, it therefore suffices to show that the composite functor $\cC_{\posQ_1} \ra \cC_\posQ \ra \cC_\posQ / \cC_{\posQ_0}$ is surjective.  But this follows from the facts that $\im(\iota_{\posQ_0,\posQ}) \subset \cC_\posQ$ equals the image of $\cK_{\posQ_0}$ and that the functor $\cK_\posQ \ra \cC_\posQ$ is surjective.  Finally, because $(^< \posQ_1) \subset (^< \posQ) \cup \posQ_0$, then we have a factorization
\[ \begin{tikzcd}
\cK_{^\leq \posQ_1}
\arrow{rr}
\arrow{rd}
&
&
\cC_\posQ / \cC_{\posQ_0}
\\
&
\cC_{\posQ_1}
\arrow[dashed]{ru}
\end{tikzcd}~, \]
from which our desired surjectivity follows.

Finally, we turn to condition \Cref{require.cofiltered.limit} of \Cref{define.fracture}.  For this, suppose that a functor $\cI^\rcone \ra \Conv_\pos$ witnesses $\posQ = \colim_\cI(\posQ_\bullet)$ as a filtered colimit.  We must show that the canonical functor
\[
\cC_\posQ
\longra
\lim_\cI ( \cC_{\posQ_\bullet})
\]
is an equivalence.

For this, for each $i \in \cI$ let $\posQ_i' \subset \posQ_i$ be the maximal subset of $\posQ_i$ which is downwards closed in $\posQ$ (i.e.\! the union of all of its downwards closed subsets).  Since $\posQ$ satisfies the dcc, we also have that $\posQ = \colim_\cI(\posQ'_\bullet)$.  Now, consider the functors
\[
\cC_\posQ
\xra{F_1}
\lim_\cI(\cC_{\posQ_\bullet})
\xra{F_2}
\lim_\cI(\cC_{\posQ'_\bullet})
~.
\]
By \Cref{lemma.limit.of.right.adjt.inclns.in.PrL}, we see that both $F_2$ and $F_2 \circ F_1$ are localizations.  We claim that they are both also conservative, and hence both equivalences, so that $F_1$ is also an equivalence.  That $F_2 \circ F_1$ is conservative follows from \Cref{lemma.if.dcc.then.conservative}.  So it only remains to see that $F_2$ is conservative.

For this, using the notation from the proof of \Cref{lemma.limit.of.right.adjt.inclns.in.PrL} let us suppose that $Y_\bullet := (Y_i \in \cC_{\posQ_i})_{i \in \cI}$ is a compatible system such that $F_2(Y_\bullet) \simeq 0$.  We must show that $Y_i \simeq 0$ for each $i \in \cI$.  By \Cref{lemma.if.dcc.then.conservative}, it suffices to show that $\Psi_{\posQ_i,q}(Y_i) \simeq 0$ for any $q \in \posQ_i$.  So, choose some $j \in \cI$ such that $(^\leq q) \subset \posQ_i'$ (in the notation of the previous paragraph).  Then we have a factorization
\[ \begin{tikzcd}[column sep=1.5cm]
\cC_{\posQ_j}
\arrow{r}{\Psi_{\posQ_j,\posQ_i}}
\arrow{rd}[swap,sloped,pos=0.8]{(F_2)_j}
&
\cC_{\posQ_i}
\arrow{r}{\Psi_{\posQ_i,q}}
&
\cC_q
\\
&
\cC_{\posQ'_j}
\arrow[dashed]{ru}
\end{tikzcd}; \]
since by assumption we have an equivalence $\Psi_{\posQ_j,\posQ_i}(Y_j) \simeq Y_i$ and $(F_2)_j(Y_j) \simeq 0$, it follows that also $\Psi_{\posQ_i,q}(Y_i) \simeq 0$.
\end{proof}

We now apply \Cref{thm.dcc.and.fracturing.gives.fracture} to obtain the canonical stratification of the $\infty$-category $\Spectra^{\gen G}$ of genuine $G$-spectra.

\begin{notation}
\label{fix.notation.for.pos.G.in.general}
We fix a compact Lie group $G$, we write $\Orb_G$ for its orbit $\infty$-category, and we write $\pos_G$ for its poset of closed subgroups ordered by subconjugacy.  We write
\[
\Spaces_*^{\gen G}
:=
\Fun ( \Orb_G^\op,\Spaces_* )
\]
for the $\infty$-category of pointed genuine $G$-spaces, and we write
\[ \begin{tikzcd}[column sep=1.5cm]
\Spaces^{\gen G}_*
\arrow[transform canvas={yshift=0.9ex}]{r}{\Sigma^\infty_G}
\arrow[leftarrow, transform canvas={yshift=-0.9ex}]{r}[yshift=-0.2ex]{\bot}[swap]{\Omega^\infty_G}
&
\Spectra^{\gen G}
\end{tikzcd} \]
for the adjunction.
\end{notation}

\begin{notation}
We write $\cK_H \subset \Spectra^{\gen G}$ for the full stable subcategory generated by $\Sigma^\infty_G(G/H)_+$ under colimits.
\end{notation}

\begin{prop}
\label{prop.can.fracture.of.gen.G.spt}
The assignment to each element $H \in \pos_G$ of the stable subcategory $\cK_H \subset \Spectra^{\gen G}$ defines a canonical fracture
\[
\Conv_{\pos_G}
\xra{\fF_G}
\Loc_{\Spectra^{\gen G}}
\]
of $\Spectra^{\gen G}$ over $\pos_G$.
\end{prop}

\begin{proof}
We apply \Cref{thm.dcc.and.fracturing.gives.fracture}.  Observe first that the poset $\pos_G$ satisfies the dcc by the compactness of the Lie group $G$.  Moreover, because the objects $\{ \Sigma^\infty_G (G/H)_+ \}_{H \in \pos_G}$ form a set of compact generators for $\Spectra^{\gen G}$, then the subcategories $\cK_H \subset \Spectra^{\gen G}$ are admissible (as a special case of \Cref{ex.of.adm.subcat.gend.by.cpct.objs}) and collectively generate $\Spectra^{\gen G}$.  So it remains to check that this assignment indeed defines a fracturing of $\Spectra^{\gen G}$ over $\pos_G$.

For this, suppose we are given closed subgroups $H \leq G$ and $K \leq G$ such that $H$ is not subconjugate to $K$, and consider any $X \in \Spectra^{\gen G}_H \subset \Spectra^{\gen G}$.  It suffices to show that $X^I \simeq 0$ for any $I$ subconjugate to $K$; replacing $K$ by $I$ (since $H$ cannot be subconjugate to $I$), it suffices to show just that $X^K \simeq 0$.

Let us write
\[
\Orb_{G,^\leq H}
\xlonghookra{\eta}
\Orb_G
\]
for the inclusion of the full subcategory on those $G$-orbits admitting a map to $G/H$ (i.e.\! those $G/I$ with $I$ subconjugate to $H$).  Writing
\[
\Spaces_*^{\gen G,^\leq H}
:=
\Fun ( \Orb_{G,^\leq H}^\op,\Spaces_*)
~,
\]
observe the factorization
\[ \begin{tikzcd}[row sep=1.5cm, column sep=1.5cm]
\Spaces^{\gen G,^\leq H}_*
\arrow{r}{\eta_!}
\arrow[dashed]{d}[swap]{\Sigma^\infty_{G,^\leq H}}
&
\Spaces^{\gen G}_*
\arrow{d}{\Sigma^\infty_G}
\\
\cK_{^\leq H}
\arrow{r}[swap]{i_L}
&
\Spectra^{\gen G}
\end{tikzcd}~. \]
Upon taking right adjoints of the two horizontal functors, we obtain another diagram
\[ \begin{tikzcd}[row sep=1.5cm, column sep=1.5cm]
\Spaces^{\gen G,^\leq H}_*
\arrow{d}[swap]{\Sigma^\infty_{G,^\leq H}}
&
\Spaces^{\gen G}_*
\arrow{d}{\Sigma^\infty_G}
\arrow{l}[swap]{\eta^*}
\\
\cK_{^\leq H}
&
\Spectra^{\gen G}
\arrow{l}{\yo}
\end{tikzcd} \]
which commutes because the canonical natural transformation comparing its two composites is an equivalence (because it is an equivalence on geometric fixedpoints for all subgroups subconjugate to $H$).

Now, let us identify $X \in \Spectra^{\gen G}_H$ with its image $X := \nu(X) \in \cK_{^\leq H}$.  Then, it remains to show that the object $i_R(X) \in \Spectra^{\gen G}$ has $i_R(X)^K \simeq 0$.  By adjunction, we have that
\[
i_R(X)^K
\simeq
\ulhom_{\Spectra^{\gen G}} \left( \Sigma^\infty_G (G/K)_+ , i_R(X) \right)
\simeq
\ulhom_{\cK_{^\leq H}} \left( \yo(\Sigma^\infty_G(G/K)_+) , X \right)
~.
\]
To show that this hom-spectrum is zero for all $X \in \cK_{^\leq H}$, it suffices to show that its underlying space is contractible for all $X \in \cK_{^\leq H}$, or equivalently that the hom-space
\[
\hom_{\Spaces^{\gen G,^\leq H}_*} ( \eta^* (G/K)_+ , \Omega^\infty_{G,^\leq H}(X))
\]
is contractible (using the evident notation).  This follows from the facts that $\Omega^\infty_{G,^\leq H}(X)$ is supported on the object $(G/H)^\circ \in \Orb_{G,^\leq H}^\op$ by assumption and that $((G/K)_+)^H \simeq \pt$ since $H$ is not subconjugate to $K$.
\end{proof}

\begin{notation}
We write
\[
\themod^{\gen G}
\in
\StLMod^\rlax_{\llax.\pos_G}
\]
for the stable left-lax left $\pos_G$-module corresponding via \Cref{thm.fracs.are.llax.modules} to the fracture
\[
\pos_G
\xra{\fF_G}
\Spectra^{\gen G}
\]
of $\Spectra^{\gen G}$ over $\pos_G$ obtained in \Cref{prop.can.fracture.of.gen.G.spt}.
\end{notation}

\begin{observation}
The value of the left-lax left $\pos_G$-module $\themod^{\gen G}$ on an object $H \in \pos_G$ is the $\infty$-category
\[
\Spectra^{\htpy \Weyl(H)}
\]
of homotopy $\Weyl(H)$-equivariant spectra, embedded in $\Spectra^{\gen G}$ as in \Cref{obs.interesting.embedding.of.htpy.Weyl.H.spt.in.gen.G.spt}.  When $G$ is abelian, its monodromy functor over a morphism $H_0 \ra H_1$ in $\pos_G$ is precisely the generalized Tate construction
\[
\Spectra^{\htpy (G/H_0)}
\xra{(-)^{\tate (H_1/H_0)}}
\Spectra^{\htpy (G/H_1)}
~.
\]
(A more elaborate formula holds in the general nonabelian case.)
\end{observation}

\begin{cor}
\label{cor.gen.G.spt.as.rlax.lim}
There is a canonical equivalence
\[
\Spectra^{\gen G}
\simeq
\lim^\rlax \left( \pos_G \overset{\llax}{\lacts} \themod^{\gen G} \right)
\]
between the $\infty$-category of genuine $G$-spectra and the right-lax limit of the stable left-lax left $\pos_G$-module $\themod^{\gen G} \in \StLMod^\rlax_{\llax.\pos_G}$.
\end{cor}

\begin{proof}
This follows by combining \Cref{thm.fracs.are.llax.modules} with \Cref{prop.can.fracture.of.gen.G.spt}.
\end{proof}

We also record a variant of \Cref{cor.gen.G.spt.as.rlax.lim} for proper-genuine $G$-spectra.

\begin{notation}
We write
\[
\pos_{^\proper G}
\in
\Conv_{\pos_G}
\]
for the convex subset of $\pos_G$ containing all but the maximal element $G \in \pos_G$, and we write
\[
\themod^{\gen^\proper G}
\in \StLMod^\rlax_{\llax.\pos_{^\proper G}}
\]
for the pulllback of
\[
\themod^{\gen G}
\in \StLMod^\rlax_{\llax.\pos_G}
\]
along the inclusion.
\end{notation}

\begin{observation}
\label{obs.pos.T.is.rcone.on.Ndiv}
In the special case that $G = \TT$, we have a canonical identification
\[
\pos_{^\proper \TT}
\cong
\Ndiv
\]
with the poset of natural numbers ordered by divisibility: the element $r \in \Ndiv$ corresponds to the proper subgroup $\Cyclic_r \leq \TT$.
\end{observation}

\begin{cor}
\label{cor.proper.gen.G.spt.as.rlax.lim}
The equivalence of \Cref{cor.gen.G.spt.as.rlax.lim} restricts to an equivalence
\[
\Spectra^{\gen^\proper G}
\simeq
\lim^\rlax \left( \pos_{^\proper G} \overset{\llax}{\lacts} \themod^{\gen^\proper G} \right)
\]
between the $\infty$-category of proper-genuine $G$-spectra and the right-lax limit of the stable left-lax left $\pos_{^\proper G}$-module $\themod^{\gen^\proper G} \in \StLMod^\rlax_{\llax.\pos_{^\proper G}}$.
\end{cor}

\begin{proof}
By \Cref{obs.fractures.restrict} (and \Cref{thm.fracs.are.llax.modules}), this restriction gives $\Spectra^{\gen G}_{\pos_{^< G}}$; since the element $\pos_{^< G} \in \Conv_{\pos_G}$ is downwards closed, this is just the full subcategory generated by $\{ \Sigma^\infty_G(G/H)_+ \}_{H \proper G}$, which is indeed $\Spectra^{\gen^\proper G}$.
\end{proof}

\subsection{A naive approach to cyclotomic spectra}
\label{subsection.naive.cyclo.spectra.for.real}

In this subsection, we apply the results of the previous subsection to prove \Cref{mainthm.cyclo.spt} (a naive description of cyclotomic spectra).

\begin{definition}
\label{define.genuine.cyclo.spectra}
The commutative monoid $\Nx$ acts on the $\infty$-category $\Spectra^{\gen^\proper \TT}$ by taking the element $r \in \Nx$ to the composite endofunctor
\begin{equation}
\label{composite.endofunctor.on.cyclonic.spectra}
\begin{tikzcd}[column sep=2cm]
\Spectra^{\gen^\proper\TT}
\arrow{r}{\Phi^{\Cyclic_r}}
&
\Spectra^{\gen^\proper(\TT/\Cyclic_r)}
\arrow{r}{(\TT/\Cyclic_r) \simeq \TT}[swap]{\sim}
&
\Spectra^{\gen^\proper\TT}
\end{tikzcd}
~,
\end{equation}
and we define its limit to be the $\infty$-category
\[
\Cyclo^\gen(\Spectra)
\]
of \bit{genuine cyclotomic spectra}.
\end{definition}

\begin{remark}
Barwick--Glasman prove in \cite{BG-cyclo} that \Cref{define.genuine.cyclo.spectra} recovers the underlying $\infty$-category of the ``model* category'' of cyclotomic spectra defined by Blumberg--Mandell in \cite{BluMan-cyclo}.
\end{remark}

\begin{construction}
We organize the action of \Cref{define.genuine.cyclo.spectra} as a \textit{right} action, which we denote by
\[
\SpgTPhi
\in
\RMod_\BN
\simeq
\coCart_\BNop
~.\footnote{This choice of handedness is motivated by our construction of the cyclotomic structure on $\THH$ in \cite{AMR-trace}.}
\]
Considering this as a cocartesian fibration over $\BNop$, we extract the full subcategory on the image of the inclusion
\[
\Spectra^{\htpy \TT}
\xlonghookra{\beta}
\Spectra^{\gen^\proper \TT}
\]
into the fiber over $\pt \in \BNop$.  This defines a locally cocartesian fibration over $\BNop$, whose cocartesian monodromy over the morphism $[1] \xra{r^\circ} \BNop$ is precisely the endofunctor
\[
\Spectra^{\htpy \TT}
\xra{(-)^{\tate \Cyclic_r}}
\Spectra^{\htpy \TT}
\]
(recall the convention established in \Cref{define.tate}); we denote the resulting object by
\[
\themod^{\Cyclo}
:=
\SphTt
\in
\RMod_{\llax.\BN}
\simeq
\loc.\coCart_{\BNop}
~.\footnote{The first of these two is only included to make contact with the discussion of \Cref{subsection.intro.cyclo.spt}; the latter will be more relevant for everything that follows.}
\]
\end{construction}

\begin{definition}
We define the $\infty$-category of \bit{cyclotomic spectra} to be the right-lax limit
\[
\Cyclo(\Spectra)
:=
\lim^\rlax \left( \SphTt \overset{\llax}{\racts} \BN \right)
~.
\]
\end{definition}

\begin{construction}
Let us consider the defining inclusion
\[ \begin{tikzcd}
\SpgTPhi
\arrow[\surjmonoleft]{rr}
\arrow{rd}
&
&
\SphTt
\arrow{ld}
\\
&
\BNop
\end{tikzcd} \]
as a morphism in $\Cat_{\loc.\cocart/\BNop} =: \LMod^\llax_{\llax.\BNop} =: \RMod^\llax_{\llax.\BN}$.  As such, it is a fiberwise right adjoint, and so by \Cref{lemma.ptwise.radjt.has.ptwise.ladjt} the fiberwise left adjoint
\[
\Spectra^{\gen^\proper \TT}
\xlongra{U}
\Spectra^{\htpy \TT}
\]
assembles into a morphism
\begin{equation}
\label{rlax.ladjt.on.Nx.modules.from.genuine.to.htpy}
\SpgTPhi
\longra
\SphTt
\end{equation}
in $\LMod^\rlax_{\llax.\BNop} =: \RMod^\rlax_{\llax.\BN}$.  Thereafter, we obtain the composite
\begin{equation}
\label{composite.from.gen.cyclo.spt.to.cyclo.spt}
\Cyclo^\gen(\Spectra)
:=
\lim \left( \SpgTPhi \right)
\xlonghookra{\ff}
\lim^\rlax \left( \SpgTPhi \right)
\longra
\lim^\rlax \left( \SphTt \right)
=:
\Cyclo(\Spectra)
\end{equation}
(where the fully faithful inclusion is that of \Cref{obs.map.from.strict.limit.to.lax.limit}).
\end{construction}

\begin{theorem}
\label{thm.in.body.cyclo.spt}
The functor \Cref{composite.from.gen.cyclo.spt.to.cyclo.spt} defines an equivalence
\[
\Cyclo^\gen(\Spectra)
\xlongra{\sim}
\Cyclo(\Spectra)
\]
of $\infty$-categories.
\end{theorem}

In order to prove \Cref{thm.in.body.cyclo.spt}, we will use the following.

\begin{observation}
\label{obs.Ndiv.mod.Nx.is.BN}
The dilation action of the commutative monoid $\Nx$ on the poset $\Ndiv$ can be presented via the cartesian fibration
\[ \begin{tikzcd}
\Ar(\BNop)
\arrow{d}{s}
\\
\BNop
\end{tikzcd}
~,
\]
whose source is then its right-lax homotopy quotient.  We would like to obtain the strict homotopy quotient of this action, for which we must freely invert the cartesian morphisms in this cartesian fibration.  But these are precisely the morphisms that are taken to equivalences under the functor
\[
\Ar(\BNop)
\xlongra{t}
\BNop
~,
\]
which admits a fully faithful right adjoint and so is precisely the desired localization.  In other words, we have obtained an identification
\[
\Ndiv
\longra
(\Ndiv)_{\htpy \Nx}
\simeq
\BNop
\]
of the strict quotient of the $\Nx$-action on $\Ndiv$ as the category $\BNop$.
\end{observation}

\begin{proof}[Proof of \Cref{thm.in.body.cyclo.spt}]
Let us pull back the morphism \Cref{rlax.ladjt.on.Nx.modules.from.genuine.to.htpy} in $\LMod^\rlax_{\llax.\BNop}$ along the homotopy quotient
\[
\Ndiv
\xra{\pi}
(\Ndiv)_{\htpy \Nx}
\simeq
\BNop
\]
of \Cref{obs.Ndiv.mod.Nx.is.BN} to obtain a morphism
\[
\pi^* \SpgTPhi
\longra
\pi^* \SphTt
\]
in $\LMod^\rlax_{\llax.\Ndiv}$.  By virtue of this construction, we can view the composite \Cref{composite.from.gen.cyclo.spt.to.cyclo.spt} as the homotopy quotient by $\Nx$ of an $\Nx$-equivariant composite
\begin{equation}
\label{Nx.equivariant.composite.showing.gen.cyclo.spt.is.cyclo.spt}
\lim \left( \Ndiv \lacts \pi^* \SpgTPhi \right)
\xlonghookra{\ff}
\lim^\rlax \left( \Ndiv \lacts \pi^* \SpgTPhi \right)
\longra
\lim^\rlax \left( \Ndiv \overset{\llax}{\lacts} \pi^* \SphTt \right)
~.
\end{equation}
Now, observe that evaluation at the initial object $1 \in \Ndiv$ defines an equivalence
\[
\lim \left( \Ndiv \lacts \pi^* \SpgTPhi \right)
\xra[\sim]{\ev_1}
\Spectra^{\gen^\proper \TT}
~.
\]
On the other hand, unwinding the definitions, we find an equivalence
\[
\pi^* \SphTt
\simeq
\themod^{\gen^\proper \TT}
~,
\]
and moreover that under these identifications the $\Nx$-equivariant composite \Cref{Nx.equivariant.composite.showing.gen.cyclo.spt.is.cyclo.spt} is precisely the canonical equivalence
\[
\Spectra^{\gen^\proper \TT}
\xlongra{\sim}
\lim^\rlax \left( \Ndiv \overset{\llax}{\lacts} \themod^{\gen^\proper \TT} \right)
\]
of \Cref{cor.proper.gen.G.spt.as.rlax.lim}.  This proves the claim.
\end{proof}

\section{The formula for $\TC$}
\label{section.proof.of.TC.formula}

In this section, we derive the explicit formula for $\TC$ of \Cref{mainthm.formula.for.TC} -- in other words, an explicit formula for the right adjoint in a certain adjunction
\[ \begin{tikzcd}[column sep=2cm, row sep=0cm]
\Spectra
\arrow[transform canvas={yshift=0.9ex}]{r}{\triv}
\arrow[leftarrow, transform canvas={yshift=-0.9ex}]{r}[yshift=-0.2ex]{\bot}[swap]{(-)^{\htpy \Cyclo}}
&
\Cyclo(\Spectra)
\\
\rotatebox{90}{$\in$}
&
\rotatebox{90}{$\in$}
\\
\TC
&
\THH
\arrow[maps to]{l}
\end{tikzcd} \]
relating spectra and cyclotomic spectra.  Here, the left adjoint is the \textit{trivial cyclotomic spectrum} functor, which we obtain in \Cref{subsection.trivial.cyclo.spt}.

This requires auxiliary input.  First of all, the construction of the functor $\triv$ makes crucial use of the $\infty$-category $\Cyclo^\htpy(\Spectra)$ of \textit{cyclotomic spectra with Frobenius lifts}, which we examine in \Cref{subsection.cyclo.spt.w.frob}; these are cyclotomic spectra whose cyclotomic structure maps are equipped with compatible factorizations
\[ \begin{tikzcd}[row sep=1.5cm, column sep=1.5cm]
T
\arrow[dashed]{r}{\tilde{\sigma}_r}
\arrow{rd}[swap]{\sigma_r}
&
T^{\htpy \Cyclic_r}
\arrow{r}
\arrow{d}
&
T
\\
&
T^{\tate \Cyclic_r}
\end{tikzcd}~. \]
In turn, the construction of this $\infty$-category along with its forgetful functor
\[
\Cyclo^\htpy(\Spectra)
\longra
\Cyclo(\Spectra)
\]
relies on a sophisticated understanding of the generalized Tate construction; this is provided by the \textit{proto Tate package}, which is the subject of \Cref{subsection.proto.tate}.

Finally, we prove \Cref{mainthm.formula.for.TC} (as \Cref{thm.formula.for.cyclo.fixedpts}) in \Cref{subsection.formula.for.TC.for.real}, using the formalism of \textit{partial adjunctions} that we introduce in \Cref{subsection.partial.adjns}.

\subsection{The proto Tate package}
\label{subsection.proto.tate}

In this subsection, we establish the \bit{proto Tate package} (\Cref{proto.tate}); this encapsulates the simultaneous and interwoven functorialities of the generalized Tate construction, as indicated in \Cref{remark.sell.proto.tate.package}.\footnote{We reserve the name ``Tate package'' for the resulting object that induces the cyclotomic structure on $\THH$ (see \cite{AMR-trace}).}

\begin{notation}
\label{notation.G.is.all.gps}
In this subsection, we use the letter
\[
G
\]
interchangeably to refer either to a specific (though never actually specified) finite group or to the set of all finite groups at once.
\end{notation}

\begin{notation}
We write $\Span(\Cat)$ for the $\infty$-category whose objects are $\infty$-categories, whose morphisms are spans of $\infty$-categories, and whose composition is given by pullback.\footnote{This can be easily constructed e.g.\! as a complete Segal space, using the fact that $\Cat$ admits finite limits.}  Then, we write
\[
\GSpan
\longrsurjmono
\Span(\Cat)
\]
for the surjective monomorphism from the subcategory whose morphisms are those spans
\begin{equation}
\label{morphism.in.GSpan}
\begin{tikzcd}
&
\cB_{01}
\arrow{ld}
\arrow{rd}[pos=0.4]{\BG\dKan}
\\
\cB_0
&
&
\cB_1
\end{tikzcd}
\end{equation}
whose forward morphism is a Kan fibration (i.e.\! a left and right fibration) with fiber a connected $\pi$-finite 1-type, i.e.\! a space of the form $\BG$ for some finite group $G$.\footnote{To be precise (in the face of the potential ambiguity presented by \Cref{notation.G.is.all.gps}), we clarify that for simplicity and because it is the only case we need to consider, we will assume that all the fibers are abstractly equivalent, even if $(\cB_1)^\gpd$ is disconnected.}\footnote{Such morphisms are clearly closed under composition in $\Span(\Cat)$, since connected $\pi$-finite 1-types are closed under fiber sequences (namely, they loop down to extension sequences of groups).}
\end{notation}

\needspace{2\baselineskip}
\begin{theorem}[The proto Tate package]
\label{proto.tate}
\begin{enumerate}
\item[]
\item There exists a canonical left $\GSpan$-module
\[ \left(
\begin{tikzcd}
\Fh
\arrow{d}
\\
\GSpan
\end{tikzcd}
\right)
\in
\LMod_\GSpan := \coCart_\GSpan
 \]
whose fiber over an object $\cB \in \GSpan$ is the $\infty$-category
\[
\Fun(\cB,\Spectra)
\]
and whose cocartesian monodromy over a morphism \Cref{morphism.in.GSpan} is the composite
\[
\Fun(\cB_0,\Spectra)
\longra
\Fun(\cB_{01},\Spectra)
\xra{(-)^{\htpy G}}
\Fun(\cB_1,\Spectra)
\]
of pullback followed by fiberwise homotopy fixedpoints.

\item There exists a canonical left-lax left $\GSpan$-module
\[
\left(
\begin{tikzcd}
\Ftau
\arrow{d}
\\
\GSpan
\end{tikzcd}
\right)
\in
\LMod_{\llax.\GSpan} := \loc.\coCart_\GSpan
\]
whose fiber over an object $\cB \in \GSpan$ is the $\infty$-category
\[
\Fun(\cB,\Spectra)
\]
and whose cocartesian monodromy over a morphism \Cref{morphism.in.GSpan} is the composite
\[
\Fun(\cB_0,\Spectra)
\longra
\Fun(\cB_{01},\Spectra)
\xra{(-)^{\tate G}}
\Fun(\cB_1,\Spectra)
\]
of pullback followed by the fiberwise Tate construction, and whose left-lax structure maps encode its lax functoriality for extensions among finite groups.

\item There exists a right-lax equivariant functor
\[
\Fh
\xra{\htpy \ra \tate}
\Ftau
\]
of left-lax left $\GSpan$-modules -- that is, a morphism in $\LMod^\rlax_{\llax.\GSpan}$ -- which is an equivalence on fibers and which over a morphism \Cref{morphism.in.GSpan} is given fiberwise by the natural transformation
\[
(-)^{\htpy G}
\longra
(-)^{\tate G}
\]
in $\Fun(\Spectra^{\htpy G},\Spectra)$.

\end{enumerate}
\end{theorem}

\begin{remark}
In fact, we do not need the full generality of the proto Tate package in this paper: we only use its restriction to the subcategory of $\GSpan$ on the forwards maps, i.e.\! the subcategory of $\Cat$ on the $\BG$-Kan fibrations (see Constructions \ref{construction.functor.to.GSpan} \and \ref{constrn.actually.use.proto.tate}).  However, we use its full strength in \cite{AMR-trace} to construct the \textit{Tate package}, which induces the cyclotomic structure on $\THH$.\footnote{The proof of the full proto Tate package is not much harder than that of its ``forwards maps only'' version -- the backwards maps are easy to manage, as they simply correspond to pullbacks -- which is why we simply prove the entire proto Tate package here.}
\end{remark}

The proof of \Cref{proto.tate} will require some preliminaries.

\begin{observation}
\label{obs.genuine.G.spt.are.natural.in.BG}
The constructions of the $\infty$-categories of homotopy and genuine $G$-spectra can be made ``coordinate-free'', in the following sense.  Given a space $\cB$ which admits an equivalence $\cB \simeq \BG$, we can form the category
\[
\cB\Fin
:=
\Fun(\cB,\Fin)
\]
of ``finite $\cB$-sets'', i.e.\! functors from $\cB$ to the category of finite sets (or equivalently, the category of Kan fibrations over $\cB$ whose fibers lie in $\Fin \subset \Spaces$, a full subcategory of $\Spaces_{/\cB}$).  Note that the Yoneda embedding defines a fully faithful inclusion
\begin{equation}
\label{include.B.into.finite.B.sets}
\cB^\op
\xlonghookra{\ff}
\cB\Fin
~.
\end{equation}
This admits an enhancement to the Burnside $\infty$-category (in fact $(2,1)$-category)
\begin{equation}
\label{include.finite.B.sets.into.its.burnside.cat}
\cB\Fin^\op
\longrsurjmono
\Burn(\cB\Fin)
:=
\Span(\cB\Fin)
\end{equation}
with the same objects, whose morphisms are given by spans and whose composition is given by pullback.  By work of Barwick \cite{Bar-Mack} (based on work of Guillou--May \cite{GM-gen}), we can then define the $\infty$-category of ``genuine $\cB$-spectra'' as
\[
\Spectra^{\gen \cB}
:=
\Fun^\oplus ( \Burn(\cB\Fin) , \Spectra)
~,
\]
the $\infty$-category of \textit{spectral Mackey functors} for $\cB\Fin$ (i.e.\! additive functors from its Burnside $\infty$-category to the $\infty$-category of spectra).  Restriction along the composite of the functors \Cref{include.B.into.finite.B.sets,include.finite.B.sets.into.its.burnside.cat} induces the forgetful functor
\[
\Spectra^{\gen \cB}
\xlongra{U}
\Fun(\cB,\Spectra)
~,
\]
a ``coordinate-free'' version of the forgetful functor
\[
\Spectra^{\gen G}
\xlongra{U}
\Spectra^{\htpy G} := \Fun(\BG,\Spectra)
\]
from genuine $G$-spectra to homotopy $G$-spectra.
\end{observation}

\begin{notation}
\label{notation.for.param.gen.and.naive.spt}
Let
\[
\left(
\cE
\da
\cB
\right)
\in
\Cat_{/\cB}
\]
be a $\BG$-Kan fibration.  Informally speaking, we use the notation
\[
\cE | \cB
\]
to denote the family of copies of $\BG$ determined by this functor; more precisely, we write
\[
\Spectra^{\htpy \cE|\cB}
:=
\Fun^\rel_{/\cB}(\cE , \Spectra)
\]
for the $\infty$-category of ``homotopy $\cE|\cB$-spectra'', and we write
\[
\Spectra^{\gen \cE|\cB}
\]
for the $\infty$-category of ``genuine $\cE|\cB$-spectra'', obtained by applying the construction of \Cref{obs.genuine.G.spt.are.natural.in.BG} fiberwise (or equivalently, to the image of the straightening $\cB \ra \Cat$, which by assumption lands in the full subgroupoid on the object $\BG \in \Spaces \subset \Cat$).  If $\cB \simeq \pt$ then we may omit it from the notation; in other words, if there exists an equivalence $\cE \simeq \BG$ then (in a slight abuse of notation) we may simply write
\[
\Spectra^{\htpy \cE} := \Spectra^{\htpy \cE|\pt}
\qquad
\textup{and}
\qquad
\Spectra^{\gen \cE} := \Spectra^{\gen \cE|\pt}
~.
\]
\end{notation}

\begin{observation}
The $\infty$-categories of \Cref{notation.for.param.gen.and.naive.spt} participate in an adjunction
\[ \begin{tikzcd}[column sep=1.5cm]
\Spectra^{\gen \cE|\cB}
\arrow[transform canvas={yshift=0.9ex}]{r}{U}
\arrow[hookleftarrow, transform canvas={yshift=-0.9ex}]{r}[yshift=-0.2ex]{\bot}[swap]{\beta}
&
\Spectra^{\htpy \cE|\cB}
\end{tikzcd} \]
in $\Cat_{/\cB}$, which on fibers over each object $b \in \cB$ recovers the adjunction
\[ \begin{tikzcd}[column sep=1.5cm]
\Spectra^{\gen \cE_{|b}}
\arrow[transform canvas={yshift=0.9ex}]{r}{U}
\arrow[hookleftarrow, transform canvas={yshift=-0.9ex}]{r}[yshift=-0.2ex]{\bot}[swap]{\beta}
&
\Spectra^{\htpy \cE_{|b}}
\end{tikzcd}~. \]
Note that we can identify the global sections of the $\cB$-parametrized $\infty$-category of homotopy $\cE|\cB$-spectra as
\[
\Gamma \left( \Spectra^{\htpy \cE|\cB} \right)
:=
\Fun_{/\cB} \left( \cB , \Spectra^{\htpy \cE|\cB} \right)
\simeq
\Fun(\cE,\Spectra)
~,
\]
and that we simply have
\[
\Spectra^{\htpy \cB|\cB}
\simeq
\Spectra^{\gen \cB | \cB}
\simeq
\Fun(\cB,\Spectra)
\]
(since in this case the group is trivial).  On the other hand, an object of the $\infty$-category of global sections
\[
\Gamma \left( \Spectra^{\gen \cE | \cB} \right)
\]
of the $\cB$-parametrized $\infty$-category of genuine $\cE|\cB$-spectra assigns a ``genuine $\cE_{|b}$-spectrum'' naturally to each object $b \in \cB$.
\end{observation}

\begin{observation}
Let
\[
K
\longhookra
G
\longrsurj
Q
\]
be a short exact sequence of finite groups, and suppose we are given a functor
\[ \begin{tikzcd}
\cE_0
\arrow{rr}
\arrow{rd}[swap]{\BG\dKan}
&
&
\cE_1
\arrow{ld}{\sB Q\dKan}
\\
&
\cB
\end{tikzcd} \]
over $\cB$ which is given in each fiber by a functor
\[
\BG
\longra
\sB Q
\]
corresponding to the quotient homomorphism.  Then, the $\cB$-parametrized $\infty$-category $\Spectra^{\gen \cE_0|\cB} \da \cB$ has fibers all equivalent to $\Spectra^{\gen G}$; these admit genuine $K$-fixedpoints functors, which assemble into a functor
\[ \begin{tikzcd}
\Spectra^{\gen \cE_0|\cB}
\arrow{rr}{(-)^K}
\arrow{rd}
&
&
\Spectra^{\gen \cE_1|\cB}
\arrow{ld}
\\
&
\cB
\end{tikzcd} \]
over $\cB$.
\end{observation}

\begin{proof}[Proof of \Cref{proto.tate}]
We work in the straightened context: we define functors
\[
\llax(\GSpan)
\xra{\Fh}
\Cat
\qquad
\textup{and}
\qquad
\llax(\GSpan)
\xra{\Ftau}
\Cat
\]
(the former of which factors through the localization $\llax(\GSpan) \ra \GSpan$), as well as a right-lax natural transformation between them.\footnote{This is in essence only a cosmetic simplification, which reduces us to describing our various fibrational structures only on cocartesian morphisms over $\GSpan$: it is a straightforward exercise to translate the constructions we give here into the fibrational language of \Cref{section.lax.actions.and.limits}.}  By \Cref{left.laxification.is.left.kan.extended}, it suffices to work at the level of $[n]$-points of $\GSpan$.  Each functor $[n] \ra \GSpan$ will determine a certain commutative diagram of $\infty$-categories, and the corresponding composite functors
\[
\llax([n])
\longra
\llax(\GSpan)
\xra{\Fh}
\Cat
\qquad
\textup{and}
\qquad
\llax([n])
\longra
\llax(\GSpan)
\xra{\Ftau}
\Cat
\]
as well as the right-lax natural transformation between them will all be defined in terms of this commutative diagram.  We describe the situation for $n \leq 2$: the general case is no different, but the combinatorics are more cumbersome.

We begin by describing these commutative diagrams.
\begin{enumerate}

\setcounter{enumi}{-1}

\item A $[0]$-point of $\GSpan$ is the data of an $\infty$-category $\cB_0$.  This is associated to the (tautologically commutative) diagram
\begin{equation}
\label{diagram.for.brax.zero.points.of.GSpan}
\Fun(\cB_0,\Spectra)
\end{equation}
of $\infty$-categories.

\item A $[1]$-point of $\GSpan$ in the data of a span \Cref{morphism.in.GSpan}.  This is associated to the (tautologically commutative) diagram
\begin{equation}
\label{diagram.for.brax.one.points.of.GSpan}
\begin{tikzcd}[row sep=2cm]
&
\Gamma \left( \Spectra^{\htpy \cB_{01}|\cB_1} \right)
\arrow[hook]{r}{\beta_G}
&
\Gamma \left( \Spectra^{\gen \cB_{01}|\cB_1} \right)
\arrow{rd}{(-)^G}
\\
\Fun(\cB_0 , \Spectra)
\arrow{ru}
&
&
&
\Fun(\cB_1,\Spectra)
\end{tikzcd}
\end{equation}
of $\infty$-categories.  Note that the composite of the latter two functors is the (fiberwise) homotopy fixedpoints functor
\[
(-)^{\htpy G}
:
\Gamma \left( \Spectra^{\htpy \cB_{01}|\cB_1} \right)
\xra{\beta_G}
\Gamma \left( \Spectra^{\gen \cB_{01}|\cB_1} \right)
\xra{(-)^G}
\Fun(\cB_1,\Spectra)
~.
\]

\item A $[2]$-point of $\GSpan$ is the data of a diagram
\begin{equation}
\label{brax.two.point.of.GSpan}
\begin{tikzcd}
&
&
\cB_{02}
\arrow{ld}
\arrow{rd}{\sB K\dKan}
\\
&
\cB_{01}
\arrow{ld}
\arrow{rd}{\sB K\dKan}
&
&
\cB_{12}
\arrow{ld}
\arrow{rd}{\sB Q\dKan}
\\
\cB_0
&
&
\cB_1
&
&
\cB_2
\end{tikzcd}
\end{equation}
of $\infty$-categories with fibrational properties as indicated and in which the square is a pullback; the composite forwards map $\cB_{02} \ra \cB_{12} \ra \cB_2$ is then a $\BG$-Kan fibration, for some short exact sequence
\[
K
\longhookra
G
\longrsurj
Q
\]
of finite groups (incarnated as a fiber sequence
\[
\sB K
\longra
\BG
\longra
\sB Q
\]
of groupoids).  This is associated to the diagram of $\infty$-categories depicted in \Cref{comm.diagram.for.brax.two.point.of.GSpan}, which commutes since the diagram
\[ \begin{tikzcd}[column sep=0cm]
&
\Gamma \left( \Spectra^{\htpy \cB_{02} | \cB_2} \right)
\arrow{rd}
\\
\Gamma \left( \Spectra^{\htpy \cB_{01} | \cB_1} \right)
\arrow{ru}
\arrow{rd}
&
&
\Gamma \left( \Spectra^{\htpy \cB_{12} | \cB_2} \right)
\\
&
\Fun(\cB_1,\Spectra)
\arrow{ru}
\end{tikzcd} \]
commutes, where both downwards functors are given by right Kan extension: these right Kan extensions are both along $\sB K$-Kan fibrations, so are given by fiberwise homotopy $K$-fixedpoints.
\begin{sidewaysfigure}
\vspace{400pt}
\begin{equation}
\label{diagram.for.brax.two.points.of.GSpan}
\hspace{-50pt}
\begin{tikzcd}[row sep=2cm]
&
&
\Gamma \left( \Spectra^{\htpy \cB_{02}|\cB_2} \right)
\arrow[hook]{r}{\beta_G}
&
\Gamma \left( \Spectra^{\gen \cB_{02}|\cB_2} \right)
\arrow{rrd}{(-)^K}
\\
&
\Gamma \left( \Spectra^{\htpy \cB_{01}|\cB_1} \right)
\arrow[hook]{r}{\beta_K}
\arrow{ru}
&
\Gamma \left( \Spectra^{\gen \cB_{01}|\cB_1} \right)
\arrow{rd}{(-)^K}
&
&
\Gamma \left( \Spectra^{\htpy \cB_{12}|\cB_2} \right)
\arrow[hook]{r}{\beta_Q}
&
\Gamma \left( \Spectra^{\gen \cB_{12} | \cB_2} \right)
\arrow{rd}{(-)^Q}
\\
\Fun(\cB_0 , \Spectra)
\arrow{ru}
&
&
&
\Fun(\cB_1,\Spectra)
\arrow{ru}
&
&
&
\Fun(\cB_2,\Spectra)
\end{tikzcd}
\end{equation}
\vspace{2.5cm}
\caption{The commutative diagram of $\infty$-categories corresponding to the $[2]$-point \Cref{brax.two.point.of.GSpan} of $\GSpan$.}
\label{comm.diagram.for.brax.two.point.of.GSpan}
\end{sidewaysfigure}
\end{enumerate}

We now construct the various asserted data, cycling through those associated to the three asserted objects in turn: the left $\GSpan$-module $\Fh$, the left-lax left $\GSpan$-module $\Ftau$, and the right-lax $\GSpan$-equivariant functor from the former to the latter.  The various simplicial structure maps that relate these data will be evident from their constructions.

\begin{enumerate}

\setcounter{enumi}{-1}

\item [$(0)_\htpy$]

The functor $\Fh$ takes the $[0]$-point of $\GSpan$ corresponding to the $\infty$-category $\cB_0$ to the functor
\[
[0]
\longra
\Cat
\]
selecting the $\infty$-category $\Fun(\cB_0,\Spectra)$, i.e.\! the unique $\infty$-category in the diagram \Cref{diagram.for.brax.zero.points.of.GSpan}.

\item [$(0)_\tate$]

The functor $\Ftau$ takes the $[0]$-point of $\GSpan$ corresponding to the $\infty$-category $\cB_0$ to the functor
\[
[0]
\simeq
\llax([0])
\longra
\Cat
\]
selecting the $\infty$-category $\Fun(\cB_0,\Spectra)$, i.e.\! the unique $\infty$-category in the diagram \Cref{diagram.for.brax.zero.points.of.GSpan}.

\item[$(0)_{\htpy \ra \tate}$]

The right-lax natural transformation
\[
\Fh
\xra{\htpy \ra \tate}
\Ftau
\]
takes the $[0]$-point of $\GSpan$ corresponding to the $\infty$-category $\cB_0$ to the morphism in the $\infty$-category
\[
\Fun^\rlax(\llax([0]),\Cat)
\simeq
\Fun^\rlax([0],\Cat)
\simeq
\Cat_{\cart/[0]^\op}
\simeq
\Cat
\]
given by the identity functor on the $\infty$-category $\Fun(\cB_0,\Spectra)$.

\item [$(1)_\htpy$]

The functor $\Fh$ takes the $[1]$-point of $\GSpan$ corresponding to the span \Cref{morphism.in.GSpan} to the functor
\[
[1]
\longra
\Cat
\]
selecting composite functor \Cref{diagram.for.brax.one.points.of.GSpan}.  However, we will describe this in another way.  Namely, the diagram \Cref{diagram.for.brax.one.points.of.GSpan} classifies a cocartesian fibration over its indexing diagram (namely the category $[3]$), and each object of $\Fun(\cB_0,\Spectra)$ determines a unique cocartesian section thereof, whose value at the object $3 \in [3]$ selects its target in $\Fun(\cB_1,\Spectra)$.  We depict a typical such a cocartesian section by the diagram
\[ \begin{tikzcd}[row sep=1.25cm]
&
Y
\arrow[rightsquigarrow]{r}
&
\beta_G Y
\arrow[rightsquigarrow]{rd}
\\
X
\arrow[rightsquigarrow]{ru}
&
&
&
Y^{\htpy G}
\end{tikzcd}~, \]
in which the squiggly arrows denote cocartesian morphisms.  Throughout the proof, we will use the convention that squiggly arrows denote cocartesian morphisms.  We will also continue to use the convention that we simply write $Y$ for the pullback of $X$ (and similarly for $Z$ below), rather than introducing more notation.

\item [$(1)_\tate$]

Given the $[1]$-point of $\GSpan$ corresponding to the span \Cref{morphism.in.GSpan}, observe that each object $X \in \Fun(\cB_0,\Spectra)$ determines a unique diagram
\begin{equation}
\label{limit.point.for.F.tau.on.brax.one.point}
\begin{tikzcd}[row sep=1.25cm]
&
Y
\arrow[rightsquigarrow]{r}
&
\beta_G Y
\arrow{d}[description]{\locL_{\Phi^G}}
\\
&
&
\locL_{\Phi^G} \beta_G Y
\arrow[rightsquigarrow]{rd}
\\
X
\arrow[rightsquigarrow]{ruu}
&
&
&
Y^{\tate G}
\end{tikzcd}
\end{equation}
in the cocartesian fibration over $[3]$ classified by the diagram \Cref{diagram.for.brax.one.points.of.GSpan} in which the straight vertical arrow, a morphism in
\[
\Gamma \left( \Spectra^{\gen \cB_{01} | \cB_1} \right)~,
\]
is given in each fiber (a copy of $\Spectra^{\gen G}$) by the indicated localization (namely that of \Cref{obs.genzd.tate}, whose genuine fixedpoints compute geometric fixedpoints (leaving the fully faithful inclusion implicit)).  Throughout the proof, we will use the convention that straight arrows interrupted by labels denote morphisms which are required to be given fiberwise by the indicated localization.  Then, using the notation in diagram \Cref{limit.point.for.F.tau.on.brax.one.point}, we declare that the composite functor
\[
[1]
\simeq
\llax([1])
\longra
\llax(\GSpan)
\xra{\Ftau}
\Cat
\]
selects the functor
\[ \begin{tikzcd}[row sep=0cm]
\Fun(\cB_0,\Spectra)
\arrow{r}
&
\Fun(\cB_1,\Spectra)
\\
\rotatebox{90}{$\in$}
&
\rotatebox{90}{$\in$}
\\
X
\arrow[maps to]{r}
&
Y^{\tate G}
\end{tikzcd}~. \]

\item [$(1)_{\htpy \ra \tate}$]

Given the $[1]$-point of $\GSpan$ corresponding to the span \Cref{morphism.in.GSpan}, observe that each object $X \in \Fun(\cB_0,\Spectra)$ determines a unique commutative diagram in the cocartesian fibration over $[3]$ classified by diagram \Cref{diagram.for.brax.one.points.of.GSpan} of the form
\begin{equation}
\label{limit.point.for.h.to.tau.on.brax.one.point}
\begin{tikzcd}[row sep=1.25cm]
&
Y
\arrow[rightsquigarrow]{r}
&
\beta_G Y
\arrow{d}[description]{\locL_{\Phi^G}}
\arrow[rightsquigarrow]{rd}
\\
&
&
\locL_{\Phi^G} \beta_G Y
\arrow[rightsquigarrow]{rd}
&
Y^{\htpy G}
\arrow{d}
\\
X
\arrow[rightsquigarrow]{ruu}
&
&
&
Y^{\tate G}
\end{tikzcd}~.
\end{equation}
(Since the square commutes, the lower straight vertical morphism must be given fiberwise by the canonical one of \Cref{obs.genzd.tate}.)  Then, we declare that the right-lax natural transformation
\[
\Fh
\xra{\htpy \ra \tate}
\Ftau
\]
takes this $[1]$-point of $\GSpan$ to the morphism in the $\infty$-category
\[
\Fun^\rlax(\llax([1]),\Cat)
\simeq
\Fun^\rlax([1],\Cat)
\simeq
\Cat_{\cart/[1]^\op}
\]
corresponding to the lax-commutative square
\[ \begin{tikzcd}[row sep=1.5cm, column sep=1.5cm]
\Fun(\cB_0,\Spectra)
\arrow{r}{X \mapsto Y^{\htpy G}}[swap, transform canvas={yshift=-0.75cm}]{\rotatebox{-135}{$\Rightarrow$}}
\arrow{d}[sloped, anchor=north]{\sim}
&
\Fun(\cB_1,\Spectra)
\arrow{d}[sloped, anchor=south]{\sim}
\\
\Fun(\cB_0,\Spectra)
\arrow{r}[swap]{X \mapsto Y^{\tate G}}
&
\Fun(\cB_1,\Spectra)
\end{tikzcd} \]
in $\Cat$ determined by the (natural) lower straight vertical morphism in diagram \Cref{limit.point.for.h.to.tau.on.brax.one.point}.

\item [$(2)_\htpy$]

Given the $[2]$-point of $\GSpan$ corresponding to the diagram \Cref{brax.two.point.of.GSpan}, observe that each object $X \in \Fun(\cB_0,\Spectra)$ determines a unique commutative diagram
\begin{equation}
\label{limit.point.for.F.h.on.brax.two.point}
\begin{tikzcd}[row sep=1.25cm]
&
&
Z
\arrow[rightsquigarrow]{r}
&
\beta_G Z
\arrow[rightsquigarrow]{rrd}
\\
&
Y
\arrow[rightsquigarrow]{ru}
\arrow[rightsquigarrow]{r}
&
\beta_K Y
\arrow[rightsquigarrow]{rd}
&
&
Z^{\htpy K}
\arrow[bend left, crossing over, rightsquigarrow]{rr}
&
(\beta_G Z)^K
\arrow{r}{\sim}
\arrow[rightsquigarrow]{rd}
&
\beta_Q \left( Z^{\htpy K} \right)
\arrow[rightsquigarrow]{rd}
\\
X
\arrow[rightsquigarrow]{ru}
&
&
&
Y^{\htpy K}
\arrow[rightsquigarrow]{ru}
&
&
&
Z^{\htpy G}
\arrow{r}{\sim}
&
\left( Z^{\htpy K} \right)^{\htpy Q}
\end{tikzcd}
\end{equation}
in the cocartesian fibration classified by diagram \Cref{diagram.for.brax.two.points.of.GSpan}.  Then, we declare that the composite functor
\[
[2]
\longra
\GSpan
\xra{\Fh}
\Cat
\]
selects the diagram
\[ \begin{tikzcd}[row sep=2cm]
&
\Fun(\cB_1,\Spectra)
\arrow{rd}[sloped, pos=0.3]{Y \mapsto Z^{\htpy Q}}
\\
\Fun(\cB_0,\Spectra)
\arrow{ru}[sloped, pos=0.7]{X \mapsto Y^{\htpy K}}
\arrow{rr}[swap]{X \mapsto Z^{\htpy G}}
&
&
\Fun(\cB_2,\Spectra)
\end{tikzcd}~, \]
which commutes via the data of the (natural) lower horizontal equivalence in diagram \Cref{limit.point.for.F.h.on.brax.two.point}.

\item [$(2)_\tate$]

Given the $[2]$-point of $\GSpan$ corresponding to the diagram \Cref{brax.two.point.of.GSpan}, observe that each object $X \in \Fun(\cB_0,\Spectra)$ determines a unique commutative diagram
\begin{equation}
\label{limit.point.for.F.tau.on.brax.two.point}
\begin{tikzcd}[row sep=1.25cm]
&
&
Z
\arrow[rightsquigarrow]{r}
&
\beta_G Z
\arrow{d}
\\
&
&
&
\locL_{\Phi^K} \beta_G Z
\arrow{d}
\arrow[rightsquigarrow]{rrd}
\\
&
Y
\arrow[rightsquigarrow]{r}
\arrow[rightsquigarrow]{ruu}
&
\beta_K Y
\arrow{d}
&
\locL_{\Phi^G} \beta_G Z
\arrow[rightsquigarrow]{rrd}
&
Z^{\tate K}
\arrow[crossing over, bend left, rightsquigarrow]{rr}
&
\Phi^K \beta_G Z
\arrow{d}[description]{\locL_{\Phi^Q}}
\arrow{r}[description]{\beta_Q U_Q}
&
\beta_Q Z^{\tate K}
\arrow{d}[description]{\locL_{\Phi^Q}}
\\
&
&
\locL_{\Phi^K} \beta_K Y
\arrow[rightsquigarrow]{rd}
&
&
&
L_{\Phi^Q} \Phi^K \beta_G Z
\arrow{r}
\arrow[rightsquigarrow]{rd}
&
L_{\Phi^Q} \beta_Q Z^{\tate K}
\arrow[rightsquigarrow]{rd}
\\
X
\arrow[rightsquigarrow]{ruu}
&
&
&
Y^{\tate K}
\arrow[rightsquigarrow, crossing over]{ruu}
&
&
&
Z^{\tate G}
\arrow{r}
&
\left( Z^{\tate K} \right)^{\tate Q}
\end{tikzcd}
\end{equation}
in the cocartesian fibration classified by diagram \Cref{diagram.for.brax.two.points.of.GSpan}.  Then, we declare that the composite functor
\[
\llax([2])
\longra
\llax(\GSpan)
\xra{\Ftau}
\Cat
\]
selects the diagram
\[ \begin{tikzcd}[row sep=2cm]
&
\Fun(\cB_1,\Spectra)
\arrow{rd}[sloped, pos=0.3]{Y \mapsto Z^{\tate Q}}
\\
\Fun(\cB_0,\Spectra)
\arrow{ru}[sloped, pos=0.7]{X \mapsto Y^{\tate K}}
\arrow{rr}[transform canvas={yshift=0.8cm}]{\rotatebox{90}{$\Rightarrow$}}[swap]{X \mapsto Z^{\tate G}}
&
&
\Fun(\cB_2,\Spectra)
\end{tikzcd}~, \]
which lax-commutes via the data of the (natural) bottom horizontal morphism in diagram \Cref{limit.point.for.F.tau.on.brax.two.point}.

\item [$(2)_{\htpy \ra \tate}$]

Given the $[2]$-point of $\GSpan$ corresponding to the diagram \Cref{brax.two.point.of.GSpan}, observe that each object $X \in \Fun(\cB_0,\Spectra)$ determines a unique commutative diagram
\begin{equation}
\label{limit.point.for.h.to.tau.on.brax.two.point}
\begin{tikzcd}[row sep=1.25cm]
&
&
Z
\arrow[rightsquigarrow]{r}
&
\beta_G Z
\arrow{d}
\arrow[rightsquigarrow]{rrd}
\\
&
&
&
\locL_{\Phi^K} \beta_G Z
\arrow{d}
\arrow[rightsquigarrow]{rrd}
&
Z^{\htpy K}
\arrow[bend left, crossing over, rightsquigarrow]{rr}
&
(\beta_Q Z)^K
\arrow{r}{\sim}
\arrow{d}
\arrow[rightsquigarrow]{rddd}
&
\beta_Q \left( Z^{\htpy K} \right)
\arrow{d}
\arrow[rightsquigarrow]{rddd}
\\
&
Y
\arrow[rightsquigarrow]{r}
\arrow[rightsquigarrow]{ruu}
&
\beta_K Y
\arrow{d}
\arrow[rightsquigarrow]{rdd}
&
\locL_{\Phi^G} \beta_G Z
\arrow[rightsquigarrow]{rrd}
&
Z^{\tate K}
\arrow[crossing over, bend left, rightsquigarrow]{rr}
\arrow[crossing over, leftarrow]{u}
&
\Phi^K \beta_G Z
\arrow{d}[description]{\locL_{\Phi^Q}}
\arrow{r}[description]{\beta_Q U_Q}
&
\beta_Q Z^{\tate K}
\arrow{d}[description]{\locL_{\Phi^Q}}
\\
&
&
\locL_{\Phi^K} \beta_K Y
\arrow[rightsquigarrow]{rdd}
&
&
&
L_{\Phi^Q} \Phi^K \beta_G Z
\arrow{r}
\arrow[rightsquigarrow]{rdd}
&
L_{\Phi^Q} \beta_Q Z^{\tate K}
\arrow[rightsquigarrow]{rdd}
\\
&
&
&
Y^{\htpy K}
\arrow[crossing over, rightsquigarrow, bend left=6]{ruuu}
&
&
&
Z^{\htpy G}
\arrow[crossing over]{r}[swap]{\sim}
\arrow{d}
&
\left( Z^{\htpy K} \right)^{\htpy Q}
\arrow{d}
\\
X
\arrow[rightsquigarrow]{ruuu}
&
&
&
Y^{\tate K}
\arrow[rightsquigarrow, crossing over]{ruuu}
\arrow[leftarrow]{u}
&
&
&
Z^{\tate G}
\arrow{r}
&
\left( Z^{\tate K} \right)^{\tate Q}
\end{tikzcd}
\end{equation}
in the cocartesian fibration classified by diagram \Cref{diagram.for.brax.two.points.of.GSpan}.  Then, we declare that the right-lax natural transformation
\[
\Fh
\xra{\htpy \ra \tate}
\Ftau
\]
takes this $[2]$-point of $\GSpan$ to the morphism in the $\infty$-category
\[
\Fun^\rlax ( \llax([2]),\Cat)
\]
corresponding to the diagram
\[ \begin{tikzcd}[row sep=2cm]
&
&
&
&
\Fun(\cB_1,\Spectra)
\arrow{rd}[sloped, pos=0.3]{Y \mapsto Z^{\tate Q}}
\\
&
\Fun(\cB_1,\Spectra)
\arrow{rrru}[sloped]{\sim}[swap, transform canvas={xshift=-1.5cm, yshift=-1.2cm}]{\rotatebox{0}{$\Rightarrow$}}
&
&
\Fun(\cB_0,\Spectra)
\arrow{rr}[transform canvas={yshift=0.8cm}]{\rotatebox{90}{$\Rightarrow$}}[swap]{X \mapsto Z^{\tate G}}
\arrow{ru}[sloped, pos=0.7]{X \mapsto Y^{\tate K}}
\arrow[leftarrow]{llld}[sloped, swap, pos=0.4]{\sim}
&
&
\Fun(\cB_2,\Spectra)
\\
\Fun(\cB_0,\Spectra)
\arrow{ru}[sloped, pos=0.7]{X \mapsto Y^{\htpy K}}
\arrow{rr}[swap]{X \mapsto Z^{\htpy G}}[transform canvas={yshift=0.8cm}]{\rotatebox{90}{$\Rightarrow$}}
&
&
\Fun(\cB_2,\Spectra)
\arrow[crossing over, leftarrow]{lu}[sloped, pos=0.1]{Y \mapsto Z^{\htpy Q}} 
\arrow{rrru}[sloped, swap]{\sim}[transform canvas={xshift=-2.7cm, yshift=-0.2cm}]{\rotatebox{50}{$\Rightarrow$}}[transform canvas={xshift=-1cm, yshift=0.6cm}]{\rotatebox{50}{$\Rightarrow$}}
\end{tikzcd}~, \]
which lax-commutes via the data described previously along with the (natural) bottom right commutative square in diagram \Cref{limit.point.for.h.to.tau.on.brax.two.point} (filling in the 3-cell). \qedhere
\end{enumerate}
\end{proof}

\begin{remark}
\label{remark.no.Tate.package.over.Z}
The proof of the proto Tate package carries through without change after replacing $\Spectra$ by \textit{any} stable $\infty$-category.  However, our construction of the Tate package in \cite{AMR-trace} uses crucially that $\Spectra$ is the stabilization of a \textit{cartesian} symmetric monoidal $\infty$-category: the Tate package arises from the \textit{diagonal package} for $\Spaces$.  Thus, while there exists a proto Tate package for $H\ZZ$-module spectra, it does not induce a Tate package over $H\ZZ$, concordant with the folklore result that $\TC$ cannot be defined in the $\ZZ$-linear setting.
\end{remark}

\begin{notation}
Consider the right action of the commutative monoid $\Nx$ on the circle group $\TT$, wherein the element $r \in \Nx$ acts by the homomorphism
\begin{equation}
\label{homomorphism.from.T.to.itself}
\TT
\xlongla{r}
\TT
\end{equation}
with kernel $\Cyclic_r \leq \TT$.  We write
\[
\WW := \TT \rtimes \Nx
\]
for the resulting semidirect product monoid.  So by definition, we have a pullback square
\[ \begin{tikzcd}
\BT
\arrow{r}
\arrow{d}
&
\BW
\arrow{d}
\\
\pt
\arrow{r}
&
\BN
\end{tikzcd} \]
in which the functor on the right is a right fibration, with cartesian monodromy functor over the morphism $[1] \xra{r} \BN$ given by the functor
\begin{equation}
\label{endomorphism.of.BT}
\BT
\xla{\sB r}
\BT
\end{equation}
obtained by applying $\sB$ to the homomorphism \Cref{homomorphism.from.T.to.itself}.
\end{notation}

\begin{remark}
The $\infty$-category $\BW$ arises most natively through manifold theory, as described in \cite{AMR-fact}.  Namely, can be described as the $\infty$-category whose objects are framed circles and whose morphisms are opposite to framed covering maps.  (The 1-fold covers determine the maximal subgroupoid $\BT \subset \BW$.)
\end{remark}

\begin{construction}
\label{construction.functor.to.GSpan}
We take the cartesian dual of the cartesian fibration
\[
(\BW \da \BN) \in \Cart_\BN
~.
\]
The resulting cocartesian fibration
\[
\left( \BW^\cartdual \da \BNop \right) \in \coCart_\BNop
\]
is classified by a functor
\[
\BNop
\longra
\Cat
\]
which sends the unique object of $\BNop$ to the space $\BT$ and sends the morphism $[1] \xra{r^\circ} \BNop$ to the endomorphism \Cref{endomorphism.of.BT}.  This is a $\BG$-Kan fibration, so determines a functor
\begin{equation}
\label{functor.from.BNop.to.GSpan}
\BNop
\longra
\GSpan
\end{equation}
(landing in the subcategory on those morphisms whose backwards functors are equivalences).
\end{construction}

\begin{observation}
It follows from the definitions that the pullback of the left-lax left $\GSpan$-module
\[
\Ftau
\in
\LMod_{\llax.\GSpan}
\]
of the proto Tate package (\Cref{proto.tate}) along the functor \Cref{functor.from.BNop.to.GSpan} is precisely the left-lax right $\BN$-module
\[
\SphTt
\in
\LMod_{\llax.\BNop}
\simeq
\RMod_{\llax.\BN}
~.
\]
\end{observation}


\begin{construction}
\label{constrn.actually.use.proto.tate}
Pulling back the entire proto Tate package (\Cref{proto.tate}) along the functor \Cref{functor.from.BNop.to.GSpan} yields a (strict) right $\BN$-module, which we denote by
\[
\SphTh
\in
\LMod_\BNop
\simeq
\RMod_\BN
~,
\]
as well as a morphism
\begin{equation}
\label{morphism.before.rlax.lims.from.cyclo.spt.w.frob.lifts.to.cyclo.spt}
\SphTh
\longra
\SphTt
\end{equation}
in $\LMod^\rlax_{\llax.\BNop} =: \RMod^\rlax_{\llax.\BN}$.
\end{construction}

\begin{remark}
We will reidentify the object $\SphTh \in \LMod_\BNop$ in simpler terms in \Cref{cor.cart.fibn.to.BN.for.htpy.cyclo.sp}.
\end{remark}


\subsection{Cyclotomic spectra with Frobenius lifts}
\label{subsection.cyclo.spt.w.frob}

Recall that a cyclotomic spectrum $T \in \Cyclo(\Spectra)$ has a cyclotomic structure map
\[
T
\xlongra{\sigma_r}
T^{\tate \Cyclic_r}
\]
for each $r \in \Nx$; this may be thought of as a \textit{Tate-valued Frobenius} map \cite{NS}.  On the other hand, it is also possible for a cyclotomic spectrum to have an honest Frobenius endomorphism: a \textit{Frobenius lift} of the cyclotomic structure map $\sigma_r$ is a (suitably equivariant) lift
\[ \begin{tikzcd}[row sep=1.5cm, column sep=1.5cm]
T
\arrow[dashed]{r}{\tilde{\sigma}_r}
\arrow{rd}[swap]{\sigma_r}
&
T^{\htpy \Cyclic_r}
\arrow{r}
\arrow{d}
&
T
\\
&
T^{\tate \Cyclic_r}
\end{tikzcd}~. \]
Requiring that these lifts be compatible for all $r \in \Nx$, we arrive at the following notion.

\begin{definition}
\label{define.cyclo.spt.with.frob.lifts}
The $\infty$-category of \bit{cyclotomic spectra with Frobenius lifts} is
\[
\Cyclo^\htpy(\Spectra)
:=
\lim^\rlax
\left( \SphTh \racts \BN \right)
~.
\]
\end{definition}

\begin{notation}
In order to emphasize the contrast with $\Cyclo^\htpy(\Spectra)$, we may write
\[
\Cyclo^\tate(\Spectra)
:=
\Cyclo(\Spectra)
\]
for the $\infty$-category of cyclotomic spectra.
\end{notation}

\begin{observation}
\label{obs.forget.frob.lifts}
Applying the right-lax limit functor
\[
\RMod^\rlax_{\llax.\BN}
\xra{\lim^\rlax}
\Cat
\]
to the morphism \Cref{morphism.before.rlax.lims.from.cyclo.spt.w.frob.lifts.to.cyclo.spt} defines a functor
\[
\Cyclo^\htpy(\Spectra)
\longra
\Cyclo^\tate(\Spectra)
=:
\Cyclo(\Spectra)
\]
which forgets the Frobenius lifts.
\end{observation}

\begin{remark}
For our purposes here, the main example of a cyclotomic spectrum with Frobenius lifts is a \textit{trivial} cyclotomic spectrum; indeed, we will construct the ``trivial cyclotomic spectrum'' functor in \Cref{subsection.trivial.cyclo.spt} as a composite
\[
\Spectra
\longra
\Cyclo^\htpy(\Spectra)
\longra
\Cyclo(\Spectra)
~.
\]
However, there is another interesting source of cyclotomic spectra with Frobenius lifts, namely the suspension spectrum of an \textit{unstable cyclotomic space}, i.e.\! a functor $\BW \ra \Spaces$.\footnote{This follows immediately from \Cref{cor.cart.fibn.to.BN.for.htpy.cyclo.sp}.}  In turn, the main example of an unstable cyclotomic space is the factorization homology
\[
\THH_\Spaces(\cC)
:=
\int_{S^1} \cC
\]
of a spatially-enriched (i.e.\! unenriched) $\infty$-category.\footnote{As factorization homology is by definition a colimit, we have an identification
\[
\Sigma^\infty_+ \left( \int_{S^1} \cC \right)
\simeq
\int_{S^1} (\fB\Sigma^\infty_+) (\cC)
~,
\]
where (as in \cite[\Cref*{trace:section.cyclo.trace}]{AMR-trace}) we write $(\fB\Sigma^\infty_+) (\cC)$ to denote the $\Spectra$-enriched $\infty$-category obtained from $\cC$ by taking hom-wise suspension spectra.  This provides another (essentially equivalent) source of cyclotomic spectra with Frobenius lifts.}  This relationship is a key ingredient in our construction in \cite{AMR-trace} of the cyclotomic trace.
\end{remark}

We now provide a general result which specializes to identify the cocartesian dual
\[
\left( (\SphTh)^\cocartdual \da \BN \right) \in \Cart_\BN
~.\footnote{As is explained in \cite{AMR-trace}, this general result plays a crucial role in the consideration of enriched factorization homology and its functoriality.}
\]
This identification will allow us to easily construct the functor
\[
\Spectra
\triv^\htpy
\Cyclo^\htpy(\Spectra)
\]
taking a spectrum to the corresponding trivial cyclotomic spectrum with Frobenius lifts in \Cref{subsection.trivial.cyclo.spt}.

\begin{notation}
We fix a cocartesian fibration
\[ (\cE \da \BW) \in \coCart_\BW~, \]
and we write
\[
\begin{tikzcd}
\cE_0
\arrow[hook, two heads]{r}
\arrow{d}
&
\cE
\arrow{d}
\\
\BT
\arrow[hook, two heads]{r}
&
\BW
\end{tikzcd}
\]
for the pullback (or equivalently, the fiber over the unique point in $\BN$).\footnote{Here we'll just be interested in the case that $\cE \xra{\sim} \BW$, but the proof of the general result (\Cref{identify.relFun.as.cart.fibn.with.htpy.fixedpts}) is no more difficult.}
\end{notation}

\begin{construction}
The space of degree-$r$ maps in $\BW$ -- that is, maps that project to the map $[1] \xra{r} \BN$ -- forms a copy of $\BC_r$.  Hence, for an object $e \in \cE_0$, cocartesian pushforward over all these maps simultaneously defines a functor
\[
\pi_r^*(e)
:
\BC_r
\longra
\cE_0
~.
\]
Assembling this construction over all $e \in \cE_0$, we obtain a functor
\[
\pi_r^*
:
\cE_0
\longra
\Fun(\BC_r,\cE_0)
~.
\]
\end{construction}

\begin{lemma}
\label{identify.relFun.as.cart.fibn.with.htpy.fixedpts}
Let $\cV$ be an $\infty$-category admitting homotopy $\Cyclic_r$-fixedpoints for all $r \in \Nx$.  Then, the functor
\begin{equation}
\label{Fun.rel.BN.E.V}
\begin{tikzcd}
\Fun^\rel_{/\BN}(\cE,\ul{\cV})
\arrow{d}
\\
\BN
\end{tikzcd}
\end{equation}
is a cartesian fibration, with cartesian monodromy functor over the morphism $[1] \xra{r} \BN$ given by the formula
\[
\begin{tikzcd}[row sep=0cm]
\Fun(\cE_0,\cV)
&
\Fun(\cE_0,\cV)
\arrow{l}
\\
\rotatebox{90}{$\in$}
&
\rotatebox{90}{$\in$}
\\
(F(\pi_r^*(-)))^{\htpy \Cyclic_r}
&
F
\arrow[mapsto]{l}
\end{tikzcd}~.
\]
\end{lemma}

\begin{proof}
The functor \Cref{Fun.rel.BN.E.V} restricts to a cartesian fibration over the morphism $[1] \xra{r} \BN$ with cartesian monodromy given by the leftwards composite
\[ \begin{tikzcd}[ampersand replacement=\&, column sep=1.5cm]
\Gamma_{\{0\}} \Cref{Fun.rel.BN.E.V}
\&
\Gamma_{[1]} \Cref{Fun.rel.BN.E.V}
\arrow[transform canvas={yshift=0.9ex}]{r}
\arrow[dashed, leftarrow, transform canvas={yshift=-0.9ex}]{r}[yshift=-0.2ex]{\bot}
\arrow{l}
\&
\Gamma_{\{1\}} \Cref{Fun.rel.BN.E.V}
\end{tikzcd}~, \]
i.e.\! the leftwards composite
\[ \begin{tikzcd}[ampersand replacement=\&, column sep=1.5cm]
\Fun(\cE_{|\{0\}},\cV)
\&
\Fun \left( \cE_{|[1]} , \cV \right)
\arrow[transform canvas={yshift=0.9ex}]{r}
\arrow[dashed, leftarrow, transform canvas={yshift=-0.9ex}]{r}[yshift=-0.2ex]{\bot}
\arrow{l}
\&
\Fun ( \cE_{|\{1\}} , \cV )
\end{tikzcd} \]
in which the right adjoint is given by right Kan extension.  Given an object $F \in \Fun(\cE_{|\{1\}},\cV)$, the value of its image on an object $e \in \cE_{|\{0\}} \subset \cE_{|[1]}$ is therefore given by the limit of the composite
\[
\cE_{|\{1\}} \underset{\cE_{|[1]}}{\times} \left( \cE_{|[1]} \right)_{e/}
\longra
\cE_{|\{1\}}
\xlongra{F}
\cV
~.
\]
Let us write
\[
\cE_{|\{1\}} \underset{\cE_{|[1]}}{\times} \left( \cE_{|[1]} \right)_{e/^{\cocart/\BW}}
\longhookra
\cE_{|\{1\}} \underset{\cE_{|[1]}}{\times} \left( \cE_{|[1]} \right)_{e/}
\]
for the inclusion of the subcategory on those morphisms with source $e \in \cE_{|[1]}$ and target lying in $\cE_{|\{1\}} \subset \cE_{|[1]}$ that are cocartesian over $\BW$.  Note that the source of this inclusion is equivalent to $\BC_r$, incarnated as the space of morphisms in $\BW$ lying over $[1] \xra{r} \BN$ with source the image of $e$ (since each such morphism admits a unique cocartesian lift to $\cE$ with source $e$).  Then, the asserted identification 
of the cartesian monodromy for \Cref{Fun.rel.BN.E.V} follows from the fact that this inclusion is initial, since it admits a right adjoint given by the factorization system on $\cE$ coming from its cocartesian fibration to $\BW$.  Moreover, the fact that it is a cartesian (and not just locally cartesian) fibration follows from the observation that these cartesian monodromies compose.
\end{proof}

\begin{cor}
\label{cor.cart.fibn.to.BN.for.htpy.cyclo.sp}
The cocartesian dual of the cocartesian fibration
\[
(\SphTh \da \BNop)
\in \coCart_\BNop
\]
is the cartesian fibration
\[
\left(
\begin{tikzcd}
\Fun^\rel_{/\BN} \left( \BW,\ul{\Spectra} \right)
\arrow{d}
\\
\BN
\end{tikzcd}
\right)
\in
\Cart_\BN
~. \]
\end{cor}

\begin{proof}
This follows from \Cref{identify.relFun.as.cart.fibn.with.htpy.fixedpts} in the case that $\cE \xra{\sim} \BW$ and $\cV = \Spectra$.
\end{proof}

\subsection{Trivial cyclotomic spectra}
\label{subsection.trivial.cyclo.spt}

In this subsection, we construction the trivial cyclotomic spectrum functor
\[
\Spectra
\xra{\triv}
\Cyclo(\Spectra)
~;
\]
this is given as \Cref{define.triv.cyclo.spt}.

\begin{construction}
Consider the morphism
\begin{equation}
\label{terminal.morphism.from.BW.to.BN.over.BN}
\begin{tikzcd}
\BN
\arrow{rd}[swap, sloped, pos=0.6]{\sim}
&
&
\BW
\arrow{ld}
\arrow{ll}
\\
&
\BN
\end{tikzcd}
\end{equation}
in $\EFib_\BN$.  Applying the functor
\[
\left( \EFib_\BN \right)^\op
\xra{\Fun^\rel_{/\BN} \left( - , \ul{\Spectra} \right)}
\Cat_{/\BN}
\]
to the morphism \Cref{terminal.morphism.from.BW.to.BN.over.BN} yields a morphism
\[ \begin{tikzcd}
\ul{\Spectra}
\arrow{rr}
\arrow{rd}
&
&
\Fun^\rel_{/\BN} \left( \BW , \ul{\Spectra} \right)
\arrow{ld}
\\
&
\BN
\end{tikzcd} \]
in $\Cat_{/\BN}$.  In fact, this is a morphism in $\RMod^\rlax_\BN := \Cat_{\cart/\BN}$, since both maps to $\BN$ are cartesian fibrations -- the former manisfestly and the latter by \Cref{identify.relFun.as.cart.fibn.with.htpy.fixedpts}.  Over the unique object of $\BN$, this is the functor
\[
\Spectra
\xra{\const}
\Fun(\BT,\Spectra)
~,
\]
a left adjoint.  Thus, by (the 2-opposite of) \Cref{lemma.ptwise.radjt.has.ptwise.ladjt}, this provides a canonical morphism
\begin{equation}
\label{llax.map.backwards.from.htpy.cyclo.sp.to.triv.sp}
\ul{\Spectra}
\longla
\Fun^\rel_{/\BN} \left( \BW , \ul{\Spectra} \right)
\end{equation}
in $\RMod^\llax_\BN$ which over the unique object of $\BN$ is the right adjoint
\[
\Spectra
\xla{(-)^{\htpy \TT}}
\Fun(\BT,\Spectra)
~.
\]
But in fact, the morphism \Cref{llax.map.backwards.from.htpy.cyclo.sp.to.triv.sp} is \textit{strictly} equivariant, i.e.\! it lies in the subcategory $\RMod_\BN := \Cart_\BN$, because homotopy fixedpoints compose in the sense that
\[
\left( (-)^{\htpy \Cyclic_r} \right)^{\htpy (\TT/\Cyclic_r)}
\simeq
(-)^{\htpy \TT}
~.
\]
Thus, by \Cref{cor.cart.fibn.to.BN.for.htpy.cyclo.sp}, taking the cartesian dual of the morphism \Cref{llax.map.backwards.from.htpy.cyclo.sp.to.triv.sp} (considered in $\Cart_\BN$) yields a morphism
\begin{equation}
\label{right.adjoint.from.Modh.to.sp}
\begin{tikzcd}
\ul{\Spectra}
\arrow{rd}
&
&
\SphTh
\arrow{ll}
\arrow{ld}
\\
&
\BNop
\end{tikzcd}
\end{equation}
in $\LMod_\BNop := \coCart_\BNop$.  This is again a fiberwise right adjoint, and so by \Cref{lem.get.r.lax.left.adjt} its fiberwise left adjoint assembles into a morphism
\begin{equation}
\label{take.rlax.lim.to.get.triv.cyclo.sp.with.frob.lifts}
\ul{\Spectra}
\longra
\SphTh
\end{equation}
in $\LMod^\rlax_\BNop$.
\end{construction}

\begin{definition}
\label{define.triv.cyclo.spt}
We define the \bit{trivial cyclotomic spectrum with Frobenius lifts} functor to be the composite
\[
\triv^\htpy
:
\Spectra
\xra{\const}
\Fun ( \BNop , \Spectra )
\simeq
\lim^\rlax \left( \ul{\Spectra} \racts \BN \right)
\xra{\lim^\rlax\Cref{take.rlax.lim.to.get.triv.cyclo.sp.with.frob.lifts}}
\lim^\rlax \left( \SphTh \overset{\llax}{\racts} \BN \right)
=:
\Cyclo^\htpy(\Spectra)
~,
\]
where the equivalence follows from \Cref{obs.rlax.lim.of.trivial.action}.  Thereafter, we define the \bit{trivial cyclotomic spectrum} functor to be the composite
\[
\triv
:
\Spectra
\xra{\triv^\htpy}
\Cyclo^\htpy(\Spectra)
\longra
\Cyclo(\Spectra)
~.
\]
\end{definition}

\subsection{Partial adjunctions}
\label{subsection.partial.adjns}

In this subsection, we introduce a formalism which will be useful in our proof of \Cref{mainthm.formula.for.TC} in \Cref{subsection.formula.for.TC.for.real}.

\begin{definition}
\label{defn.of.partial.adjn}
A (\bit{left-})\bit{partial adjunction} is the data of two $\infty$-categories $\cC$ and $\cD$ equipped with full subcategories $\cC_0 \subset \cC$ and $\cD_0 \subset \cD$ and a lax-commutative square
\begin{equation}
\label{the.partial.adjn}
\begin{tikzcd}
\cC_0
\arrow{r}{L}[swap, transform canvas={yshift=-0.3cm}]{\stackrel{\eta}{\Rightarrow}}
\arrow[hook]{d}[swap]{\ff}
&
\cD_0
\arrow[hook]{d}{\ff}
\\
\cC
&
\cD
\arrow{l}{R}
\end{tikzcd}
\end{equation}
such that for all $c \in \cC_0$ and $d \in \cD$, the induced composite map
\[
\hom_\cD(Lc,d)
\xlongra{R}
\hom_\cC(RLc,Rd)
\xlongra{\eta_c^*}
\hom_\cC(c,Rd)
\]
is an equivalence.  We refer to $L$ as the \bit{partial left adjoint} of the partial adjunction and to $R$ as the \bit{right adjoint} of the partial adjunction.
\end{definition}

\begin{observation}
\label{obs.partial.adjns.compose}
Partial adjunctions compose: if in the diagram
\[ \begin{tikzcd}
\cC_0
\arrow{r}{L_1}[swap, transform canvas={yshift=-0.3cm}]{\stackrel{\eta_1}{\Rightarrow}}
\arrow[hook]{d}[swap]{\ff}
&
\cD_0
\arrow[hook]{d}{\ff}
\arrow{r}{L_2}[swap, transform canvas={yshift=-0.3cm}]{\stackrel{\eta_2}{\Rightarrow}}
&
\cE_0
\arrow[hook]{d}{\ff}
\\
\cC
&
\cD
\arrow{l}{R_1}
&
\cE
\arrow{l}{R_2}
\end{tikzcd} \]
the two squares are partial adjunctions, then so is the outer rectangle.
\end{observation}

\begin{observation}
\label{obs.partial.adjn.gives.adjn}
Given a partial adjunction \Cref{the.partial.adjn}, a factorization
\[ \begin{tikzcd}
\cC_0
\arrow[hook]{d}[swap]{\ff}
&
\cD_0
\arrow[dashed]{l}[swap]{\tilde{R}}
\arrow[hook]{d}{\ff}
\\
\cC
&
\cD
\arrow{l}{R}
\end{tikzcd} \]
of the restriction $R_{|\cD_0}$ induces an adjunction
\[ \begin{tikzcd}[column sep=2cm, row sep=0cm]
\cC_0
\arrow[transform canvas={yshift=0.9ex}]{r}{L}
\arrow[leftarrow, transform canvas={yshift=-0.9ex}]{r}[yshift=-0.2ex]{\bot}[swap]{\tilde{R}}
&
\cD_0
\end{tikzcd}
~.
\]
\end{observation}

\subsection{The formula for $\TC$}
\label{subsection.formula.for.TC.for.real}

\begin{theorem}
\label{thm.formula.for.cyclo.fixedpts}
There exists a canonical functor
\begin{equation}
\label{htpy.T.fixedpoints.plus.plus}
\Cyclo(\Spectra)
\xra{(-)^{\htpy \TT}}
\Fun(\sd(\BNop),\Spectra)
\end{equation}
taking a cyclotomic spectrum $T \in \Cyclo(\Spectra)$ to the diagram
\begin{equation}
\label{formula.for.htpy.T.fixedpoints.plus.plus}
\begin{tikzcd}[row sep=0cm, column sep=1.5cm]
&
\sd(\BNop)
\arrow{r}{T^{\htpy \TT}}
&
\Spectra
\\
&
\rotatebox{90}{$\in$}
&
\rotatebox{90}{$\in$}
\\
(r_1,\ldots,r_k)
=:
\hspace{-2.5cm}
&
W
\arrow[maps to]{r}
&
\left( T^{\tate \Cyclic_W} \right)^{\htpy \TT}
&
\hspace{-1.7cm}
:= \left( T^{\tate \Cyclic_{r_1} \cdots \tate \Cyclic_{r_k}} \right)^{\htpy \TT}
\end{tikzcd}~.
\end{equation}
Moreover, postcomposing with the limit gives a canonical factorization
\[
\begin{tikzcd}[row sep=1.5cm]
\Spectra
\arrow[leftarrow]{rr}{(-)^{\htpy \Cyclo}}
&
&
\Cyclo(\Spectra)
\arrow{ld}{(-)^{\htpy \TT}}
\\
&
\Fun(\sd(\BN),\Spectra)
\arrow{lu}{\lim}
\end{tikzcd}
\]
of the right adjoint to the trivial cyclotomic spectrum functor
\[
\Spectra
\xra{\triv}
\Cyclo(\Spectra)
.
\]
\end{theorem}

\begin{notation}
\label{Y.for.Yorlax}
For a base $\infty$-category $\cB$, we simply write
\[
\Yorlax := \Yo^\rlax
\]
for the right-lax Yoneda embedding
\[ \begin{tikzcd}[row sep=0cm]
\loc.\coCart_\cB
\arrow{r}{\Yorlax}
&
\Fun \left( \left( \bDelta_{/\cB} \right)^\op , \Cat \right)
\\
\rotatebox{90}{$\in$}
&
\rotatebox{90}{$\in$}
\\
(\cE \da \cB)
\arrow[mapsto]{r}
&
\left(
\left( [n] \xlongra{\varphi} \cB \right)
\longmapsto
\Fun^\cocart_{/[n]} \left( \sd([n]) , \varphi^* \cE \right)
\right)
\end{tikzcd} \]
of \Cref{define.Yo.rlax}, so that $\lim^\rlax := \lim \circ \Yorlax$.
\end{notation}

\begin{notation}
Complementing \Cref{Y.for.Yorlax}, we introduce the notation
\[ \begin{tikzcd}[row sep=0cm]
\loc.\coCart_\cB
\arrow{r}{\Zorlax}
&
\Fun \left( \left( \bDelta_{/\cB} \right)^\op , \Cat \right)
\\
\rotatebox{90}{$\in$}
&
\rotatebox{90}{$\in$}
\\
(\cE \da \cB)
\arrow[mapsto]{r}
&
\left(
\left( [n] \xlongra{\varphi} \cB \right)
\longmapsto
\Fun_{/[n]} \left( \sd([n]) , \varphi^* \cE \right)
\right)
\end{tikzcd}
~.
\]
\end{notation}

\begin{observation}
\label{Yorlax.maps.to.Zorlax}
Directly from the definitions, there is a natural transformation
\[ \begin{tikzcd}[column sep=1.5cm]
\loc.\coCart_\cB
\arrow[bend left]{r}{\Yorlax}[swap, transform canvas={yshift=-0.55cm}]{\rotatebox{-90}{$\Rightarrow$}}
\arrow[bend right]{r}[swap]{\Zorlax}
&
\Fun \left( \left( \bDelta_{/\cB} \right)^\op , \Cat \right)
\end{tikzcd} \]
whose components are pointwise fully faithful embeddings.
\end{observation}

\begin{observation}
\label{obs.lim.of.Zorlax.of.trivial.action}
Consider the projection from the product
\[
\ul{\cG}
:=
\cG \times \cB
\longra
\cB
\]
as an object of $\LMod_\cB \subset \LMod_{\llax.\cB}$.  Then, the limit of the presheaf
\[
\Zorlax \left( \ul{\cG} \right) \in \Fun \left( \left( \bDelta_{/\cB} \right)^\op , \Cat \right)
\]
is given by
\begin{align*}
\lim \left( \Zorlax \left( \ul{\cG} \right) \right)
& :=
\lim_{\left( [n] \xra{\varphi} \cB \right) \in \left( \bDelta_{/\cB} \right)^\op} \Fun_{/[n]} \left( \sd([n]) , \varphi^* \ul{\cG} \right)
\\
& \simeq
\lim_{\left( [n] \xra{\varphi} \cB \right) \in \left( \bDelta_{/\cB} \right)^\op}
\Fun ( \sd([n]) , \cG )
\\
& \simeq
\Fun ( \sd(\cB) , \cG )
~.
\end{align*}
\end{observation}

\begin{proof}[Proof of \Cref{thm.formula.for.cyclo.fixedpts}]
Our proof takes places within the context of the string of composable partial adjunctions of \Cref{string.of.partial.adjns}.
\begin{sidewaysfigure}

\vspace{400pt}

\[ \begin{tikzcd}[row sep=2cm, column sep=2cm]
&
\Fun(\BNop,\Spectra)
&
\Cyclo^\htpy(\Spectra)
&
\Cyclo(\Spectra)
\\[-2cm]
&
\rotatebox{-90}{$\simeq$}
&
\rotatebox{-90}{$:=$}
&
\rotatebox{-90}{$:=$}
\\[-2cm]
\Spectra
\arrow{ruu}[sloped, pos=0.6]{\const}
\arrow{r}[swap, transform canvas={yshift=-1.4cm, xshift=0.5cm}]{\Rightarrow}
\arrow[equal]{d}
&
\lim \left( \Yorlax \left( \ul{\Spectra} \right) \right)
\arrow{r}[swap, transform canvas={yshift=-1.4cm, xshift=0.1cm}]{\Rightarrow}
\arrow[hook]{d}
&
\lim
\left(
\Yorlax
\left(
\SphTh
\right)
\right)
\arrow{r}[swap, transform canvas={yshift=-1.4cm}]{\Rightarrow}
\arrow[hook]{d}
&
\lim
\left(
\Yorlax
\left(
\SphTt
\right)
\right)
\arrow[hook]{d}
\\
\Spectra
&
\lim \left( \Zorlax \left( \ul{\Spectra} \right) \right)
\arrow{l}
&
\lim
\left(
\Zorlax
\left(
\SphTh
\right)
\right)
\arrow{l}
&
\lim
\left(
\Zorlax
\left(
\SphTt
\right)
\right)
\arrow{l}
\\[-2cm]
&
\rotatebox{-90}{$\simeq$}
\\[-2cm]
&
\Fun(\sd(\BNop),\Spectra)
\arrow{luu}[sloped, pos=0.4]{\lim}
\end{tikzcd} \]

\vspace{2.5cm}

\caption{The string of partial adjunctions in the proof of \Cref{thm.formula.for.cyclo.fixedpts}.}

\label{string.of.partial.adjns}

\end{sidewaysfigure}
It will be immediate that the upper composite is given by the functor
\[
\Spectra
\xra{\triv}
\Cyclo(\Spectra)
~.
\]
Moreover, we take the functor \Cref{htpy.T.fixedpoints.plus.plus} to be the composite
\[ \begin{tikzcd}[row sep=2cm, column sep=2cm]
&[-2.2cm]
&
&
\Cyclo(\Spectra)
\arrow[dashed, bend right=15]{lllddd}[swap]{(-)^{\htpy \TT}}
\\[-2cm]
&[-2.2cm]
&
&
\rotatebox{-90}{$:=$}
\\[-2cm]
&[-2.2cm]
&
&
\lim \left( \Yorlax \left( \SphTt \right) \right)
\arrow[hook]{d}
\\
\Fun ( \sd(\BNop) , \Spectra )
&[-2.2cm]
\simeq
\lim \left( \Zorlax \left( \ul{\Spectra} \right) \right)
&
\lim \left( \Zorlax \left( \SphTh \right) \right)
\arrow{l}
&
\lim \left( \Zorlax \left( \SphTt \right) \right)
\arrow{l}
\end{tikzcd}~; \]
it will follow from unwinding the definitions that this does indeed act as asserted in formula \Cref{formula.for.htpy.T.fixedpoints.plus.plus}.  The result will then follow immediately from Observations \ref{obs.partial.adjns.compose} \and \ref{obs.partial.adjn.gives.adjn}.  Thus, it remains to describe the three partial adjunctions of \Cref{string.of.partial.adjns}.  We declare immediately that all downwards functors besides the leftmost one arise from \Cref{Yorlax.maps.to.Zorlax}; as fully faithful embeddings are stable under limits in $\Ar(\Cat)$, all downwards functors are indeed fully faithful embeddings, as required by \Cref{defn.of.partial.adjn}.  We also note that the identifications of the two limits in the second column with the indicated functor $\infty$-categories respectively follow from Observations \ref{obs.rlax.lim.of.trivial.action} \and \ref{obs.lim.of.Zorlax.of.trivial.action}.

Let us first address the partial adjunction on the left in \Cref{string.of.partial.adjns}.  Observe that we have an evident identification
\[ \begin{tikzcd}[row sep=1.5cm] 
\Spectra
\arrow{r}{\const}
\arrow[dashed]{rd}[sloped, swap, pos=0.6]{\const}
&
\Fun(\BNop,\Spectra)
\arrow[hook]{d}
\\
&
\Fun(\sd(\BNop),\Spectra)
\end{tikzcd} \]
of the indicated composite.  Then, we obtain the desired partial adjunction the adjunction
\[ \begin{tikzcd}[column sep=2cm, row sep=0cm]
\Spectra
\arrow[transform canvas={yshift=0.9ex}]{r}{\const}
\arrow[leftarrow, transform canvas={yshift=-0.9ex}]{r}[yshift=-0.2ex]{\bot}[swap]{\lim}
&
\Fun(\sd(\BNop),\Spectra)
\end{tikzcd}~: \]
it follows immediately from the resulting diagram
\[ \begin{tikzcd}[row sep=1.5cm, column sep=1.5cm]
\Spectra
\arrow{r}
\arrow[equal]{d}
&
\Fun(\BNop,\Spectra)
\arrow[hook]{d}
\\
\Spectra
\arrow{r}[swap, transform canvas={xshift=0.7cm, yshift=-1cm}]{\Rightarrow}
\arrow[equal]{d}
&
\Fun(\sd(\BNop),\Spectra)
\arrow[equal]{d}
\\
\Spectra
&
\Fun(\sd(\BNop),\Spectra)
\arrow{l}
\end{tikzcd} \]
in which the natural transformation is the unit and the upper square commutes.\footnote{It is also possible to obtain this partial left adjoint as the limit of a morphism
\[
\const(\Spectra)
\longra
\ul{\Spectra}
\]
in $\Fun \left( \left( \bDelta_{/\BNop} \right)^\op , \Cat\right)$.  However, the right adjoint does not exist at this level.}

Let us next address the partial adjunction in the middle in \Cref{string.of.partial.adjns}.  Its upper functor is the right-lax limit of the morphism \Cref{take.rlax.lim.to.get.triv.cyclo.sp.with.frob.lifts} in $\LMod^\rlax_\BNop$.  By construction, this morphism is the left adjoint of an adjunction in $\LMod^\rlax_\BNop$, whose right adjoint lies in the subcategory $\LMod_\BNop \subset \LMod^\rlax_\BNop$.  From the resulting diagram
\[ \begin{tikzcd}[row sep=1.5cm, column sep=1.5cm]
\Adj
\arrow{r}
&
\LMod^\rlax_\BNop
\arrow[hook]{r}{\ff}
&
\Fun \left( \left( \bDelta_\BNop \right)^\op , \Cat \right)
\arrow{r}{\lim}
&
\Cat
\\
{[1]}
\arrow{u}{\sf r.adjt}
\arrow{r}
&
\LMod_\BNop
\arrow[hook, two heads]{u}
\arrow[hook]{ru}[pos=0.55]{\Yorlax}[swap, sloped, transform canvas={xshift=0.1cm, yshift=-0.2cm}]{\rotatebox{-90}{$\Rightarrow$}}
\arrow[bend right]{ru}[swap, pos=0.4]{\Zorlax}
\end{tikzcd} \]
we extract the diagram
\[ \begin{tikzcd}[row sep=1.5cm, column sep=1.5cm]
\lim \left( \Yorlax \left( \ul{\Spectra} \right) \right)
\arrow{r}[swap, transform canvas={xshift=0.2cm, yshift=-1cm}]{\Rightarrow}
\arrow[equal]{d}
&
\lim \left( \Yorlax \left( \SphTh \right) \right)
\arrow[equal]{d}
\\
\lim \left( \Yorlax \left( \ul{\Spectra} \right) \right)
\arrow[hook]{d}
&
\lim \left( \Yorlax \left( \SphTh \right) \right)
\arrow{l}
\arrow[hook]{d}
\\
\lim \left( \Zorlax \left( \ul{\Spectra} \right) \right)
&
\lim \left( \Zorlax \left( \SphTh \right) \right)
\arrow{l}
\end{tikzcd} \]
in which the natural transformation is the unit and the lower square commutes, from which the desired partial adjunction immediately follows.

Let us finally address the partial adjunction on the right in \Cref{string.of.partial.adjns}.  Its upper functor is the right-lax limit of the morphism \Cref{morphism.before.rlax.lims.from.cyclo.spt.w.frob.lifts.to.cyclo.spt} in $\LMod^\rlax_{\llax.\BNop}$.  As this is an equivalence on underlying $\infty$-categories, it is not hard to see that it canonically induces the remaining data of the partial adjunction: indeed, this all arises diagramatically.  
\end{proof}

\begin{observation}
\label{obs.h.to.tau.map.for.cyclo.spt.w.frob.lifts}
The proof of \Cref{thm.formula.for.cyclo.fixedpts} directly furnishes a natural transformation
\[ \begin{tikzcd}[row sep=2cm]
\Cyclo^\htpy(\Spectra)
\arrow{rr}{\fgt}
\arrow{rd}[swap]{(-)^{\htpy \WW}}[sloped, transform canvas={xshift=0.7cm,yshift=0.3cm}]{\Uparrow}
&
&
\Cyclo^\tate(\Spectra)
\arrow{ld}{(-)^{\htpy \Cyclo}}
\\
&
\Spectra
\end{tikzcd}~. \]
Specifically, this is obtained from the diagram in \Cref{string.of.partial.adjns} by applying the lower composite
\[
\Spectra
\longla
\lim \left( \Zorlax \left( \ul{\Spectra} \right) \right)
\longla
\lim ( \Zorlax ( \SphTh ) )
\]
to the unit of the partial adjunction on the right.
\end{observation}

\begin{remark}
\label{unwind.sd.BN}
Let us describe the category $\sd(\BN)$ and the diagram \Cref{formula.for.htpy.T.fixedpoints.plus.plus} in detail.  First of all, an object of $\sd(\BN)$ is unambiguously specified by a possibly empty word
\[ W := (r_1,\ldots,r_k) \]
in $\Nx$ not containing any 1's.  Then, the morphisms are generated by the operations of
\begin{itemize}
\item adding a letter at the beginning, e.g.\! $(2,3) \ra (5,2,3)$,
\item adding a letter at the end, e.g.\! $(2,3) \ra (2,3,5)$, and
\item factoring a letter, e.g.\! $(5,6,4) \ra (5,3,2,4)$;
\end{itemize}
by convention, the operations of adding the same letter before and after the empty word $\emptyword$ are distinct.  Thus, a small portion of this category is depicted by the diagram
\begin{figure}[h]
\begin{tikzcd}[column sep=2cm]
\emptyword
\arrow[transform canvas={yshift=0.8ex}]{r}
\arrow[transform canvas={yshift=-0.8ex}]{r}
\arrow[transform canvas={yshift=0.725ex, xshift=0.325ex}]{rd}
\arrow[transform canvas={yshift=-0.725ex, xshift=-0.325ex}]{rd}
\arrow[transform canvas={yshift=0.55ex, xshift=0.55ex}]{rdd}
\arrow[transform canvas={yshift=-0.55ex, xshift=-0.55ex}]{rdd}
&
(2)
\arrow[transform canvas={yshift=0.8ex}]{r}
\arrow[transform canvas={yshift=-0.8ex}]{r}
\arrow{rd}
\arrow{rdd}
&
(2,2)
&
\cdots
\\
&
(3)
\arrow{r}
\arrow{rd}
\arrow[transform canvas={yshift=0.55ex, xshift=0.55ex}]{rdd}
\arrow[transform canvas={yshift=-0.55ex, xshift=-0.55ex}]{rdd}
&
(2,3)
&
\cdots
\\
&
(6)
\arrow{ru}
\arrow{r}
&
(3,2)
&
\cdots
\\
&
\vdots
&
(3,3)
&
\cdots
\\
&
&
\vdots
&
\ddots
\end{tikzcd}
\caption{A small portion of the category $\sd(\BN)$.}
\label{picture.of.sd.BN}
\end{figure}
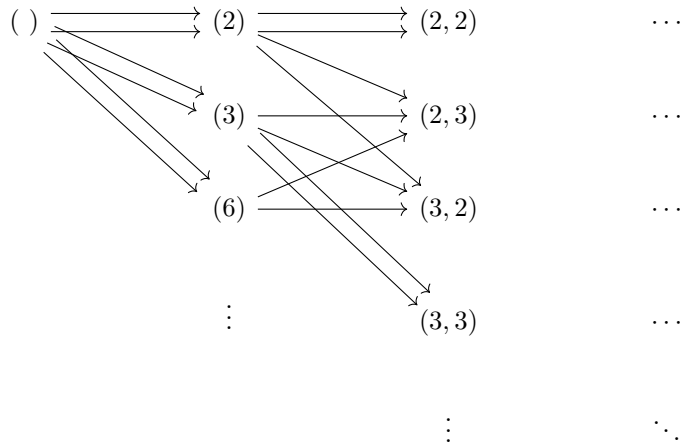
in \Cref{picture.of.sd.BN}, in which for instance the composite
\[
\emptyword \sim ( 1 )
\longra
(2 , 1 )
\sim (2)
\longra
(2,3)
\]
is identified with the composite
\[
\emptyword \sim ( 1 )
\longra
( 1 , 3)
\sim (3)
\longra
(2,3)
~.
\]
From here, the corresponding structure maps of diagram \Cref{formula.for.htpy.T.fixedpoints.plus.plus} are respectively given by
\begin{itemize}
\item the map
\begin{equation}
\label{structure.map.prepend.letter}
\left( \left( T^{\tate \Cyclic_2} \right)^{\tate \Cyclic_3} \right)^{\htpy \TT}
\xra{ \left( \left( \left( \sigma_5 \right)^{\tate \Cyclic_2} \right)^{\tate \Cyclic_3} \right)^{\htpy \TT} }
\left( \left( \left( T^{\tate \Cyclic_5} \right)^{\tate \Cyclic_2} \right)^{\tate \Cyclic_3} \right)^{\htpy \TT}
\end{equation}
determined by the cyclotomic structure map $\sigma_5$,
\item the map
\begin{equation}
\label{structure.map.append.letter}
\left( \left( T^{\tate \Cyclic_2} \right)^{\tate \Cyclic_3} \right)^{\htpy \TT}
\simeq
\left( \left( \left( T^{\tate \Cyclic_2} \right)^{\tate \Cyclic_3} \right)^{\htpy \Cyclic_5} \right)^{\htpy (\TT/\Cyclic_5)}
\longra
\left( \left( T^{\tate \Cyclic_2 \tate \Cyclic_3} \right)^{\tate \Cyclic_5} \right)^{\htpy (\TT/\Cyclic_5)}
\simeq
\left( \left( T^{\tate \Cyclic_2 \tate \Cyclic_3} \right)^{\tate \Cyclic_5} \right)^{\htpy \TT}
\end{equation}
determined by the natural transformation
\[
(-)^{\htpy \Cyclic_5}
\longra
(-)^{\tate \Cyclic_5}
\]
coming from the definition of the Tate construction, in which
\begin{itemize}
\item the first equivalence comes from the fact that homotopy fixedpoints compose, while
\item the second equivalence is simply the identification $(\TT/\Cyclic_5) \simeq \TT$,
\end{itemize}
and
\item the map
\begin{equation}
\label{structure.map.factor.letter}
\left( \left( T^{\tate \Cyclic_5} \right)^{\tate \Cyclic_6} \right)^{\tate \Cyclic_4}
\xra{ \left( \chi_{3,2} \right)^{\tate \Cyclic_4} }
\left( \left( \left( T^{\tate \Cyclic_5} \right)^{\tate \Cyclic_3} \right)^{\tate \Cyclic_2} \right)^{\tate \Cyclic_4}
\end{equation}
determined by the component of the natural transformation $\chi_{3,2}$ at the object $T^{\tate \Cyclic_5}$.
\end{itemize}
In particular, in the diagram \Cref{formula.for.htpy.T.fixedpoints.plus.plus}, the object $T^{\htpy \TT}$ itself -- the value at the empty word $\emptyword \in \sd(\BN)$ -- maps out to every other object $\left( T^{\tate \Cyclic_W} \right)^{\htpy \TT}$ in a multitude of different ways, according to the various possible sequences of prepending, appending, and factoring elements of $\Nx$ that take $\emptyword$ to $W$.
\end{remark}

\bibliographystyle{amsalpha}
\bibliography{cyclo}{}

\end{document}